\documentclass{amsart}

\usepackage[all,hyperref,numberbysection]{bi-discrete}
\usepackage{amsmath}
\usepackage{upgreek}
\usepackage{mathrsfs}
\usepackage{microtype}
\numberwithin{equation}{section}
\usepackage{a4wide}

\usepackage[backend=biber,maxbibnames=5,maxalphanames=5,style=alphabetic,sorting=nty,sortcites=true,bibencoding=utf8,giveninits,url=false,isbn=false]{biblatex}
\AtBeginBibliography{\small}

\begin{document}

\author{Sun-Sig Byun}
\author{Kyeongbae Kim}
\author{Kyeong Song}
\address{Sun-Sig Byun: Department of Mathematical Sciences and Research Institute of Mathematics, Seoul National University, Seoul 08826, Korea}
\email{byun@snu.ac.kr}

\address{Kyeongbae Kim: Department of Mathematical Sciences, Seoul National University, Seoul 08826, Korea}
\email{kkba6611@snu.ac.kr}

\address{Kyeong Song: School of Mathematics, Korea Institute for Advanced Study, Seoul 02455, Korea}
\email{kyeongsong@kias.re.kr}

\makeatletter
\@namedef{subjclassname@2020}{\textup{2020} Mathematics Subject Classification}
\makeatother

\subjclass[2020]{35R09, 35B65, 35J60}

\keywords{fractional $p$-Laplacian, Reifenberg flat, boundary regularity}
\thanks{S.-S. Byun was supported by Mid-Career Bridging Program through Seoul National University.  K. Kim was supported by NRF-2022R1A2C1009312. K. Song was supported by a KIAS individual grant (MG091702) at Korea Institute for Advanced Study.}

\title[Nonlinear nonlocal equations in Reifenberg flat domains]{Nonlinear nonlocal equations in Reifenberg flat domains}

\begin{abstract}
We consider nonhomogeneous fractional $p$-Laplace equations defined on a bounded nonsmooth domain which goes beyond the Lipschitz category. 
Under a sufficient flatness assumption on the domain in the sense of Reifenberg, 
we establish several fine boundary regularity results for solutions, and their gradient, near the boundary. To the best of our knowledge, each of our results is new even in the linear case. 
\end{abstract}

\maketitle

\tableofcontents


\section{Introduction}
This paper is concerned with nonlocal problems of the type 
\begin{equation}\label{eq.diri}
\left\{
\begin{alignedat}{3}
(-\Delta_p)^s{u}&= f&&\qquad \mbox{in  $\Omega\cap B_2$}, \\
{u}&=0&&\qquad  \mbox{in $B_2\setminus \Omega$},
\end{alignedat} \right.
\end{equation}
where $\Omega \subset \mathbb{R}^n$ is a bounded domain with possibly nonsmooth boundary $\partial\Omega$ and $(-\Delta_p)^s$ is the $s$-fractional $p$-Laplace operator defined by 
\begin{equation*}
    (-\Delta_p)^su(x)\coloneqq p.v.\,\int_{\bbR^n}\frac{|u(x)-u(y)|^{p-2}(u(x)-u(y))}{|x-y|^{n+sp}}\,dy, \quad x \in \mathbb{R}^n,
\end{equation*}
with $p\in(1,\infty)$ and $s\in(0,1)$. In this paper, we are interested in boundary regularity for solutions, and their gradient, under a considerably mild regularity assumption on the domain $\Omega$.

Boundary regularity theory for partial differential equations presents nontrivial and challenging issues, particularly when the underlying domain lacks a graph structure and classical techniques such as flattening or reflection are no longer applicable. A central question in this direction is whether boundary regularity results can still be achieved in the absence of smoothness, or even Lipschitz continuity, of the domain. 
In this respect, it was shown in \cite{ByuWan04, ByuRyu} that sharp global gradient estimates for local elliptic problems remain valid under the Reifenberg flatness assumption (see Definition \ref{defn.rei}). 
The notion of Reifenberg flatness, which was introduced by Reifenberg \cite{Rei60} in his study of Plateau-type problems, is characterized by a scaling-invariant flatness that accommodates substantial boundary roughness. This geometric framework enables one to recover key analytic tools necessary for establishing boundary regularity, see \cite{Tor97} for a well-written summary. 
We also refer to \cite{KenTor99, LewNys12} and references therein for more on Reifenberg flatness in the context of harmonic measures and free boundary regularity results for local elliptic equations. 


We now turn to regularity theory for nonlocal problems. 
Despite its nonlocal nature, the fractional $p$-Laplacian shares key variational and structural features with its local counterpart and has attracted significant attention in recent years. Interior regularity for such nonlocal problems have been systematically studied. For instance, without being exhaustive, basic regularity results for problems with measurable kernels are well established in \cite{Coz17, DicKuuPal,DicKuuPal2, Kas09, KorKuuPal16, KimLeeLee23,AusBorEgeSar19,KuuMinSir152, KuuMinSir15,NguOkSon24, Sch} and references therein.  For higher regularity results under better regularity assumptions on the data, see also \cite{BisTop24, BogDuzLiBisSer24d, BraLin17, BraLinSch18, ByuKim,  DieKimLeeNow24, DieKimLeeNow24j, DieNow23, GarLin24,Fal20,KuuSimYan22,BogDuzLiaMorcz25,BogDuzLiaBisSer24,MenSchYee21,CafSil11}.

However, on the contrary to interior ones, boundary regularity issues for nonlocal problems exhibit several peculiarities, even in the case of linear equations in sufficiently regular domains; see \cite{FerRos24} for an overview. 
For instance, while any solution $u$ to \eqref{eq.diri} with $p=2$ and $f=\mathrm{const.}$ is locally $C^{\infty}$-regular, one cannot expect more than boundary $C^s$-regularity for $u$ due to explicit examples. 
Such a phenomenon suggests that boundary regularity for nonlocal problems needs substantially different approaches compared to the case of corresponding local problems, in which the interior and boundary regularity often coincide via the reflection argument.

Specifically, in view of higher regularity near the boundary, it was observed in \cite{Gru15,RosSer14,RosSer17} that the function $u/d^s$, where $d$ denotes the distance function to $\partial\Omega$ (see \eqref{defn.dist}), enjoys better boundary regularity than $u$ itself.  
Heuristically, the limit of $u/d^s$ as $s\nearrow1$ can be considered as the normal derivative of $u$ at a boundary point. 
There are several recent results for linear nonlocal problems in this direction, concerned with various regularity assumptions on the domain and the kernel. 
We refer to \cite{Pra25, Pra25a,BorNoc23} for higher H\"older and Sobolev regularity results for the solution itself, and \cite{KimWei24,RosWei24,Gru24,AbeGru23,ByuKimKum25} for higher regularity results for $u/d^s$. 
Note in particular that in \cite{Pra25, Pra25a}, global $C^{s-\epsilon}$-regularity was proved for linear nonlocal problems in Reifenberg flat domains. 
We also refer to \cite{ChoKimRyu23, DonRyu24,AbdFerLeoYou23} for other kinds of global Calder\'on-Zygmund type estimates. 


Up to now, there are few boundary regularity results for fractional $p$-Laplacian type problems in the literature. 
In \cite{IanMosSqu16}, global $C^{\alpha}$-regularity for some $\alpha \in (0,s]$ of weak solutions $u$ was obtained via comparison principles and barrier arguments. 
Such results, combined with the interior higher regularity results in \cite{BraLinSch18, GarLin24}, in turn lead to global $C^{s}$-regularity of $u$. 
Later in \cite{IanMos24, IanMosSqu20}, global H\"older regularity for $u/d^s$ was obtained. We remark that all the papers \cite{IanMosSqu16,IanMosSqu20,IanMos24} considered $C^{1,1}$-domains. 
We also refer to \cite{BorLiwNoc24} for boundary higher differentiability results for solutions in Lipschitz domains. 
However, to our knowledge, there is no boundary regularity result for nonlocal problems with $p$-growth in nonsmooth domains.

The aim of this paper is to develop a boundary regularity theory for nonlinear nonlocal problems of the type \eqref{eq.diri} in nonsmooth domains. Specifically, we establish several fine regularity results near the boundary, provided the domain is sufficiently flat. 
This opens new paths for studying nonlocal problems in rough domains, and the methods develop in this paper may extend to more general nonlocal operators with nonstandard growth. 

\subsection{Assumptions and main results}

We first recall the notion of $(\delta,R_0)$-Reifenberg flatness. 
We refer to the next section for more on basic properties of Reifenberg flat domains. 

\begin{definition}\label{defn.rei}
   Let $\delta>0$ and $R_0>0$ be fixed. We say that $\Omega$ is $(\delta,R_0)$-Reifenberg flat if for any $z\in\partial\Omega$ and $r\leq R_0$, there is a coordinate system $\{y_1,\ldots,y_n\}$, which may depend on $r$ and $z$, such that $z=(z',0)\in\bbR^{n-1}\times \bbR$ in this coordinate system and 
   \begin{equation*}
      B_r(z',0)\cap \{y_n>\delta r\}\subset  B_r(z)\cap \Omega\subset B_r(z',0)\cap \{y_n>-\delta r\}.
   \end{equation*}
\end{definition}

\begin{remark} \,
\begin{itemize}
\item[(i)] Note that any Lipschitz domain with a small Lipschitz constant is $(\delta,R_0)$-Reifenberg flat for some sufficiently small constants $\delta$ and $R_0$ (see \cite{Tor97}).
\item[(ii)] A notable consequence of Reifenberg flatness is the following measure density condition (see for instance \cite[(2.12) in Lemma 2.7]{ByuWan04}): if $\Omega$ is $(\delta,R_0)$-Reifenberg flat with $\delta<1/2$ and $R_0>0$, then 
\begin{align}\label{cond.meas}
    \sup_{0<\rho<R_0}\sup_{z\in\Omega}\frac{|B_\rho|}{|B_\rho(z)\cap \Omega|}\leq 4^n\quad\text{and}\quad \inf_{0<\rho<R_0}\inf_{z\in\partial\Omega}\frac{|B_\rho(z)\setminus \Omega|}{|B_\rho|}\geq 4^{-n}.
\end{align}
Therefore, we always consider a sufficiently flat domain $\Omega$ to ensure that it satisfies the measure density condition \eqref{cond.meas}. 
\end{itemize}
\end{remark}

Before we introduce a notion of weak solution to \eqref{eq.diri}, we recall function spaces; see for instance \cite{BraLin17,DicKuuPal} for more details. 
We consider the following fractional space 
    \begin{align*}
    X_0^{s,p}(\Omega)\coloneqq\{u\in W^{s,p}(\bbR^n)\,:\,u\equiv 0 \text{ on }\bbR^n\setminus \Omega\}
\end{align*}
equipped with the norm
\begin{align*}
    \|u\|_{X_0^{s,p}(\Omega)}\coloneqq \|u\|_{L^p(\Omega)}+[u]_{W^{s,p}(\mathbb{R}^n)}.
\end{align*}
We also consider the tail space
    \begin{align*}
        L^{p-1}_{sp}(\bbR^n)\coloneqq\left\{u:\bbR^n\to\bbR\,:\,\int_{\bbR^n}\frac{|u(y)|^{p-1}}{(1+|y|)^{n+sp}}\,dy<\infty\right\}.
    \end{align*}
Accordingly, for any $u \in L^{p-1}_{sp}(\bbR^n)$ and any ball $B_r(z) \subset \bbR^n$, we denote the nonlocal tail of $u$ by
\begin{align*}
     \mathrm{Tail}(u;B_r(z))\coloneqq \left(r^{sp}\int_{\bbR^n\setminus B_r(z)}\frac{|u(y)|^{p-1}}{|y-z|^{n+sp}}\,dy\right)^{1/(p-1)}.
\end{align*}

We are now ready to state a definition of weak solution to \eqref{eq.diri}.
\begin{definition}
    We say that $u\in W^{s,p}(B_R(x_0))\cap L^{p-1}_{sp}(\bbR^n)$ is a weak solution to 
    \begin{equation*}
\left\{
\begin{alignedat}{3}
(-\Delta_p)^s{u}&= f&&\qquad \mbox{in  $\Omega\cap B_R(x_0)$}, \\
{u}&=0&&\qquad  \mbox{in $B_R(x_0)\setminus \Omega$},
\end{alignedat} \right.
\end{equation*}
where $f\in L^{p_*}(\Omega\cap B_R(x_0))$ with
\begin{equation}\label{defn.pstar}
    p_*\coloneqq 
    \left\{
    \begin{aligned}
        &\frac{np}{np-n+sp} & \text{if}\;\; sp < n, \\
        &\text{any number in }(1,\infty) & \text{if}\;\; sp \ge n,
    \end{aligned}
    \right.
\end{equation}
if it satisfies 
\begin{align*}
    \iint_{\bbR^{2n}}\frac{|u(x)-u(y)|^{p-2}(u(x)-u(y))(\phi(x)-\phi(y))}{|x-y|^{n+sp}}{\,dx\,dy}=\int_{\Omega\cap B_R(x_0)}f \phi\,dx
\end{align*}
for any $\phi\in X^{s,p}_0(\Omega\cap B_R(x_0))$.
\end{definition}

We now introduce our main results. 
The first result is concerned with integrability estimates for the function $u/d^s$, where the function $d$ is defined by
\begin{align}\label{defn.dist}
    d(x)\equiv d_\Omega(x)\coloneqq \mathrm{dist}(x,\bbR^n\setminus \Omega).
\end{align}
We also consider the following functional 
\begin{align*}
    \widetilde{E}(u;B_r(z))\coloneqq\left(\dashint_{B_r(z)\cap \Omega}|u|^{p-1}\,dx\right)^{1/(p-1)}+ \mathrm{Tail}(u;B_r(z)).
\end{align*}

\begin{theorem}\label{thm.lp}
     Let $u\in W^{s,p}(B_2)\cap L^{p-1}_{sp}(\bbR^n)$ be a weak solution to \eqref{eq.diri} 
with $f\in L^{p_*}(\Omega\cap B_2)$. For any $q\in[p_*,n/s)$, there is a sufficiently small $\delta=\delta(n,s,p,q)>0$ such that if $\Omega$ is $(\delta,R_0)$-Reifenberg flat with $R_0>0$, then we have 
\begin{align*}
    \|u/d^s\|_{L^{\frac{nq(p-1)}{n-sq}}(\Omega\cap B_{1})}\leq c\left(\widetilde{E}(u;B_2)+\|f\|^{1/(p-1)}_{L^{q}(\Omega\cap B_2)}\right)
\end{align*}
for some constant $c=c(n,s,p,q,R_0)$.
\end{theorem}
\begin{remark} \,
\begin{itemize}
\item[(i)] When $p=2$ and $\Omega$ is $C^{1,\alpha}$ for any $\alpha>0$, the same estimate is proved in \cite[Theorem 1.4]{ByuKimKum25} (see also \cite[Remark 1.6]{ByuKimKum25} for similar results in this direction). Moreover, the sharpness of Theorem~\ref{thm.lp} is shown in \cite[Appendix A]{ByuKimKum25}. 
\item[(ii)] Theorem \ref{thm.lp} extends the previous result in \cite{ByuKimKum25} to degenerate/singular nonlocal problems in nonsmooth domains. To our knowledge, such a result is new even for Lipschitz domains.
\end{itemize}
\end{remark}

We next present a boundary higher Sobolev regularity result. 
\begin{theorem}\label{thm.hig}
        Let $u\in W^{s,p}(B_2)\cap L^{p-1}_{sp}(\bbR^n)$ be a weak solution to \eqref{eq.diri}, 
where $f\in L^q(\Omega\cap B_2)$ with $q>n/s$.  Let us fix $\gamma\in[1,\infty)$ with 
\begin{equation*}
    s+\frac1\gamma<\min\left\{1,\left(sp-\frac{n}{q}\right)\frac1{p-1}\right\}.
\end{equation*} For any $\sigma<s+1/\gamma$, there is a sufficiently small $\delta=\delta(n,s,p,q,\gamma,\sigma)>0$ such that if $\Omega$ is $(\delta,R_0)$-Reifenberg flat with $R_0>0$, then we have
\begin{equation*}
    \|u\|_{W^{\sigma,\gamma}(B_{1})}\leq c\left(\widetilde{E}(u;B_2)+\|f\|^{1/(p-1)}_{L^{q}(\Omega\cap B_2)}\right)
\end{equation*}
for some constant $c=c(n,s,p,R_0,q,\gamma,\sigma)$.
\end{theorem}
\begin{remark}\,
\begin{itemize}
\item[(i)]
    The result in Theorem \ref{thm.hig} is sharp in the sense that we cannot take $\sigma = s+1/\gamma$. 
    Indeed, we can see that the function $v(x)=(x_n)_+^{s}$ is a weak solution to \eqref{eq.diri} with $\Omega=\bbR^n_+$ and $f=0$, which satisfies $v\in W^{s+1/\gamma-\epsilon,\gamma}(B_{1/2})$ 
    for any $\epsilon>0$, but $v \notin W^{s+1/\gamma,\gamma}(B_{1/2})$. 

\item[(ii)]
    We would like to point out that in \cite{BorLiwNoc24,BorNoc23}, Besov-type regularity estimates for \eqref{eq.diri}, such as $u\in B^{s+1/\max\{p,2\}}_{p,\infty}(\Omega)$, are obtained when $\Omega$ is Lipschitz and $f$ belongs to a suitable Besov space. Theorem \ref{thm.hig} improves such results by obtaining the differentiability order arbitrarily close to $s+1/\gamma$, which is bigger than $s+1/p$ provided $s<(p-1)/p$, when $\Omega$ is Reifenberg flat and $f$ is sufficiently regular. 

\item[(iii)]
    We also note that Theorem \ref{thm.hig} does not follow from the results in \cite{BorLiwNoc24,BorNoc23}, as a $(\delta,R_0)$-Reifenberg flat domain need not be Lipschitz, even if $\delta$ is sufficiently small (see \cite{Tor97}).
\end{itemize}
\end{remark}
In particular, we also have a sharp boundary H\"older regularity result. 
\begin{theorem}\label{cor}
     Let $u\in W^{s,p}(B_2)\cap L^{p-1}_{sp}(\bbR^n)$ be a weak solution to \eqref{eq.diri}, 
where $f\in L^{n/s}(\Omega\cap B_2)$. For any $\alpha\in(0,s)$, there is a sufficiently small $\delta=\delta(n,s,p,\alpha)>0$ such that if $\Omega$ is $(\delta,R_0)$-Reifenberg flat with $R_0>0$, then we have $u\in C^\alpha(B_{1})$ with the estimate
\begin{align*}
    \|u\|_{C^\alpha(B_{1})}\leq c\left(\widetilde{E}(u;B_2)+\|f\|^{1/(p-1)}_{L^{n/s}(\Omega \cap B_2)}\right),
\end{align*}where $c=c(n,s,p,\alpha)$.
\end{theorem}
\begin{remark} \,
\begin{itemize}
    \item[(i)] When $p=2$ and $f\in L^\infty$, the same kind of results are proved in \cite{Pra25,Pra25a}. We extend those results by not only weakening the regularity assumption on the right-hand side, but also considering the fractional $p$-Laplace operator which can be degenerate or singular. 
    \item[(ii)] On the other hand, in \cite[Theorem 3.11]{LiaZhaDonHon20}, boundary $C^{\beta}$-regularity for some $\beta\in(0,2s-n/q]$ of solutions is proved when $p=2$ and $f\in L^{q}$ with $q\in(n/(2s),n/s)$.
    \end{itemize}
\end{remark}
In the linear case, we further present a boundary Calder\'on-Zygmund type estimate. 
\begin{theorem}\label{thm.cz}
   Let us fix $q\in[2_*,n/(2s))$. Let $u\in W^{s,2}(B_1)\cap L^{1}_{2s}(\bbR^n)$ be a weak solution to 
    \begin{equation*}
\left\{
\begin{alignedat}{3}
(-\Delta)^s{u}&= f&&\qquad \mbox{in  $\Omega\cap B_2$}, \\
{u}&=0&&\qquad  \mbox{in $B_2\setminus \Omega$},
\end{alignedat} \right.
\end{equation*}
where $f\in L^q(\Omega\cap B_2)$. Then for any $\gamma<nq/(n-2sq)$ and $\sigma\in(s,\min\{s+1/\gamma,1\})$ with $\sigma-n/\gamma<2s-n/q$, there is a sufficiently small $\delta=\delta(n,s,q,\gamma,\sigma)>0$ such that if  $\Omega$ is $(\delta,R_0)$-Reifenberg flat with  $R_0>0$, then we have
\begin{align*}
    \|u\|_{W^{\sigma,\gamma}(B_{1})}\leq c\left(\widetilde{E}(u;B_2)+\|f\|_{L^q(\Omega \cap B_2)}\right)
\end{align*}
for some constant $c=c(n,s,q,\gamma,\sigma)$.
\end{theorem}
\begin{remark} \,
\begin{itemize}
\item[(i)] 
    To our knowledge, similar results available in the literature are only known for $C^{1,\alpha}$-domains, whose proof relies on sharp two-sided Green function estimates (see \cite[Theorem 1.4]{AbdFerLeoYou23} and \cite[Remark 1.4]{KimWei24}). On the other hand, our approach can cover nonsmooth domains, for which sharp Green function estimates are not known.
\item[(ii)]
    Indeed, the proof of Theorem \ref{thm.cz} employs comparison estimates in $L^\gamma$ space for sufficiently large $\gamma>2$ (see Remark \ref{rmk.comp.high}), whose validity for nonlinear equations is not clear. This is why we confine ourselves to linear equations in Theorem \ref{thm.cz}, and our approach is not directly applied to fractional $p$-Laplace equations with $p\neq 2$. 
\end{itemize}
\end{remark}

The last result is concerned with sharp gradient integrability near the boundary.
We believe that the following theorem will be very useful for the investigation of Calder\'on-Zygmund type gradient estimates on nonsmooth domains.
\begin{theorem}\label{thm.lip}
     Let $u\in W^{s,p}(B_2)\cap L^{p-1}_{sp}(\bbR^n)$ be a weak solution to \eqref{eq.diri}, 
where $s>(p-1)/p$ and $f\in L^\infty(\Omega\cap B_2)$. For any $q<1/(1-s)$, there is a sufficiently small $\delta=\delta(n,s,p,q)>0$ such that if $\Omega$ is $(\delta,R_0)$-Reifenberg flat with $R_0>0$, then we have 
\begin{align*}
    \|\nabla u\|_{L^q( B_{1})}\leq c\left(\widetilde{E}(u;B_2)+\|f\|^{1/(p-1)}_{L^\infty(\Omega\cap B_2)}\right)
\end{align*}
for some constant $c=c(n,s,p,R_0,q)$.
\end{theorem}

\subsection{Techniques}
Here we briefly sketch the idea of proof of our main results. Regarding the proof of Theorem \ref{thm.lp}, we only explain the case when $f=0$ to skip some technical parts. 
To this end, we first find a regular solution $v$ which is close to $u$ in the sense that the $L^{p-1}$-norm of $(u-v)/d^s$ is sufficiently small, when the domain is sufficiently flat. As in the local case \cite{ByuRyu, ByuWan04}, this procedure is based on a compactness argument. More precisely, we assume that there is a sequence of weak solutions $\{u_i
\}\subset W^{s,p}(B_1)\cap L^{p-1}_{sp}(\bbR^n)$ to 
 \begin{equation*}
\left\{
\begin{alignedat}{3}
(-\Delta_p)^s{u}_i&= 0&&\qquad \mbox{in  $\Omega_i\cap B_1$}, \\
{u}_i &=0&&\qquad  \mbox{in $B_1\setminus \Omega_i$},
\end{alignedat} \right.
\end{equation*}
where $\Omega_i$ is $(1/i,1)$-Reifenberg flat, such that $u_i$ is not locally close to any regular function $v$. A difficulty here is that we do not know any global Sobolev regularity of $u_i$, which allows us to use a compactness argument in the whole domain. Thus, by using a localization argument, we consider a weak solution $w_i\in W^{s,p}(\bbR^n)$ to 
\begin{equation*}
\left\{
\begin{alignedat}{3}
(-\Delta_p)^s{w}_i&= g_i&&\qquad \mbox{in  $\Omega_i\cap B_{1/2}$}, \\
{w_i}&=0&&\qquad  \mbox{in $B_{1/2}\setminus \Omega_i$},
\end{alignedat} \right.
\end{equation*}
where $g_i\in L^\infty(\Omega_i\cap B_{1/2})$, $w_i=u_i$ in $B_{1/2}$ and $w_i\equiv 0$ in $\bbR^n\setminus B_1$. Now, due to the fact that $w_i\in W^{s,p}(\bbR^n)$, we find a limiting function $w_\infty \in W^{s,p}(\bbR^n)$ weakly solving
\begin{equation*}
\left\{
\begin{alignedat}{3}
(-\Delta_p)^s{w}_\infty&= g_\infty&&\qquad \mbox{in  $ B^+_{1/2}$}, \\
{w}_\infty&=0&&\qquad  \mbox{in $B_{1/2}^-$}
\end{alignedat} \right.
\end{equation*}
with $g_\infty\in L^\infty(B^+_{1/2})$, which enjoys $C^s$-regularity up to the boundary. 
In addition, we obtain that the $L^{p-1}$ norm of $(u_i-w_\infty)/d^s$ is small when $i$ is sufficiently large. This gives a contradiction and thus we are always able to find a regular solution which is close to $u$ when the domain is $(\delta,R_0)$-Reifenberg flat for sufficiently small $\delta>0$. 
We highlight that we do not use any Liouville type theorem which is frequently used to obtain comparison estimates beyond the solution level for linear nonlocal equations. With this comparison estimate and perturbation arguments based on \cite{CafPer98}, we obtain the desired higher integrability result for the function $u/d^s$.

Next, using the estimate given in Theorem \ref{thm.lp} and a interior higher H\"older regularity for the fractional $p$-Laplace equations, we prove boundary higher Sobolev and H\"older regularity in nonsmooth domains. In particular, we investigate the integral
\begin{equation*}
    I\coloneqq\sup_{0<|h|\ll1}\int_{B_{1/8}}\frac{|u(x+h)-u(x)|^\gamma}{|h|^{\sigma \gamma}}\,dx,
\end{equation*}
where $\gamma\in[1,\infty)$ and $\sigma<s+1/\gamma$, depending on whether the points $x$ and $x+h$ are close to the boundary or not. When both points are near the boundary, we want to obtain an estimate of the form
\begin{align*}
    I\leq c\int_{B_{1/4}}\frac{|u(x)|^\gamma}{d^{\sigma \gamma}(x)}\,dx,
\end{align*}
whose right-hand side is readily seen to be finite when $\sigma \leq s$; however, this is not clear when $\sigma>s$. 
To ensure the finiteness of the integral on the right-hand side, we prove that the distance function $d$ satisfies $d^{-\epsilon}\in L^1$ for any $\epsilon\in(0,1)$. Note that when $\Omega$ is either a Lipschitz domain or a half-space, this is a direct consequence, as $\mathcal{H}^{n-1}(\partial\Omega)<\infty$ and the distance function behaves like $d_\Omega(x)\eqsim x_n$. However, these two properties are not true when $\Omega$ is $(\delta,R_0)$-Reifenberg flat even though $\delta$ is sufficiently small (see \cite[Remark 2.2]{LewNysVog13}). Thus, we instead find a suitable covering of level sets of the distance function to derive a good decay estimate of the size of level sets in terms of the flatness of the domain (see Lemma \ref{lem.reidist}), which leads to $d^{-\epsilon}\in L^1$ for any $\epsilon\in(0,1)$ (see Lemma \ref{lem.basic1}). Using this, we obtain $u/d^\sigma\in L^\gamma$ for any $\sigma<s+1/\gamma$, which enables us to control the term $I$ in the boundary case. On the other hand, when at least one of the points is far from the boundary, we use the interior higher H\"older estimates for solutions, established in \cite{BraLinSch18,GarLin24} (see the proof of Theorem \ref{thm.hig} below for more details). Consequently, in any case, we derive $I<\infty$ and then apply an embedding lemma to get the desired higher Sobolev and H\"older estimates near the boundary. 

Subsequently, using Theorem \ref{thm.hig} and a refined difference quotient technique together with a comparison estimate in $L^\gamma$ space for sufficiently large $\gamma$ (see Lemma \ref{lem.comp.high} and  and Remark \ref{rmk.comp.high}), we can prove Calder\'on-Zygmund type estimates below the gradient level for linear nonlocal equations on nonsmooth domains. Note that this kind of difference quotient technique was first introduced in \cite{KriMin05, Miin07} in the context of local problems.  

Lastly, we prove sharp boundary gradient estimates by analyzing the term $I$ with $\sigma=1$. More precisely, we employ the same argument as in the proof of Theorem \ref{thm.hig} along with the interior Lipschitz estimates established in the recent paper \cite{BisTop24}.

\section{Preliminaries}
\subsection{Notation}
Throughout this paper, we denote by $c$ an universal constant which is bigger than or equal to 1. Moreover, we use parentheses to indicate the relevant dependence of a constant; for instance, $c=c(n,s,p)$ means that $c$ depends only on $n,s,$ and $p$.

We denote by $B_{r}(z) \coloneqq \{x \in \bbR^n : |x-z| < r\}$ the $n$-dimensional ball with center $z\in \bbR^n$ and radius $r>0$. In particular, we write $B_{r} \equiv B_{r}(0)$. 
Moreover, we write $x\coloneqq (x',x_n)\in\bbR^{n-1}\times \bbR$ and
\begin{equation*}
    B^\pm_r(z)\coloneqq \{x\in B_r(z)\,:\,\pm (x_n-z_n)>0\}.
\end{equation*}

We also recall the definition of the distance function $d_\Omega(x)\coloneqq \mathrm{dist}(x,\bbR^n\setminus \Omega)$ given in \eqref{defn.dist}. When the domain is clear from context, we omit the subscript denoting the domain $\Omega$, and simply write $d(x)\equiv d_\Omega(x)$.


\subsection{Basic properties of Reifenberg flat domains}
We observe the following scaling property of a localized equation. 
\begin{lemma}\label{lem.scale}
    Let $u$ be a weak solution to 
    \begin{equation*}
\left\{
\begin{alignedat}{3}
(-\Delta_p)^s{u}&= f&&\qquad \mbox{in  $\Omega \cap B_{r}(x_0)$}, \\
{u}&=0&&\qquad  \mbox{in $B_{r}(x_0)\setminus \Omega$},
\end{alignedat} \right.
\end{equation*}
and define 
\begin{equation*}
    u_{r,x_0}(x)\coloneq \frac{u(rx+x_0)}{ r^s},\quad f_{r,x_0}(x) \coloneqq {r^sf(rx+x_0)}{} .
\end{equation*}
Then for any $\lambda>0$, $u_{r,x_0}/\lambda$ is a weak solution to
\begin{equation*}
\left\{
\begin{alignedat}{3}
(-\Delta_p)^s({u}_{r,x_0}/\lambda)&= f_{r,x_0}/\lambda^{p-1}&&\qquad \mbox{in  $\Omega_{r,x_0} \cap B_{1}$}, \\
{u}_{r,x_0}/\lambda&=0&&\qquad  \mbox{in $B_{1}\setminus \Omega_{r,x_0}$},
\end{alignedat} \right.
\end{equation*}
where 
\begin{equation*}
    \Omega_{r,x_0} \coloneqq (\Omega-x_0)/r\coloneqq\{(x-x_0)/r\,:\,x\in \Omega\}.
\end{equation*}
In addition, if $\Omega$ is $(\delta,R_0)$-Reifenberg flat, then $\Omega_{r,x_0}$ is $(\delta,R_0/r)$-Reifenberg flat.
\end{lemma}
The next lemma is concerned with the scaling property of the distance function. 
\begin{lemma}\label{lem.rei}
    Let $\Psi:\bbR^n\to\bbR^n$ be an orthogonal transformation. For $x_0\in \bbR^n$ and $r>0$, define $\Psi_{r,x_0}:\bbR^n \to\bbR^n$ by
    \begin{equation*}
        \Psi_{r,x_0}(x)\coloneqq\Psi((x-x_0)/r), \quad x \in \bbR^n.
    \end{equation*}
 Then we have
    \begin{equation*}
        \mathrm{dist}(x,\bbR^n\setminus \Omega)= r\mathrm{dist}(\Psi_{r,x_0}(x),\bbR^n\setminus \Psi_{r,x_0}(\Omega))\quad\text{for any }x\in\bbR^n.
    \end{equation*}
    In particular, we have
    \begin{align*}
        d_\Omega(x)=rd_{\Omega_{r,x_0}}((x-x_0)/r).
    \end{align*}
\end{lemma}
\begin{proof}
Note that $x\in \bbR^n\setminus \Omega$ if and only if $\Psi_{r,x_0}(x)\in \bbR^n\setminus \Psi_{r,x_0}(\Omega)$. Thus, it suffices to prove that if $x\in \Omega$, then 
\begin{equation}\label{ineq1.dist}
        \mathrm{dist}(x,\bbR^n\setminus \Omega)= r\mathrm{dist}(\Psi_{r,x_0}(x),\bbR^n\setminus \Psi_{r,x_0}(\Omega)).
    \end{equation}
    Let $\overline{x}\in\bbR^n\setminus \Omega$ satisfy $|x-\overline{x}|=\mathrm{dist}(x,\bbR^n\setminus \Omega)$. Then we have $\Psi_{r,x_0}(\overline{x})\in \bbR^n\setminus \Psi_{r,x_0}(\Omega)$ with 
    \begin{equation*}
        |\Psi_{r,x_0}(x)-\Psi_{r,x_0}(\overline{x})|=|x-\overline{x}|/r=|x-\overline{x}|/r,
    \end{equation*}
    where we have used the fact that $\Psi$ is orthogonal. Moreover, for any $\overline{x_1}=\Psi_{r,x_0}(x_1)\in \bbR^n\setminus \Psi_{r,x_0}(\Omega)$, we have 
    \begin{equation*}
        |\Psi_{r,x_0}(x)-\overline{x_1}|=|\Psi_{r,x_0}(x)-\Psi_{r,x_0}(x_1)|=|x-\overline{x}|/r\geq |x-x_0|/r=|\Psi_{r,x_0}(x)-\Psi_{r,x_0}(\overline{x})|,
    \end{equation*}
    which implies \eqref{ineq1.dist}. This completes the proof.
\end{proof}

We now provide a decay type estimate for the size of an $r$-neighborhood of $\partial\Omega$ in terms of the flatness of $\Omega$, which will be used in Section \ref{sec.sob}. Our argument is based on \cite[Lemma B.3.5]{FerRos24}.
\begin{lemma}\label{lem.reidist}
    For $r\in(0,2^{-10}]$,  consider
    \begin{align*}
        D_r\coloneqq\{x\in\Omega\,:d_\Omega(x)\leq r\}.
    \end{align*}
   If $\Omega$ is $(\delta,1)$-Reifenberg flat for $\delta \in (0, 2^{-n}]$, then we have
    \begin{align*}
        |D_{\delta r}\cap B_1|\leq c\left(\delta|D_{r}\cap B_1|+r\right)
    \end{align*}
    for some constant $c=c(n)\geq1$.
\end{lemma}
\begin{remark}
    Note that if the domain is flat, then we directly have
    \begin{align*}
        |D_{\delta r}\cap B_1|\leq \delta|D_{r}\cap B_1|.
    \end{align*}
\end{remark}
\begin{proof}
Let us fix $r\in(0,2^{-10}]$ and $\delta \in (0,2^{-n}]$.
   Recall from \eqref{defn.rei} that for each $z_0\in \partial\Omega$, there is a new coordinate system $\{y_1,\ldots ,y_n\}$ such that $z_0=(z_0',0)$ and
    \begin{align}\label{ineq0.reidist}
         B_{5r}(z_0)\cap \{y_n>5 r\delta\}\subset B_{5r}(z_0)\cap \Omega\subset B_{5r}(z_0)\cap \{y_n>-5r\delta\}.
    \end{align}    
    Now we are going to construct a suitable covering of $D_{\delta r}\cap B_1$. To this end, we first use the Vitali covering lemma to have a countable collection of mutually disjoint balls $\{B_{r}(z_i)\}_{i\in I}$, where $z_i\in \partial\Omega\cap B_{1+2r}$, satisfying
    \begin{align}\label{ineq00.reidist}
        \bigcup_{z\in \partial\Omega\cap B_{1+2r}}B_r(z)\subset \bigcup_{i\in I}B_{5r}(z_i).
    \end{align}
    In addition, we get
    \begin{align}\label{ineq000.reidist}
        \bigcup_{i\in I}(B_{r/4}(z_i)\cap \Omega)\subset D_{r}\quad\text{and}\quad (D_{r}\cap B_1)\subset\bigcup_{z\in \partial\Omega\cap B_{1+2r}}B_r(z)\subset \bigcup_{i\in I}B_{5r}(z_i).
    \end{align}
    Note that either
    \begin{align*}
         B_{5r}(z_i)\subset B_{1+12r}\setminus B_{1-8r}\quad\text{or}\quad B_{5r}(z_i)\subset B_1
    \end{align*}
    must hold; we accordingly set
    \begin{align*}
        I_1\coloneqq \{i\in I\,:\,B_{5r}(z_i)\subset B_{1+12 r}\setminus B_{1-8r} \}
    \end{align*}
    and
    \begin{align}\label{defn.I1}
        I_2\coloneqq\{i\in I\,:\,B_{5r}(z_i)\subset B_1\}.
    \end{align}
    If $i\in I_1$, then we have 
    \begin{align}\label{ineq1.reidist}
        \sum_{i\in I_1}|B_{5r}(z_i)|\leq 5^n\sum_{i\in I_1}|B_{r}(z_i)|\leq 5^n|B_{1+12r}\setminus B_{1-8r}|\leq cr,
    \end{align}
    where we have used the fact that $\{B_r(z_i)\}_{i\in I}$ is mutually disjoint. 
   For each $i\in {I}_2$, we define 
    \begin{equation*}
        U_r(z_i)\coloneqq B_{5r}(z_i)\cap \{-10r\delta<y_n<10r\delta \}
    \end{equation*}
    on the coordinate system $\{y_1,\ldots, y_n\}$ (see \eqref{ineq0.reidist}). Then we have
    \begin{align}\label{ineq11.reidist}
        |U_r(z_i)|\leq c\delta|B_{5r}(z_i)| \le c\delta|B_{r/4}(z_i)|
    \end{align}
    for some constant $c=c(n)$. Now we want to prove 
    \begin{equation}\label{ineq2.reidist}
        D_{\delta r}\cap B_1\subset  \left(\bigcup_{i\in I_1} B_{5r}(z_i)\right)\cup \left(\bigcup_{i\in I_2}U_{r}(z_i)\right).
    \end{equation}
    Suppose that $x\in \Omega\cap B_1 $ with $d_\Omega(x)\leq \delta r$, then there is a point $\overline{x}\in \partial\Omega$ such that $d_\Omega(x)=|x-\overline{x}|$.
    By \eqref{ineq00.reidist}, $B_{r}(\overline{x})\subset B_{5r}(z_i)$ for some $i\in I$. If $i\in I_1$, then \eqref{ineq2.reidist} directly follows from the fact that $x\in B_{\delta r}(\overline{x})\subset B_r(\overline{x})\subset B_{5r}(z_i)$. We next assume $i\in I_2$. Note that, on the coordinate system $\{y_1,\ldots, y_n\}$, \eqref{ineq0.reidist} and the fact that $d_\Omega(x)\leq \delta r$ imply
    \begin{equation*}
        \overline{x}\in B_{5r}(z_i)\cap \partial\Omega\subset B_{5r}(z_i)\cap \{-6r\delta\leq y_n\leq5r\delta \}. 
    \end{equation*}
    Therefore, we have 
    \begin{align*}
        x\in B_{5r}(z_i)\cap \{-7r\delta<y_n<7r\delta \}\subset U_r(z_i),
    \end{align*}
    which implies \eqref{ineq2.reidist}.
    Now we obtain the desired estimate as follows:
    \begin{align*}
        |D_{\delta r}\cap B_1| & \leq  \sum_{i\in I_1}| B_{5r}(z_i)|+\sum_{i\in I_2}|U_{r}(z_i)| \\ 
        &\leq c\left[r+\delta\sum_{i\in I_2}|B_{r/4}(z_i)|\right]
        \leq c\left[r+\delta\sum_{i\in I_2}|B_{r/4}(z_i)\cap\Omega|\right]
        \leq c(r+\delta|D_r\cap B_1|)
    \end{align*}
    for some constant $c=c(n)$, where we have used \eqref{ineq1.reidist}, \eqref{ineq2.reidist}, \eqref{ineq11.reidist}, \eqref{defn.I1}, \eqref{ineq000.reidist}, \eqref{cond.meas} and the fact that $\{B_{r}(z_i)\}_{i\in I}$ is mutually disjoint. 
\end{proof}

\subsection{Supremum estimates near the boundary}
With the measure density condition \eqref{cond.meas}, we now provide supremum estimates near the boundary for fractional $p$-Laplace equations. To shorten notations, we consider the $s$-fractional gradient of $u$ defined by 
\begin{equation*}
    d_su(x,y)\coloneqq \frac{u(x)-u(y)}{|x-y|^s}\quad\text{for any }x\neq y.
\end{equation*}
First, we give an energy estimate near the boundary.
\begin{lemma}\label{lem.ene}
    Let $u\in W^{s,p}(B_r(x_0))\cap L^{p-1}_{sp}(\bbR^n)$ be a weak solution to 
     \begin{equation}\label{eq.ene}
\left\{
\begin{alignedat}{3}
(-\Delta_p)^s{u}&= f&&\qquad \mbox{in  $B_r(x_0)\cap \Omega$}, \\
{u}&=0&&\qquad  \mbox{in $B_r(x_0)\setminus \Omega$},
\end{alignedat} \right.
\end{equation}
where $f\in L^{p_*}(B_r(x_0)\cap\Omega)$ and $\Omega$ satisfies \eqref{cond.meas} with $R_0= r$. Then there is a constant $c=c(n,s,p)$ such that 
\begin{align*}
    \int_{B_{3r/4}(x_0)}\dashint_{B_{3r/4}(x_0)}|d_su|^p\frac{\,dx\,dy}{|x-y|^{n}}&\leq c\widetilde{E}(u/r^s;B_r(x_0))^p+c\left(\dashint_{B_r(x_0)\cap \Omega}|r^sf|^{p_*}\,dx\right)^{\frac{p}{p_*(p-1)}} .
\end{align*}
\end{lemma}
\begin{proof}
By Lemma \ref{lem.scale}, it suffices to consider the case that $B_r(x_0)=B_1$. Let us fix $7/8\leq \rho<R\leq 1$ and take a cut-off function $\psi\in C_c^\infty(B_{(\rho+R)/2})$ with $0 \le \psi \le 1$, $\psi\equiv 1$ in $B_\rho$ and $|\nabla\psi|\leq 16/(R-\rho)$. Testing \eqref{eq.ene} with $u\psi^{p}$, using \eqref{cond.meas} and then following the proof of \cite[Theorem 1.4]{DicKuuPal}, we get 
    \begin{equation}\label{ineq1.ene}
\int_{B_{R}}\dashint_{B_{R}}|d_s(u\psi)|^p\frac{\,dx\,dy}{|x-y|^{n}}\leq \frac{c}{(R-\rho)^{np}}\left[\dashint_{B_R}|u|^p\,dx+\mathrm{Tail}(u;B_R)^p \right]+c\dashint_{B_R\cap\Omega}|fu|\psi^{p}\,dx.
    \end{equation}
    For the last term in the right-hand side, we use H\"older's inequality 
    and the fractional Sobolev-Poincar\'e inequality given in \cite[Lemma 2.4]{Now23} to obtain
    \begin{equation*}
    \begin{aligned}
         \dashint_{B_R\cap\Omega}|fu|\psi^{p}\,dx&\leq c\left(\dashint_{B_R\cap \Omega}|f|^{p_*}\,dx\right)^{1/p_*}\left(\int_{B_R}\dashint_{B_{R}}|d_s(u\psi)|^{p}\frac{\,dx\,dy}{|x-y|^{n}}\right)^{1/p}\\
         &\quad+c\left(\dashint_{B_R\cap \Omega}|f|^{p_*}\,dx\right)^{1/p_*}\left(\dashint_{B_R}|u\psi|^{p}\,dx\right)^{1/p},
    \end{aligned}
    \end{equation*}
    where $c=c(n,s,p)$. We next apply
    \cite[Corollary 4.9]{Coz17} to the last term in the above display, together with the fact that $u\psi\equiv 0$ in $B_{R}\setminus \Omega$, where $|B_{R}\setminus\Omega|\geq 4^{-n}|B_R|$, to see that
    \begin{equation}\label{ineq2.ene}
        \dashint_{B_R\cap\Omega}|fu|\psi^{p}\,dx \leq c\left(\dashint_{B_R\cap \Omega}|f|^{p_*}\,dx\right)^{1/p_*}\left(\int_{B_{R}}\dashint_{B_{R}}|d_s(u\psi)|^p\frac{\,dx\,dy}{|x-y|^{n}}\right)^{1/p}
    \end{equation}
    for some constant $c=c(n,s,p)$. 
    Plugging this last display into \eqref{ineq1.ene}, applying Young's inequality and then reabsorbing terms, we deduce 
    \begin{align}\label{ineq3.ene}
        \int_{B_{R}}\dashint_{B_{R}}|d_s(u\psi)|^p\frac{\,dx\,dy}{|x-y|^{n}}&\leq \frac{c}{(R-\rho)^{np}}\left[\dashint_{B_R}|u|^p\,dx+\mathrm{Tail}(u;B_R)^p +\left(\dashint_{B_R\cap\Omega}|f|^{p_*}\,dx\right)^{\frac{p}{p_*(p-1)}}\right]
    \end{align}
    for some constant $c=c(n,s,p)$. Similarly to the estimate of \eqref{ineq2.ene}, by applying \cite[Corollary 4.9]{Coz17} to the left-hand side of \eqref{ineq3.ene}, we get
    \begin{align*}
       \left(\dashint_{B_{R}}|u\psi|^{p_1} \,dx\right)^{\frac{p}{p_1}}\leq \frac{c}{(R-\rho)^{np}}\left[\dashint_{B_R}|u|^p\,dx+\mathrm{Tail}(u;B_R)^p +\left(\dashint_{B_R\cap\Omega}|f|^{p_*}\,dx\right)^{\frac{p}{p_*(p-1)}}\right]
    \end{align*}
    for some constant $c=c(n,s,p)$, where 
    \begin{align*}
        p_1\coloneqq\begin{cases}
            np/(n-sp)&\quad\text{if }sp<n,\\
            2p&\quad\text{if }sp\geq n.
        \end{cases}
    \end{align*}
    Note in particular that $p-1 < p < p_1$. Accordingly, we have the following interpolation inequality:
    \begin{equation*}
        \left(\dashint_{B_R}|u|^{p}\,dx \right)^{1/p} \le \left(\dashint_{B_R}|u|^{p-1}\,dx\right)^{\theta/(p-1)}\left(\dashint_{B_R}|u|^{p_1}\,dx\right)^{(1-\theta)/p_1},
    \end{equation*}
    where the constant $\theta = \theta(n,s,p)\in(0,1)$ is given by
    \begin{align*}
        \frac1p=\frac{\theta}{p-1}+\frac{1-\theta}{p_1}.
    \end{align*}
    Using this and Young's inequality, along with the fact that $\psi\equiv 1$ in $B_\rho$ and $7/8\leq \rho < R\leq 1$, we get
    \begin{align*}
       \left(\int_{B_{\rho}}|u|^{p_1} \,dx\right)^{\frac{p}{p_1}}&\leq \frac{1}{2}\left(\int_{B_{R}}|u|^{{p}_1} \,dx\right)^{\frac{p}{{p}_1}}\\
       &\,+ \frac{c}{(R-\rho)^{\frac{np}{\theta}}}\left[\left(\dashint_{B_R}|u|^{p-1}\,dx\right)^{\frac{p}{p-1}}+\mathrm{Tail}(u;B_R)^p +\left(\dashint_{B_R\cap\Omega}|f|^{p_*}\,dx\right)^{\frac{p}{p_*(p-1)}}\right]
    \end{align*}
    for some constant $c=c(n,s,p)$.
    After a few simple computations, we arrive at
    \begin{align*}
       \left(\int_{B_{\rho}}|u|^{p_1} \,dx\right)^{\frac{p}{p_1}}\leq \frac{1}{2}\left(\int_{B_{R}}|u|^{{p}_1} \,dx\right)^{\frac{p}{p_1}}
       + \frac{c}{(R-\rho)^{\frac{np}{\theta}}}\left[\widetilde{E}(u;B_1)^p +\left(\dashint_{B_1\cap\Omega}|f|^{p_*}\,dx\right)^{\frac{p}{p_*(p-1)}}\right]
    \end{align*}
    for some constant $c=c(n,s,p)$, whenever $7/8\leq \rho< R\leq 1$. By using \cite[Lemma 6.1]{Giu03}, we get 
    \begin{align}\label{ineq4.ene}
      \int_{B_{7/8}}|u|^{p}\,dx\leq \left(\int_{B_{7/8}}|u|^{p_1}\,dx\right)^{p/p_1}\leq c\left[\widetilde{E}(u;B_1)^p +\left(\dashint_{B_1\cap\Omega}|f|^{p_*}\,dx\right)^{\frac{p}{p_*(p-1)}}\right],
    \end{align}
    where we have also used H\"older's inequality for the first inequality. On the other hand, by \eqref{ineq3.ene} with this time $\rho=3/4$ and $R=7/8$, we have 
    \begin{align*}
        [u]^p_{W^{s,p}(B_{3/4})}\leq c\left(\|u\|^p_{L^p(B_{7/8})}+\mathrm{Tail}(u;B_{7/8})^p+\|f\|^{p/(p-1)}_{L^{p_*}(B_{7/8}\cap \Omega)}\right).
    \end{align*}
    Combining this last display and \eqref{ineq4.ene} leads to the desired estimate.
\end{proof}
We next give a boundedness result.
\begin{lemma}\label{lem.bdd}
    Let $u\in W^{s,p}(B_r(x_0))\cap L^{p-1}_{sp}(\bbR^n)$ be a weak solution to 
     \begin{equation*}
\left\{
\begin{alignedat}{3}
(-\Delta_p)^s{u}&= f&&\qquad \mbox{in  $B_r(x_0)\cap \Omega$}, \\
{u}&=0&&\qquad  \mbox{in $B_r(x_0)\setminus \Omega$},
\end{alignedat} \right.
\end{equation*}
where $f\in L^\infty(B_r(x_0) \cap \Omega)$ and $\Omega$ satisfies \eqref{cond.meas} with $R_0\leq r$. Then there is a constant $c=c(n,s,p)$ such that
\begin{equation*}
\|u\|_{L^\infty(B_{7r/8}(x_0))}\leq c\widetilde{E}(u;B_{r}(x_0))+c\|r^{sp}f\|^{1/(p-1)}_{L^\infty(B_r(x_0)\cap\Omega)}.
\end{equation*}
\end{lemma}
\begin{proof}
    Similarly to the proof of \cite[Theorem 5]{KorKuuPal16}, we get 
    \begin{align*}
         \|u\|_{L^\infty(B_{7r/8}(x_0))}&\leq c\left(\dashint_{B_{r}(x_0)}|u|^p \,dx\right)^{1/p}+c\mathrm{Tail}(u;B_r(x_0))^p+c\|r^{sp}f\|^{1/(p-1)}_{L^\infty(B_r(x_0)\cap\Omega)}
    \end{align*}
    by using Lemma \ref{lem.ene} and a standard De Giorgi type iteration as in \cite[Proposition 6.1]{Coz17}. We now follow the arguments given in \cite[Corollary 2.1]{KuuMinSir15} to get the desired estimate.
\end{proof}
We end this section with a boundary version of the localization argument.
\begin{lemma}\label{lem.locarg}
    Let $u\in W^{s,p}(B_{4R}(x_0))\cap L^{p-1}_{sp}(\bbR^n)$ be a weak solution to 
    \begin{equation*}
\left\{
\begin{alignedat}{3}
(-\Delta_p)^s{u}&= f&&\qquad \mbox{in  $\Omega \cap B_{4R}(x_0)$}, \\
{u}&=0&&\qquad  \mbox{in $B_{4R}(x_0)\setminus \Omega$}.
\end{alignedat} \right.
\end{equation*}
Let us fix a cut-off function $\psi\in C^\infty_{c}(B_{11R/4}(x_0))$ with $0 \le \psi \le 1$ and $\psi\equiv 1$ in $B_{5R/2}(x_0)$. Then  $v\coloneqq u\psi\in W^{s,p}(\bbR^n)$ is a weak solution to 
\begin{equation}\label{eq.locarg}
\left\{
\begin{alignedat}{3}
(-\Delta_p)^s{v}&= f+\mu&&\qquad \mbox{in  $\Omega \cap B_{2R}(x_0)$}, \\
{v}&=0&&\qquad  \mbox{in $B_{2R}(x_0)\setminus \Omega$},
\end{alignedat} \right.
\end{equation}
where 
\begin{equation}\label{mu.locarg}
    \mu(x)\coloneqq \int_{\bbR^n\setminus B_{5R/2}(x_0)}\frac{(|d_sv|^{p-2}d_sv)(x,y)-(|d_su|^{p-2}d_su)(x,y)}{|x-y|^{n+s}}\,dy, \quad x \in B_{2R}(x_0)
\end{equation}
satisfies
\begin{align*}
    |\mu(x)|\leq cR^{-sp}\left(|u(x)|^{p-1}+\mathrm{Tail}(u;B_{5R/2}(x_0))^{p-1}\right).
\end{align*}
\end{lemma}
\begin{proof}
    By following the proof of \cite[Lemma 3.1]{DieKimLeeNow24}, we see that $v$ is a weak solution to \eqref{eq.locarg} with the right-hand side $\mu=\mu(x)$ determined in \eqref{mu.locarg}. Due to the facts that $|v(x)|\leq |u(x)|$ for any $x\in \bbR^n$ and that 
    \begin{align*}
        |x-y|\geq |y-x_0|-|x-x_0|\geq |y-x_0|/4
    \end{align*}
    for any $y\in \bbR^n\setminus B_{5R/2}(x_0)$ and  $x\in B_{2R}(x_0)$, we have 
    \begin{align*}
        |\mu(x)|\leq c\int_{\bbR^n\setminus B_{5R/2}(x_0)}\frac{|u(x)|^{p-1}+|u(y)|^{p-1}}{|x_0-y|^{n+sp}}\,dy
    \end{align*}
    for some constant $c=c(n,s,p)$. After a few simple computations, we get the desired estimate.
\end{proof}

\section{Regularity in flat domains}
In this section, we prove a supremum estimate for $v/d^s$, where $v$ is a weak solution to a nonhomogeneous fractional $p$-Laplace equation in a flat domain, by employing a barrier constructed in \cite[Lemma 4.3]{IanMosSqu16}.
\begin{lemma}\label{lem.dsest}
    Let $v\in W^{s,p}(B_r(z))\cap L^{p-1}_{sp}(\bbR^n)$ be a weak solution to 
    \begin{equation*}
\left\{
\begin{alignedat}{3}
(-\Delta_p)^s{v}&= f&&\qquad \mbox{in  $B_r^{+}(z)$}, \\
{v}&=0&&\qquad  \mbox{in $B_r^{-}(z)$},
\end{alignedat} \right.
\end{equation*}
where $f\in L^\infty(B_r^+(z))$. Then there exists a constant $c=c(n,s,p)$ such that 
\begin{equation*}
    \|v/d_{B^+_r(z)}^s\|_{L^{\infty}(B^{+}_{r/2}(z))}\leq c\left(\widetilde{E}(v/r^s;B_{r}(z))+\|r^{s}f\|^{1/(p-1)}_{L^\infty(B_r^+(z))}\right).
\end{equation*}
\end{lemma}
\begin{proof}
Since $v_{r,z}(x)\coloneqq v(rx+z)/r^s$ is a weak solution to 
\begin{equation*}
\left\{
\begin{alignedat}{3}
(-\Delta_p)^s{v}_{r,z}&= f_{r,z}&&\qquad \mbox{in  $B_1^{+}$}, \\
{v}_{r,z}&=0&&\qquad  \mbox{in $B_1^{-}$},
\end{alignedat} \right.
\end{equation*}
where $f_{r,z}(x)\coloneqq r^sf(rx+z)$, by Lemma \ref{lem.rei}, it suffices to consider the case that $B_r(z)=B_1$.
First, we note from Lemma \ref{lem.bdd} that
    \begin{align}\label{ineq1.dsest}
        \|v\|_{L^\infty(B_{7/8})}\leq c_0\left(\widetilde{E}(v;B_1)+\|f\|^{1/(p-1)}_{L^\infty(B_1^+)}\right)
    \end{align}
    for some constant $c_0=c_0(n,s,p)$. Next, let us choose $x_0=(x_0',x_{0,n})\in B_{1/2}^{+}$ and define $v_0(x)\coloneqq v(x+(x_0',0))$ to see that $v_0$ is a weak solution to
     \begin{equation*}
\left\{
\begin{alignedat}{3}
(-\Delta_p)^s{v_0}&= f_0&&\qquad \mbox{in  $B_{1/4}^{+}$}, \\
{v_0}&=0&&\qquad  \mbox{in $B_{1/4}^{-}$},
\end{alignedat} \right.
\end{equation*}
where $f_0(x)\coloneqq f(x+(x_0',0))$.
We now consider a cut-off function $\psi\in C_c^\infty(B_{3/8})$ with $0 \le \psi \le 1$ and  $\psi\equiv 1$ in $B_{1/4}$. Then we see that $h\coloneqq v_0\psi$ is a weak solution to 
    \begin{equation}\label{heq.dsest}
\left\{
\begin{alignedat}{3}
(-\Delta_p)^s{h}&= g&&\qquad \mbox{in  $B_{1/8}^{+}$}, \\
{h}&=0&&\qquad  \mbox{in $(\bbR^n\setminus B_{3/8})\cup \bbR^{n}_-$},
\end{alignedat} \right.
\end{equation}
where
\begin{align*}
    g(x)\coloneqq &-\int_{\bbR^n\setminus B_{1/4}}\frac{|v_0(x)-v_0(y)|^{p-2}(v_0(x)-v_0(y))}{|x-y|^{n+sp}}\,dy\\
    &\quad+\int_{\bbR^n\setminus B_{1/4}}\frac{|v_0(x)-h(y)|^{p-2}(v_0(x)-h(y))}{|x-y|^{n+sp}}\,dy+f_0(x)\quad\text{in }B^+_{1/8}.
\end{align*}
A simple computation, together with \eqref{ineq1.dsest} and the definition of $v_0$, shows that
\begin{equation}\label{ineq2.dsest}
\begin{aligned}
    \|g\|_{L^\infty(B^+_{1/8})}&\leq c\int_{\bbR^n\setminus B_{1/4}}\frac{|v_0(x)|^{p-1}+|v_0(y)|^{p-1}}{|y|^{n+sp}}\,dy + c\|f_0\|_{L^\infty(B^+_{1/8})} \\
    &\leq c_1\left(\widetilde{E}(v;B_1) +\|f\|_{L^\infty(B_1^+)}^{1/(p-1)}\right)^{p-1}
\end{aligned}
\end{equation}
for some constant $c_1=c_1(n,s,p)$.
By \cite[Lemma 4.3]{IanMosSqu16}, there is a barrier $w\in C^s(\bbR^n)$ such that 
\begin{equation*}
\left\{
\begin{alignedat}{3}
(-\Delta_p)^s&{w}\geq a&&\qquad \mbox{in  $B_{r_0}(e_n)\setminus \overline{B}_1$}, \\
c_2^{-1}(|x|-1)_+^s\leq& w\leq c_2(|x|-1)_+^s&&\qquad  \mbox{in $\bbR^n$}
\end{alignedat} \right.
\end{equation*}
for some $r_0=r_0(n,s,p)>0$, $a=a(n,s,p)\in(0,1)$ and $c_2=c_2(n,s,p)\geq1$. Note that we may assume $r_0\leq 1/8$, as $w$ also weakly solves
\begin{equation*}
    (-\Delta_p)^sw\geq a\quad\text{in }B_{\min\{r_0,1/8\}}(e_n)\setminus \overline{B}_1.
\end{equation*}
Now, for $\widetilde{w}(x)\coloneqq w(x+e_n)$ and any constant $M \ge 1$, we observe that
\begin{equation}\label{eq2.dsest}
\left\{
\begin{alignedat}{3}
(-\Delta_p)^s&{(M\widetilde{w})}\geq M^{p-1}a&&\qquad \mbox{in  $B_{r_0}\setminus \overline{B_{1}(-e_n)}$ }, \\
c_2^{-1}M(|x+e_n|-1)_+^s\leq &M\widetilde{w}\leq c_2M(|x+e_n|-1)_+^s&&\qquad  \mbox{in $\bbR^n$}.
\end{alignedat} \right.
\end{equation}
By taking 
\begin{equation}\label{cond.M.dsest}
    M\coloneqq\left[(c_1/a)^{1/(p-1)}+c_0c_2(4n/r_0)^s\right]\left(\widetilde{E}(v;B_1)+\|f\|^{1/(p-1)}_{L^\infty(B_1^+)}\right)
\end{equation}
and using \eqref{ineq2.dsest}, we have
\begin{equation}\label{comp1.dsest}
    (-\Delta_p)^sh\leq (-\Delta_p)^s(M\widetilde{w})\quad\text{in }B_{r_0}^{+}.
\end{equation}
In light of \eqref{heq.dsest} and \eqref{eq2.dsest}, we observe 
\begin{align}\label{comp2.dsest}
    h=0\leq M\widetilde{w}\quad\text{in }\bbR_-^n\cup (\bbR^n_+\setminus B_{3/8}).
\end{align}
We also observe that, for any $x = (x',x_n)\in \bbR_+^{n}\setminus B_{r_0}^+$, either
\begin{align*}
    |x'|\geq r_0/(4n)\quad\text{or}\quad x_n\geq r_0/(4n)
\end{align*}
holds, which implies
\begin{align}\label{ineq3.dsest}
    c_2^{-1}M(r_0/4n)^s \leq c_2^{-1}M(|x+e_n|-1)_+^s\quad\text{for any }x=(x',x_n)\in \bbR_+^{n}\setminus B_{r_0}^+.
\end{align}
Therefore, using \eqref{ineq1.dsest} together with the fact that $|h(x)|=|(v_0\psi)(x)|\leq |v_0(x)|$, \eqref{cond.M.dsest}, \eqref{eq2.dsest} and \eqref{ineq3.dsest},  we have 
\begin{equation}\label{comp3.dsest}
  h(x)\leq \|  v\|_{L^\infty(B_{7/8})} \leq c_0\left(\widetilde{E}(v;B_1)+\|f\|^{1/(p-1)}_{L^\infty(B_1^+)}\right) \leq c_2^{-1}M(r_0/4n)^s
  \leq M\widetilde{w}(x)
\end{equation}
for any $x\in B_{3/4}^+\setminus B_{r_0}^+$.
Therefore, by the comparison principle given in \cite[Proposition 2.10]{IanMosSqu16}, \eqref{comp1.dsest}, \eqref{comp2.dsest} and \eqref{comp3.dsest} imply
\begin{equation*}
    h\leq M\widetilde{w}\quad\text{in }B_{r_0}^{+},
\end{equation*}
which gives 
\begin{equation*}
    |h(x)|\leq c\left(\widetilde{E}(v;B_1)+\|f\|^{1/(p-1)}_{L^\infty(B_1^+)}\right)(|x+e_n|-1)_+^s\quad\text{in }B_{r_0}^{+}.
\end{equation*}
When $x_{0,n}\leq r_0$, 
since $v_0\equiv h$ in $B_{1/4}$, we have 
\begin{equation*}
|v(x_0',x_{0,n})|=|v_0(0,x_{0,n})|=|h(0,x_{0,n})|\leq c\left(\widetilde{E}(v;B_1)+\|f\|^{1/(p-1)}_{L^\infty(B_1^+)}\right)d_{B^+_{1}}^s(x_0).
\end{equation*}
When instead $x_{0,n}\geq r_0$, since $r_0<1$ depends only on $n,s$ and $p$, we have $x_{0,n}\leq cd_{B^+_1}(x_0)$ and so
\begin{equation*}
    |v(x_0)|=|v(x_0',x_{0,n})|\leq \|v\|_{L^\infty(B_{7/8})} \leq c\left(\widetilde{E}(v;B_1)+\|f\|^{1/(p-1)}_{L^\infty(B_1^+)}\right)d_{B^+_1}^s(x_0)
\end{equation*}
for some constant $c=c(n,s,p)$.
Since we have arbitrarily chosen $x_0\in B^{+}_{1/2}$, we conclude with
\begin{align*}
    \|v/d^s\|_{L^\infty(B^{+}_{1/2})}\leq c\left(\widetilde{E}(v;B_1)+\|f\|^{1/(p-1)}_{L^\infty(B_1^+)}\right)
\end{align*}
for some constant $c=c(n,s,p)$.
\end{proof}

\section{Comparison estimates}
In this section, we deal with a weak solution $u\in W^{s,p}(B_{5r}(z))\cap L^{p-1}_{sp}(\bbR^n)$ to 
    \begin{equation}\label{eq.main.comp}
\left\{
\begin{alignedat}{3}
(-\Delta_p)^s{u}&= f&&\qquad \mbox{in  $\Omega \cap B_{5r}(z)$}, \\
{u}&=0&&\qquad  \mbox{in $B_{5r}(z)\setminus \Omega$},
\end{alignedat} \right.
\end{equation}
where $d_\Omega(z)\leq r$ and $\Omega$ is $(\delta,5r)$-Reifenberg flat with $\delta \in (0,1/2)$ to ensure \eqref{cond.meas}.
We start with a nonlocal Hardy-type inequality.
\begin{lemma}\label{lem.hardy}
    For $\alpha\in(0,1)$ and $p>1$, assume that $u\in W^{\alpha,p}(B_r(z))$ satisfies $u\equiv 0$ in $B_r(z)\setminus \Omega$.
    Then for any $\rho\in[1/2,1)$, we have
    \begin{align}\label{ineq1.hardy}
        \left(\int_{ B_{\rho r}(z)\cap \Omega}\left|\frac{u}{d_\Omega^\alpha}\right|^p\,dx\right)^{1/p}\leq c[u]_{W^{\alpha,p}(B_r(z))}
    \end{align}
    where $c=c(n,\alpha,p,1/(1-\rho))$. 
\end{lemma}
\begin{proof}
    The proof is based on that of \cite[Lemma 2.4]{KimWei24} together with a localization argument. Let $\psi\in C_c^\infty(B_{(\rho+3)r/4}(z))$ be a cut-off function such that $0\le \psi \le 1$, $\psi\equiv 1$ in $B_{(\rho+1)r/2}(z)$ and $|\nabla\psi|\leq \frac{c}{(1-\rho)r}$. We first observe that $w\coloneqq u\psi$ satisfies
    \begin{align}\label{ineq2.hardy}
       \left(\int_{ B_{r/2}(z)\cap \Omega}\left|\frac{w}{d_\Omega^\alpha}\right|^p\,dx\right)^{1/p}\leq c[w]_{W^{\alpha,p}(\bbR^n)}
    \end{align}
    for some constant $c=c(n,\alpha,p)$. Note that the domain $\Omega$ is $(\delta,5r)$-Reifenberg so that \eqref{cond.meas} holds with $R_0=5r$.
   Thus, by following the proof of \cite[Lemma 2.4]{KimWei24} together with the facts that $d_\Omega(x)\leq 2r$ for any $x\in B_{r/2}(z)\cap \Omega$ and that \eqref{cond.meas} holds with $R_0=5r$, we get \eqref{ineq2.hardy}. Since $w\equiv 0$ in $\bbR^n\setminus B_{(\rho+3)r/4}(z)$ and $|w(x)|\leq |u(x)|$, we have
    \begin{align*}
        [w]_{W^{s,p}(\bbR^n)}^{p}\leq [w]_{W^{s,p}(B_{r}(z))}^{p}+c\int_{\bbR^n\setminus B_r(z)}\int_{B_{(\rho+3)r/4}(z)}\frac{|u(x)|^p}{|x-y|^{n+sp}}\,dx\,dy\eqqcolon I_1+I_2.
    \end{align*}
    Using \cite[Lemma 4.7]{Coz17} together with the facts that $0 \le \psi \le 1$ and that $|\nabla \psi|\leq \frac{c}{(1-\rho)r}$, we have
    \begin{equation}\label{ineq3.hardy}
    \begin{aligned}
        I_1&\leq [u]_{W^{s,p}(B_r(z))}^{p}+\int_{B_r(z)}\int_{B_r(z)}\frac{|u(x)|^p\|\nabla\psi\|_{L^\infty}^p}{|x-y|^{n+sp-p}}\,dx\,dy\\
        &\leq [u]_{W^{s,p}(B_r(z))}^{p}+\frac{c}{(1-\rho)^p}\int_{B_r(z)}|u(x)/r^s|^p\,dx\leq \frac{c}{(1-\rho)^p}[u]^p_{W^{s,p}(B_r(z))}
    \end{aligned}
    \end{equation}
    for some $c=c(n,s,p)$. 
    Using \cite[Lemma 4.7]{Coz17} together with the fact that $u\equiv 0$ in $B_r(z)\setminus \Omega$ and \eqref{cond.meas}, we obtain
    \begin{equation}\label{ineq4.hardy}
    \begin{aligned}
        I_2\leq c\int_{\bbR^n\setminus B_r(z)}\int_{B_{(\rho+3)r/4}(z)}\frac{|u(x)|^p}{|z-y|^{n+sp}}\,dx\,dy&\leq \frac{c}{(1-\rho)^p}\int_{B_r(z)}|u(x)/r^s|^p\,dx\\
        &\leq \frac{c}{(1-\rho)^p}[u]^p_{W^{s,p}(B_r(z))}
    \end{aligned}
    \end{equation}
    for some $c=c(n,s,p)$. Thus, \eqref{ineq1.hardy} follows from \eqref{ineq3.hardy},  \eqref{ineq4.hardy} and the fact that $u=w$ in $B_{r/2}(z)$.
\end{proof}
Now we are ready to prove the first comparison estimate.
\begin{lemma}\label{lem.comp0}
    Let $u$ be a weak solution to \eqref{eq.main.comp}. Then there is a weak solution $w\in W^{s,p}(B_{4r}(z))\cap L^{p-1}_{sp}(\bbR^n)$ to 
    \begin{equation}\label{eq.comp0}
\left\{
\begin{alignedat}{3}
(-\Delta_p)^s{w}&= 0&&\qquad \mbox{in  $\Omega \cap B_{4r}(z)$}, \\
{w}&=u&&\qquad  \mbox{in $\bbR^n\setminus (\Omega\cap B_{4r}(z))$}
\end{alignedat} \right.
\end{equation}
such that 
\begin{align*}
    \left(\dashint_{\Omega\cap B_{4r}(z)}\left|\frac{u-w}{d^s}\right|^{{p-1}}\,dx\right)^{\frac1{p-1}}
    &\leq  c\left[\mbox{\Large$\chi$}_{\{p<2\}}\widetilde{E}\left(\frac{u}{r^s};B_{5r}(z)\right)+\left(\dashint_{\Omega\cap B_{5r}(z)}|r^sf|^{p_*}\,dx\right)^{\frac{1}{p_*(p-1)}}\right]^{\frac{2-p}{2}}\\
        &\quad\times \left(\dashint_{\Omega\cap B_{5r}(z)}|r^sf|^{p_*}\,dx\right)^{\frac{p}{2p_*(p-1)}}
\end{align*}
for some $c=c(n,s,p)$, where the constant $p_*$ is given in  \eqref{defn.pstar}.
\end{lemma}
\begin{proof}
Note that the existence of a weak solution $w$ to \eqref{eq.comp0} follows from \cite[Proposition 2.12]{BraLinSch18}. 
By Lemma \ref{lem.scale} and Lemma \ref{lem.rei}, we may assume $z=0$ and $r=1$.
By testing $u-w$ into 
    \begin{equation*}
        (-\Delta_p)^su-(-\Delta_p)^sw=f\quad\text{in }\Omega\cap B_4,
    \end{equation*}
    we derive
    \begin{align}\label{ineq1.comp0}
        \iint_{\bbR^{2n}}(|d_su|+|d_sw|)^{p-2}|d_s(u-w)|^2\frac{\,dx\,dy}{|x-y|^n}\leq c\left(\int_{\Omega\cap B_{4}}|f|^{p_*}\,dx\right)^{\frac{p}{p_*(p-1)}}
    \end{align}
    for some $c=c(n,s,p)$. When $p\geq2$, \eqref{ineq1.comp0} together with \cite[Lemma 2.1]{DieKimLeeNow24} directly implies 
    \begin{equation}\label{ineq4.comp0}
    \begin{aligned}
    & \iint_{B_{9/2}\times B_{9/2}}|d_s(u-w)|^p\frac{\,dx\,dy}{|x-y|^n} \\
         &\leq c\left[\mbox{\Large$\chi$}_{\{p<2\}}\widetilde{E}(u;B_5)+\left(\dashint_{\Omega\cap B_{5}}|f|^{p_*}\,dx\right)^{\frac{1}{p_*(p-1)}}\right]^{\frac{p(2-p)}{2}}\left(\dashint_{\Omega\cap B_{5}}|f|^{p_*}\,dx\right)^{\frac{p^2}{2p_*(p-1)}}
    \end{aligned}
    \end{equation}
    for some $c=c(n,s,p)$. We now prove \eqref{ineq4.comp0} when $1<p<2$. 
    To do this, by testing $w-u$ into \eqref{eq.comp0}, we derive
    \begin{align*}
        \iint_{B_{9/2}\times B_{9/2}}|d_sw|^p\frac{\,dx\,dy}{|x-y|^n}\leq c\iint_{(B_{9/2}\times B_{9/2})\cup\mathscr{C}\Omega_{9/2}}|d_su|^p\frac{\,dx\,dy}{|x-y|^n}
    \end{align*}
    for some constant $c=c(n,s,p)$, where $\Omega_{9/2}\coloneqq\Omega\cap B_{9/2}$ and $\mathscr{C}\Omega_{9/2}\coloneqq(\bbR^n\setminus \Omega_{9/2}\times \bbR^n\setminus \Omega_{9/2})^c$. 
    Since 
    \begin{equation}\label{ineq11.comp0}
    \begin{aligned}
        \iint_{(B_{9/2}\times B_{9/2})\cup\mathscr{C}\Omega_{9/2}}|d_su|^p\frac{\,dx\,dy}{|x-y|^n}&\leq c\iint_{B_{19/4}\times B_{19/4}}|d_su|^p\frac{\,dx\,dy}{|x-y|^n} +c\mathrm{Tail}(u;B_{19/4})^p,
    \end{aligned}
    \end{equation}
    we further estimate
    \begin{align}\label{ineq2.comp0}
        \iint_{B_{9/2}\times B_{9/2}}|d_sw|^p\frac{\,dx\,dy}{|x-y|^n}\leq c\left[\widetilde{E}(u;B_5)+\left(\dashint_{\Omega\cap B_{5}}|f|^{p_*}\,dx\right)^{\frac{1}{p_*(p-1)}}\right]^p,
    \end{align}
    where we have used Lemma \ref{lem.ene}.
    In light of H\"older's inequality, we next obtain
    \begin{equation}\label{ineq3.comp0}
    \begin{aligned}
        &\iint_{B_{9/2}\times B_{9/2}}|d_s(u-w)|^p\frac{\,dx\,dy}{|x-y|^n}\\
        &=\iint_{B_{9/2}\times B_{9/2}}(|d_su|+|d_sw|)^{-p(p-2)/2}(|d_su|+|d_sw|)^{p(p-2)/2}|d_s(u-w)|^p\frac{\,dx\,dy}{|x-y|^n}\\
        &\leq \left(\iint_{B_{9/2}\times B_{9/2}}(|d_su|+|d_sw|)^{p}\frac{\,dx\,dy}{|x-y|^n}\right)^{(2-p)/2}\\ &\qquad\times\left(\iint_{B_{9/2}\times B_{9/2}}(|d_su|+|d_sw|)^{p-2}|d_s(u-v)|^2\frac{\,dx\,dy}{|x-y|^n}\right)^{p/2}.
    \end{aligned}
    \end{equation}
    We now plug \eqref{ineq1.comp0}, \eqref{ineq11.comp0} and \eqref{ineq2.comp0} into the right-hand side of \eqref{ineq3.comp0}, in order to obtain \eqref{ineq4.comp0}. 
    In addition, by H\"older's inequality and Lemma \ref{lem.hardy}, we have
    \begin{align*}
         \left(\dashint_{\Omega\cap B_4}|(u-w)/d^s|^{p-1}\,dx\right)^{1/(p-1)}&\leq \left(\dashint_{\Omega\cap B_4}|(u-w)/d^s|^{p}\,dx\right)^{1/p}\leq c[u-w]_{W^{s,p}(B_{9/2})}.
    \end{align*}
    By plugging this into the left-hand side of \eqref{ineq4.comp0}, we conclude with the desired estimate.
\end{proof}
Since we are going to investigate the function $u/d^{s}_{\Omega}$, we also consider the functional
\begin{equation*}
    \mathcal{E}_\Omega(w;B_{r}(x_0))\coloneqq \left(\dashint_{\Omega\cap B_r(x_0)}|w/d_\Omega^s|^{{p-1}}\,dx\right)^{1/(p-1)}+\max\{d_\Omega(x_0),r\}^{-s}\mathrm{Tail}(w;B_r(x_0)).
\end{equation*}
Note that such a functional is first introduced in \cite{KimWei24} when $p=2$. By considering Lemma \ref{lem.scale} and Lemma \ref{lem.rei}, we deduce 
\begin{align*}
    \mathcal{E}_{\Omega_{r,x_0}}(u_{r,x_0};B_{1})=\mathcal{E}_{\Omega}(u;B_r(x_0)),
\end{align*}
where the function $u_{r,x_0}(x)\coloneqq u(rx+x_0)/r^s$ and the domain $\Omega_{r,x_0}$ are determined in Lemma \ref{lem.scale}.  Moreover, for any $\rho \in (0,1)$, we have
\begin{align}\label{ineq.ftnl}
    \mathcal{E}_\Omega(w;B_{\rho r}(x_0))\leq c(n,s,p,\rho)\mathcal{E}_\Omega(w;B_{ r}(x_0)).
\end{align}
Before presenting the second comparison estimate, we note the following lemma.
\begin{lemma}\label{lem.C0}
    Let $w\in W^{s,p}(B_{4r}(z))\cap L^{p-1}_{sp}(\bbR^n)$ be a weak solution to 
    \begin{equation*}
\left\{
\begin{alignedat}{3}
(-\Delta_p)^s{w}&= 0&&\qquad \mbox{in  $\Omega \cap B_{4r}(z)$}, \\
{w}&=0&&\qquad  \mbox{in $B_{4r}(z)\setminus \Omega$}
\end{alignedat} \right.
\end{equation*}
with 
\begin{equation}\label{ineq0.C0}
    \mathcal{E}_\Omega(w;B_{4r}(z))\leq \lambda
\end{equation} 
for some $\lambda>0$.
In addition, assume that $\Omega$ satisfies \eqref{cond.meas} with $R_0=4r$.
Let us fix a cut-off function $\xi\in C_c^\infty(B_{11r/4}(z))$ with $0 \le \xi \le 1$, $\xi\equiv 1$ in $B_{5r/2}(z)$ and $|\nabla\xi|\leq 8/r$. Then $v\coloneqq w\xi$ is a weak solution to
 \begin{equation}\label{eq.c0.seq}
\left\{
\begin{alignedat}{3}
(-\Delta_p)^s{v}&= g&&\qquad \mbox{in  $\Omega \cap B_{2r}(z)$}, \\
{v}&=0&&\qquad  \mbox{in $B_{2r}(z)\setminus \Omega$}
\end{alignedat} \right.
\end{equation}
for some $g\in L^\infty(B_r(z))$. 
Moreover, there is a constant $c_3=c_3(n,s,p)\geq1$ such that
\begin{equation*}
    \left(\int_{B_{3r/2}(z)}\dashint_{B_{3r/2}(z)}|d_sv|^p\frac{\,dx\,dy}{|x-y|^n}\right)^{1/p}+\mathcal{E}_\Omega(v;B_{r}(z))+\left\|r^{s}g\right\|^{1/(p-1)}_{L^\infty(B_{2r}(z))}\leq c_3\lambda.
\end{equation*}
\begin{proof}
   By Lemma \ref{lem.locarg}, with $u$ replaced by $w$, we observe $v=w\xi$ is a weak solution to \eqref{eq.c0.seq} with 
    \begin{align}\label{ineq1.C0}
        \|g\|_{L^\infty(B_{2r}(z))}\leq cr^{-sp}\left(\|w\|^{p-1}_{L^\infty(B_{2r}(z))}+\mathrm{Tail}(w;B_{5r/2}(z))^{p-1}\right)
    \end{align}
    for some constant $c=c(n,s,p)$. By Lemma \ref{lem.ene} and Lemma \ref{lem.bdd}, we have 
    \begin{align}\label{ineq2.C0}
        r^{s-n/p}[w]_{W^{s,p}(B_{2r}(z))}+\|w\|_{L^\infty(B_{2r}(z))}\leq c\widetilde{E}(w;B_{4r}(z)).
    \end{align}
    Since $d_{\Omega}(x)\leq 5r$ for any $x\in B_{4r}(z)$, we have
    \begin{align}\label{ineq3.C0}
        \widetilde{E}(w;B_{4r}(z))+\mathrm{Tail}(w;B_{5r/2}(z))\leq cr^s\mathcal{E}_\Omega(w;B_{4r}(z))
    \end{align}
    for some constant $c=c(n,s,p)$. Combining all the estimates \eqref{ineq1.C0}, \eqref{ineq2.C0} and \eqref{ineq3.C0} together with \eqref{ineq0.C0} yields 
    \begin{align}\label{ineq4.C0}
       r^{-n/p}[v]_{W^{s,p}(B_{3r/2}(z))} +\|r^sg\|^{1/(p-1)}_{L^\infty(B_{2r}(z))}\leq c\lambda
    \end{align}
    for some constant $c=c(n,s,p)$. In addition, a few simple computations lead to
    \begin{align*}
        \mathcal{E}_\Omega(v;B_{r}(z))\leq c\mathcal{E}_\Omega(w;B_{4r}(z))\leq c\lambda.
    \end{align*}
    Using this and \eqref{ineq4.C0}, we get the desired estimate.
\end{proof}
\end{lemma}
As in the proof of Lemma \ref{lem.hardy}, we find that there is a constant $c_4=c_4(n,s,p) \ge 1$ such that for any $u\in W^{s,p}(B_{2r})$ with $u\equiv 0$ in $B_{2r}^{-}$, there holds
\begin{equation}\label{defn.c1}
    \left(\dashint_{B_r^+}|u/d^s_{B_{2r}^+}|^p\,dx\right)^{1/p}\leq c_4\left(\dashint_{B_{3r/2}}\int_{B_{3r/2}}|d_su|^p\frac{\,dx\,dy}{|x-y|^n}\right)^{1/p}.
\end{equation}
Recalling the constant $c_3=c_3(n,s,p)$ determined in Lemma \ref{lem.C0}, here we fix the constant $C_0 = C_0 (n,s,p) \ge 1$ as
\begin{equation}\label{defn.C0}
    C_0=5c_3c_4.
\end{equation}
With this constant $C_0=C_0(n,s,p)$, we are now ready to prove the second comparison estimate.
\begin{lemma}\label{lem.comprei}
    For any $\epsilon>0$, there is a constant $\delta=\delta(n,s,p,\epsilon)>0$ such that for any weak solution $w\in W^{s,p}(B_{4r}(z))\cap L^{p-1}_{sp}(\bbR^n)$ to 
    \begin{equation*}
\left\{
\begin{alignedat}{3}
(-\Delta_p)^s{w}&= 0&&\qquad \mbox{in  $\Omega \cap B_{4r}(z)$}, \\
{w}&=0&&\qquad  \mbox{in $B_{4r}(z)\setminus \Omega$},
\end{alignedat} \right.
\end{equation*}
where $\Omega$ is $(\delta,5r)$-Reifenberg flat with
\begin{equation*}
    B^+_{4r}(z)\subset B_{4r}(z)\cap\Omega\subset B_{4r}(z)\cap\{x_n>z_{n}-4r\delta\}
\end{equation*}
and
\begin{equation*}
   \mathcal{E}_\Omega(w;B_{4r}(z))\leq \lambda,
\end{equation*}
there is a weak solution $v\in W^{s,p}(B_{2r}(z))\cap L^{p-1}_{sp}(\bbR^n)$ to 
\begin{equation}\label{eq.contv}
\left\{
\begin{alignedat}{3}
(-\Delta_p)^s{v}&= g&&\qquad \mbox{in  $B^+_{2r}(z)$}, \\
{v}&=0&&\qquad  \mbox{in $B^-_{2r}(z)$}
\end{alignedat} \right.
\end{equation}
such that
\begin{equation}\label{eq.est.v.lambda}
\mathcal{E}_{B_{2r}^+(z)}(v;B_r(z))+\|r^sg\|^{1/(p-1)}_{L^\infty(B_{2r}(z))}\leq C_0\lambda
\end{equation}
and
\begin{equation*}
   \left( \dashint_{B_{r/2}^+(z)}|w/d_\Omega^s-v/d_\Omega^s|^{{p-1}}\,dx\right)^{1/(p-1)}\leq \epsilon\lambda,
\end{equation*}
where the constant $C_0=C_0(n,s,p)$ is determined in \eqref{defn.C0}.
\end{lemma}
\begin{proof}
By Lemma \ref{lem.scale} and Lemma \ref{lem.rei}, it suffices to consider the case that $r=1$, $z=0$ and $\lambda=1$.

    We prove the lemma by contradiction. Suppose that there is  a constant $\epsilon_0>0$, a sequence of domains $\{\Omega_k\}$ and a sequence of weak solutions $\{w_k\}\subset W^{s,p}(B_4)\cap L^{p-1}_{sp}(\bbR^n)$ to 
     \begin{equation*}
\left\{
\begin{alignedat}{3}
(-\Delta_p)^s{w}_k&= 0&&\qquad \mbox{in  $\Omega_k \cap B_{4}$}, \\
{w}_k&=0&&\qquad  \mbox{in $B_{4}\setminus \Omega_k$},
\end{alignedat} \right.
\end{equation*}
where $\Omega_k$ is $(1/k,4)$-Reifenberg flat with
\begin{equation}\label{cond}
    B_{4}^+\subset B_{4}\cap\Omega_k\subset B_{4}\cap\{x_n>-4/k\}
\end{equation}
and
\begin{equation}\label{ass1.comprei}
    \mathcal{E}_{\Omega_k}(w_k;B_{4})\leq 1
\end{equation}
such that  
\begin{equation}\label{contr.comprei}
    \left( \dashint_{B^+_{1/2}}|w_k/d_k^s-v/d_k^s|^{{p-1}}\,dx\right)^{1/(p-1)}> \epsilon_0
\end{equation}
for any weak solution $v$ to \eqref{eq.contv} with $r=1$, $z=0$ and $\lambda=1$. Here we have assumed $k\geq2$ and have written 
\begin{equation*}
    d_k(x)\coloneqq \mathrm{dist}(x,\bbR^n\setminus \Omega_k).
\end{equation*}
First, note that $\Omega_k$ satisfies \eqref{cond.meas} with $R_0=5$, as $\Omega_k$ is $(1/k,5)$-Reifenberg flat with $1/k\leq 1/2$.
Let us take a cut-off function $\xi\in C_c^\infty(B_{11/4})$ with $0 \le \xi \le 1$, $\xi\equiv 1$ in $B_{5/2}$ and $|\nabla \xi|\leq 8$. Then by Lemma \ref{lem.C0}, we see that $v_k\coloneqq w_k\xi$ is a weak solution to
 \begin{equation*}
\left\{
\begin{alignedat}{3}
(-\Delta_p)^s{v}_k&= g_k&&\qquad \mbox{in  $\Omega_k \cap B_{2}$}, \\
{v}_k&=0&&\qquad  \mbox{in $B_{2}\setminus \Omega_k$}
\end{alignedat} \right.
\end{equation*}
with
\begin{equation}\label{ineq00.comprei}    \left(\int_{B_{3/2}}\dashint_{B_{3/2}}|d_sv_k|^p\frac{\,dx\,dy}{|x-y|^n}\right)^{1/p}+\mathcal{E}_{\Omega_k}(v_k;B_{1})+\|g_k\|_{L^\infty(B_2)}\leq c_3,
\end{equation}
where the constant $c_3=c_3(n,s,p)$ is determined in Lemma \ref{lem.C0}. Now we want to prove 
\begin{equation*}
    \|v_k\|_{W^{s,p}(\bbR^n)}\leq c.
\end{equation*}
By Lemma \ref{lem.ene}, we get
\begin{align*}
    \dashint_{B_{3}}\int_{B_{3}}|d_s w_k|^p\frac{\,dx\,dy}{|x-y|^{n}}&\leq c\left(\dashint_{\Omega\cap B_{7/2}}|w_k|^{{p-1}}\,dx\right)^{p/(p-1)}+c\mathrm{Tail}(w_k;B_{7/2})^p
\end{align*}
for some constant $c=c(n,s,p)$. Since $d_k(x)\leq 8$ for any $x\in \Omega\cap B_{4}$, we derive
\begin{align*}
    \|w_k\|_{W^{s,p}(B_3)}\leq c\mathcal{E}_{\Omega_k}(w_k;B_{7/2})
\end{align*}
and so
\begin{align}\label{ineqk2.comprei}
    \|v_k\|_{W^{s,p}(\bbR^n)}\leq c\|w_k\|_{W^{s,p}(B_{3})}\leq c\mathcal{E}_{\Omega_k}(w;B_4)\leq c,
\end{align}
where we have used \eqref{ass1.comprei} for the last inequality. 
Therefore, by the weak compactness argument given in \cite[Theorem 7.1]{DinPalVal12}, there is a function $v_\infty\in W^{s,p}(\bbR^n)\subset W^{s,p}(B_{2})\cap L^{p-1}_{sp}(\bbR^n)$ such that
\begin{align}\label{ineqk3.comprei}
   v_k\rightharpoonup v_\infty\quad\text{in }W^{s,p}(\bbR^n),\quad v_k\to v_\infty\quad\text{in }L_{\mathrm{loc}}^p(\bbR^n)
\end{align}
and
\begin{align}\label{ineqk33.comprei}
    \lim_{k\to\infty}v_k(x)=v_\infty(x)\quad\text{for almost every }x\in \bbR^n.
\end{align}
Moreover, by the weak $*$-compactness, there exists a function $g_{\infty} \in L^{\infty}(B_2)$ such that
\begin{align}\label{ineq333.comprei}
    \lim_{k\to\infty}\int_{B_2}g_k\psi \,dx =\int_{B_2}g_\infty\psi \,dx
\end{align}
for any $\psi\in L^1(B_2)$, and that
\begin{equation}\label{ineq3333.comprei}
    \|g_\infty\|_{L^\infty(B_2)}\leq \liminf_{k}\|g_k\|_{L^\infty(B_2)}.
\end{equation}
Using \eqref{ineqk33.comprei} and \eqref{ineq00.comprei} together with Fatou's lemma, we obtain
\begin{equation}\label{ineq000.comprei}
\begin{aligned}
    &\mathcal{E}_{B_2^+}(v_\infty;B_1)+\|g_\infty\|_{L^\infty(B_2)}\\
    &\leq \left(\dashint_{ B_1^+}|v_\infty/d_{B^+_2}^s|^{{p}}\,dx\right)^{1/p}+\mathrm{Tail}(v_\infty;B_1)+\|g_\infty\|_{L^\infty(B_2)}\\
    &\leq c_4\left(\int_{B_{3/2}}\dashint_{B_{3/2}}|d_sv_\infty|^p\frac{\,dx\,dy}{|x-y|^n}\right)^{1/p}+\mathrm{Tail}(v_\infty;B_1)+\|g_\infty\|_{L^\infty(B_2)}\\
    &\leq \liminf_{k\to\infty}\left[c_4\left(\int_{B_{3/2}}\dashint_{B_{3/2}}|d_sv_k|^p\frac{\,dx\,dy}{|x-y|^n}\right)^{1/p}+\mathrm{Tail}(v_k;B_1)+\|g_k\|_{L^\infty(B_2)}\right],
\end{aligned}
\end{equation}
where we have also used the fact that $v\equiv 0$ in $B^-_{2}$ and \eqref{defn.c1}. Indeed, since $d_k(0)\leq 5$, we have
\begin{align}\label{ineq0000.comprei}
    \mathrm{Tail}(v_k;B_1)\leq 5^s\max\{d_k(0),1\}^{-s}\mathrm{Tail}(v_k;B_1)\leq 5^s\mathcal{E}_{\Omega_k}(v_k;B_1).
\end{align}
Combining \eqref{ineq000.comprei} and \eqref{ineq0000.comprei} together with \eqref{ineq00.comprei} and \eqref{ineq3333.comprei}, we get
\begin{align*}
    \mathcal{E}_{B_2^+}(v_\infty;B_1)+\|g_\infty\|_{L^\infty(B_1)}\leq 5c_3c_4=C_0.
\end{align*}
On the other hand, for any $\psi\in X_0^{s,p}(B_2^+) \subset X_0^{s,p}(\Omega_k\cap B_{2})$, we have 
\begin{align*}
    \iint_{\bbR^{2n}}|d_s v_k|^{p-2}d_sv_k d_s\psi\frac{\,dx\,dy}{|x-y|^{n}}=\int_{B_2\cap \Omega_k}g_k\psi \,dx =\int_{B^+_2}g_k\psi \,dx .
\end{align*}
Thus, by the first convergence property given in \eqref{ineqk3.comprei} and \eqref{ineq333.comprei}, we have 
\begin{align*}
    \iint_{\bbR^{2n}}|d_s v_\infty|^{p-2}d_s v_{\infty} d_s\psi\frac{\,dx\,dy}{|x-y|^{n}}=\int_{B_2^+}g_\infty\psi\,dx,
\end{align*}
which shows that $v_\infty$ is a weak solution to 
\begin{equation*}
\left\{
\begin{alignedat}{3}
(-\Delta_p)^s{v_\infty}&= g_\infty&&\qquad \mbox{in  $B^+_2$}, \\
{v}_\infty&=0&&\qquad  \mbox{in $B^-_2$}
\end{alignedat} \right.
\end{equation*}
with
\begin{equation*}
    \mathcal{E}_{B^+_2}(v_\infty;B_1)+\|g_\infty\|_{L^\infty(B_{2})}\leq C_0.
\end{equation*}
Let us fix
\begin{equation*}
  \sigma\coloneqq s-\frac{s}{2}\frac{1}{(p-1)p}
\end{equation*}
to see that $\sigma-n/(p-1)= s-n/p$. Then H\"older's inequality implies
\begin{align*}
    \left(\dashint_{B_{1/2}^+}|(v_k-v_\infty)/d_k^{s}|^{{{p-1}}}\,dx\right)^{\frac1{p-1}}&\leq \left(\dashint_{B_{1/2}^+}|(v_k-v_\infty)/d_k^{\sigma}|^{{p}}\,dx\right)^{1/p}
    \left(\dashint_{B_{1/2}^+}{d_k^{-s/2}}\,dx\right)^{\frac{1}{p(p-1)}}.
\end{align*}
Since $x_n\leq d_k(x)$ for any $x\in B^+_{1/2}$ by \eqref{cond}, we have
\begin{align*}
    \left(\dashint_{B_{1/2}^+}{d_k^{-s/2}}\,dx\right)^{\frac{1}{p(p-1)}}\leq \left(\dashint_{B_{1/2}^+}{x_n^{-s/2}}\,dx\right)^{\frac{1}{p(p-1)}}\leq c
\end{align*}
and therefore
\begin{align}\label{ineq40.comprei}
     \left(\dashint_{B_{1/2}^+}|(v_k-v_\infty)/d_k^{s}|^{{p-1}}\,dx\right)^{1/(p-1)}\leq c\left(\dashint_{B_{1/2}^+}|(v_k-v_\infty)/d_k^{\sigma}|^{{p}}\,dx\right)^{1/p}
\end{align}
for some constant $c=c(n,s,p)$. 
On the other hand, due to the fact that $w_k\equiv v_k$ in $B_{2}$, \eqref{cond} and Lemma \ref{lem.hardy}, we get
\begin{equation}\label{ineq4.comprei}
\begin{aligned}
    \|w_k/d_k^{\sigma}-v_\infty/d_k^{\sigma}\|_{L^{p}(B^+_{1/2})}&=\|v_k/d_k^\sigma-v_\infty/d_k^\sigma\|_{L^{{p}}(B^+_{1/2})}\\
    &\leq \|v_k/d_k^\sigma-v_\infty/d_k^\sigma\|_{L^p(B_{1/2}\cap \Omega_k)}\\
    &\leq c[v_k-v_\infty]_{W^{\sigma,p}(B_{1})}
\end{aligned}
\end{equation}
for some $c=c(n,s,p)$; recall that $\sigma$ depends only on $s$ and $p$. By applying \cite[Theorem 1]{BreMir18} to the last term in \eqref{ineq4.comprei}, 
we deduce
\begin{align*}
    \|v_k/d_k^\sigma-v_\infty/d_k^\sigma\|_{L^p(B^+_{1/2})}\leq c[v_k-v_\infty]^{\sigma/s}_{W^{s,p}(B_{1})}\|v_k-v_\infty\|^{(s-\sigma)/s}_{L^{p}(B_{1})}.
\end{align*}
From this last estimate, together with \eqref{ineq40.comprei}, \eqref{ineq4.comprei}, \eqref{ineqk2.comprei} and \eqref{ineqk3.comprei}, we finally have
\begin{align*}
     \left(\dashint_{B_{1/2}^+}|(v_k-v_\infty)/d_k^{s}|^{{{p-1}}}\,dx\right)^{1/(p-1)}\leq c\|v_k-v_\infty\|^{\frac{1}{2(p-1)p}}_{L^{p}(B_{1/2})}\leq \epsilon_0/2
\end{align*}
by taking sufficiently large $k>1$. This contradicts \eqref{contr.comprei}, which completes the proof.
\end{proof}
In light of Lemma \ref{lem.comprei} and the Reifenberg flatness of $\Omega$, we derive an improved comparison estimate on the full domain $\Omega\cap B_{r/2}(z)$, instead of the subdomain $B_{r/2}^+(z)$. 
\begin{lemma}\label{lem.comprei1}
For any $\epsilon>0$, there is a constant $\delta=\delta(n,s,p,\epsilon)>0$ such that for any weak solution $w\in W^{s,p}(B_{4r}(z))\cap L^{p-1}_{sp}(\bbR^n)$ to 
    \begin{equation*}
\left\{
\begin{alignedat}{3}
(-\Delta_p)^s{w}&= 0&&\qquad \mbox{in  $\Omega \cap B_{4r}(z)$}, \\
{w}&=0&&\qquad  \mbox{in $B_{4r}(z)\setminus \Omega$},
\end{alignedat} \right.
\end{equation*}
where $\Omega$ is $(\delta,5r)$-Reifenberg flat with
\begin{equation}\label{ass.rei1}
    B^+_{4r}(z)\subset B_{4r}(z)\cap\Omega\subset B_{4r}(z)\cap\{x_n>z_n-4r\delta\}
\end{equation}
and
\begin{equation}\label{ineq.rei1}
    \mathcal{E}_\Omega(w;B_{4r}(z))\leq \lambda,
\end{equation}
there is a weak solution $v\in W^{s,p}(B_{2r}(z))\cap L^{p-1}_{sp}(\bbR^n)$ to \eqref{eq.contv} such that
\begin{equation*}
   \left( \dashint_{\Omega\cap B_{r/2}(z)}|w/d_\Omega^s-v/d_\Omega^s|^{{p-1}}\,dx\right)^{1/(p-1)}\leq \epsilon\lambda\quad\text{and}\quad \|v/d_\Omega^s\|_{L^\infty(B_{r/2}(z)\cap \Omega)}\leq c\lambda
\end{equation*}
for some constant $c=c(n,s,p)$.
\end{lemma}
\begin{proof}
    Let  $\epsilon_0>0$ be a free parameter which will be chosen later. By Lemma \ref{lem.comprei}, there is a small constant $\delta=\delta(n,s,p,\epsilon_0)>0$ such that
   \begin{equation}\label{ineq0.comprei1}
   \left( \dashint_{ B^+_{r/2}(z)}|w/d_\Omega^s-v/d_\Omega^s|^{{p-1}}\,dx\right)^{1/(p-1)}\leq \epsilon_0\lambda,
\end{equation}
where $v$ is a weak solution to \eqref{eq.contv}.  Then by H\"older's inequality, we have
\begin{equation}\label{ineq1.comprei1}
\begin{aligned}
    & \left( r^{-n}\int_{ (\Omega\cap B_{r/2}(z))\setminus B_{r/2}^+(z)}\frac{|w-v|^{{p}}}{d_\Omega^{s{p}}}\,dx\right)^{1/(p-1)} \\
    &\leq \left( r^{-n}\int_{ (\Omega\cap B_{r/2}(z))\setminus B_{r/2}^+(z)}\frac{|w-v|^p}{d_\Omega^{sp}}\,dx\right)^{1/p} \left(\frac{\left|(\Omega\cap B_{r/2}(z))\setminus B_{r/2}^+(z)\right|}{r^n}\right)^{\frac{1}{p(p-1)}}.
\end{aligned}
\end{equation}
First, note from Lemma \ref{lem.hardy} and \eqref{cond.meas} with $R_0=4r$ that
\begin{align*}
    \left( r^{-n}\int_{ (\Omega\cap B_{r/2}(z))\setminus B_{r/2}^+(z)}|(w-v)/d_\Omega^s|^p\,dx\right)^{1/p}&\leq c  \left( \dashint_{ \Omega\cap B_{r/2}(z)}|(w-v)/d_\Omega^s|^p\,dx\right)^{1/p}\\
    &\leq cr^{-n/p}[w-v]_{W^{s,p}(B_{3r/4}(z))},
\end{align*}
where $c=c(n,s,p)$. Next, by Lemma \ref{lem.ene}, we have 
\begin{align*}
    r^{-n}[w]^p_{W^{s,p}(B_{3r/4}(z))}\leq c\left(\dashint_{\Omega\cap B_r(z)}|w/r^s|^{{p}}\,dx\right)^{p/(p-1)}+c\mathrm{Tail}(w/r^s;B_r(z))^p.
\end{align*}
Since $d_\Omega(x)\leq 2r$ for any $x\in B_r(z)$, we have
\begin{align*}
    r^{-n/p}[w]^p_{W^{s,p}(B_{3r/4}(z))}\leq c\mathcal{E}_\Omega(w;B_r(z))\leq c\lambda
\end{align*}
for some $c=c(n,s,p)$, where we have used \eqref{ineq.rei1} for the last inequality. Similarly, we have
\begin{align*}
    r^{-n/p}[v]_{W^{s,p}(B_{3r/4}(z))}\leq c\lambda
\end{align*}
for some constant $c=c(n,s,p)$, by using \eqref{eq.est.v.lambda}. Thus we get 
\begin{align}\label{ineq2.comprei1}
    \left( r^{-n}\int_{ (\Omega\cap B_{r/2}(z))\setminus B_{r/2}^+(z)}|(w-v)/d_\Omega^s|^{{p}}\,dx\right)^{1/p}\leq c\lambda,
\end{align}
where $c=c(n,s,p)$. In addition, by \eqref{ass.rei1}, we derive
\begin{align}\label{ineq3.comprei1}
    \left(\frac{\left|(\Omega\cap B_{r/2}(z))\setminus B_{r/2}^+(z)\right|}{r^n}\right)^{\frac{1}{p(p-1)}}\leq c\delta^{\frac{1}{p(p-1)}}
\end{align}
for some constant $c=c(n,s,p)$. We now plug \eqref{ineq2.comprei1} and \eqref{ineq3.comprei1} into \eqref{ineq1.comprei1} to get
\begin{align*}
    \left( r^{-n}\int_{ (\Omega\cap B_{r/2}(z))\setminus B_{r/2}^+(z)}|(w-v)/d_\Omega^s|^{{p-1}}\,dx\right)^{1/(p-1)}\leq c\delta^{\frac{1}{p(p-1)}}\lambda
\end{align*}
for some constant $c=c(n,s,p,\sigma)$, which together with \eqref{ineq0.comprei1} leads to
\begin{align*}
    \left( \dashint_{ \Omega\cap B_{r/2}(z)}|(w-v)/d_\Omega^s|^{{p-1}}\,dx\right)^{\frac1{p-1}}&\leq c\left( \dashint_{ B^+_{r/2}(z)}|w/d_\Omega^s-v/d_\Omega^s|^{{p-1}}\,dx\right)^{\frac1{p-1}}\\
    &\quad+c\left( \frac1{r^{n}}\int_{ (\Omega\cap B_{r/2}(z))\setminus B_{r/2}^+(z)}|(w-v)/d_\Omega^s|^{{p-1}}\,dx\right)^{\frac1{p-1}}\\
    &\leq c(\epsilon_0+\delta^{\frac{1}{p(p-1)}})\lambda,
\end{align*}
where $c=c(n,s,p)$. Therefore, by choosing $\varepsilon_0$ in a suitable way depending on $\varepsilon$, we get
\begin{equation*}
    \left( \dashint_{ \Omega\cap B_{r/2}(z)}|w/d_\Omega^s-v/d_\Omega^s|^{{p-1}}\,dx\right)^{1/(p-1)}\leq \epsilon\lambda,
\end{equation*}
for a sufficiently small $\delta=\delta(n,s,p,\sigma,\varepsilon)>0$. Moreover, by Lemma \ref{lem.dsest} and the fact that $d_{r}(x)\coloneqq \mathrm{dist}(x,B^-_{r}(z))\leq cd_\Omega(x)$ for any $x\in B^+_{r/2}(z)$, we have 
\begin{align*}
    \|v/d_\Omega^s\|_{L^\infty(B^+_{r/2}(z))}\leq \|v/d_{r}^s\|_{L^\infty(B^+_{r/2}(z))}\leq c\widetilde{E}(v/r^s;B_{r}(z))+c\|r^sg\|^{1/(p-1)}_{L^\infty(B^+_{r}(z))}
\end{align*}
for some constant $c=c(n,s,p)$. Since $d_{r}(x)\leq 2r$ for any $x\in B^+_{r}(z)$, we derive 
\begin{align*}
     \|v/d_\Omega^s\|_{L^\infty(B^+_{r/2}(z))}\leq c\mathcal{E}_{B^+_{2r}(z)}(v;B_{r}(z))+c\|r^sg\|^{1/(p-1)}_{L^\infty(B^+_{r}(z))}\leq c\lambda,
\end{align*}
where we have also used \eqref{eq.est.v.lambda}. 
In addition, since $v\equiv 0$ in $B_{r/2}^{-}(z)$, we have 
\begin{align*}
    \|v/d_{\Omega}^s\|_{L^\infty(\Omega\cap B_{r/2}^+(z))}\leq c\lambda
\end{align*}
for some constant $c=c(n,s,p)$, which completes the proof.
\end{proof}
\section{Boundary integrability of $u/d^s$ in nonsmooth domains}
In this section, based on the comparison estimates given in the previous section, we prove Theorem \ref{thm.lp}. 
Let $u\in W^{s,p}(B_{3/2})\cap L^{p-1}_{sp}(\bbR^n)$ be a weak solution to
 \begin{equation*}
\left\{
\begin{alignedat}{3}
(-\Delta_p)^s{u}&= f&&\qquad \mbox{in  $\Omega \cap B_{3/2}$}, \\
u&=0&&\qquad  \mbox{in $B_{3/2}\setminus \Omega$},
\end{alignedat} \right.
\end{equation*}
where $\Omega$ is $(\delta,2)$-Reifenberg flat with $\delta<1/2$. We now set 
    \begin{equation}\label{defn.lambda1}
        \lambda_1\coloneqq \mathcal{E}_\Omega(u;B_{3/2})+\left(\dashint_{\Omega\cap B_{3/2}}|f|^{p_*}\,dx\right)^{\frac1{p_*(p-1)}}.
    \end{equation}
    Next, we recall maximal and fractional maximal functions. Let us write
\begin{align*}
    M^\Omega_{r}(v)(z)\coloneq \sup_{0<\rho \le r}\dashint_{\Omega\cap B_\rho(z)}|v|\,dx\quad\text{and}\quad M^\Omega_{\beta,r}(v)(z)\coloneq \sup_{0<\rho \le r}\rho^{\beta}\dashint_{\Omega\cap B_\rho(z)}|v|\,dx,
\end{align*}
where $\beta \in [0,n)$, $r\in(0,1/2]$ and $z\in B_{1/2}$.
By \cite[(2.4)]{ByuKimKum25} and \eqref{cond.meas} with $r=1/2$, we have 
\begin{align}\label{ineq.maxfrac}
    \|M^\Omega_{\beta,r}(v)\|_{L^{\frac{np}{n-p\beta}}(B_{r/2}(z))}\leq c\|f\|_{L^{p}(B_r(z))}
\end{align}
for some constant $c=c(n,s,p,\beta)$.

    For any $\lambda\geq \lambda_1$ and $N\geq1$, let us define
    \begin{equation*}
        \mathcal{C}\coloneqq \left\{z\in \Omega\cap B_{1/4}:M_{2^{-6}}^{\Omega}(|u/d_\Omega^s|^{{p-1}})(z)>(N\lambda)^{{p-1}}\right\}
    \end{equation*}
    and
    \begin{align*}
        \mathcal{D}&\coloneqq\left \{z\in \Omega\cap B_{1/4}\,:\,M_{2^{-6}}^{\Omega}(|u/d_\Omega^s|^{p-1})(z)>\lambda^{p-1}\right\}\\
        &\qquad\cup \left\{z\in \Omega\cap B_{1/4}\,:\,M^{\Omega}_{sp_*,2^{-6}}(|f|^{p_*})(z)>(\delta\lambda)^{{p_*}{(p-1)}}\right\}.
    \end{align*}
    With these notations, we start with the following measure density lemma.
    \begin{lemma}\label{lem.measep}
        Let $\lambda \ge \lambda_1$. 
        For any $\epsilon>0$, there are  constants $N=N(n,s,p)\geq1$ and $\delta=\delta(n,s,p,\epsilon)\in(0,1)$ such that if $\Omega$ is $(\delta,64\rho)$-Reifenberg flat and
    \begin{align*}
        \Omega\cap B_{\rho/8}(z)\not\subset \mathcal{D},
    \end{align*}
    where $z\in \Omega\cap B_{1/4}$ and $\rho\in(0,2^{-20}]$, then
    \begin{equation}\label{goal0.measep}
        |\mathcal{C}\cap B_{\rho/8}(z)|\leq \epsilon|B_{\rho/8}|.
    \end{equation}
    \end{lemma}
    \begin{proof}
    Assume that $\Omega \cap B_{\rho/8}(z) \not\subset \mathcal{D}$, then there exists a point $x_0\in \Omega\cap B_{\rho/8}(z)$ such that 
    \begin{align}\label{ineq0.measep}
        \sup_{0<r\leq 2^{-6}}\dashint_{\Omega\cap B_r(x_0)}|u/d_{\Omega}^s|^{p-1}\,dx\leq \lambda^{p-1}\quad\text{and}\quad \sup_{0<r\leq 2^{-6}}\dashint_{\Omega\cap B_r(x_0)}r^{sp_*}|f|^{p_*}\,dx\leq (\delta\lambda)^{{p_*}{(p-1)}}.
    \end{align}
    With $y\in \overline{\Omega}$ being any point with $|y-x_0|\leq 5\rho$, we first show
    \begin{align}\label{ineq00.measep}
        \mathcal{E}_\Omega(u;B_{2^6\rho}(y))\leq c\lambda
    \end{align}
    for some constant $c=c(n,s,p)$. 
    To this end, we observe
    \begin{align*}
         \mathcal{E}_{\Omega}(u;B_{2^6\rho}(y)) &\leq \left(\dashint_{\Omega\cap B_{2^{6}\rho}(y)}|u/d_{\Omega}^s|^{p-1}\,dx\right)^{1/(p-1)}\\
         &\quad+c\max\{d_{\Omega}(y),2^6\rho\}^{-s}\left(\sum_{i=0}^{l}2^{-{sp}i}\dashint_{ B_{2^{6+i}\rho}(y)}|u|^{p-1}\,dx\right)^{1/(p-1)}\\
         &\quad+c2^{-\frac{sp}{p-1}l}\max\{d_{\Omega}(y),2^6\rho\}^{-s}\mathrm{Tail}(u;B_{2^{6+l}\rho}(y))\\
         &\leq \left(\dashint_{\Omega\cap B_{2^{6}\rho}(y)}|u/d_{\Omega}^s|^{p-1}\,dx\right)^{\frac1{p-1}}+c\left(\sum_{i=0}^{l}2^{-{s}i}\dashint_{\Omega\cap B_{2^{6+i}\rho}(y)}|u/d_{\Omega}^s|^{p-1}\,dx\right)^{1/(p-1)}\\
         &\quad+c2^{-\frac{sp}{p-1}l}\max\{d_{\Omega}(y),2^6\rho\}^{-s}\mathrm{Tail}(u;B_{2^{6+l}\rho}(y)) \\
         & \eqqcolon J_{1} + J_{2} + J_{3},
    \end{align*}
where $l$ is the positive integer satisfying
\begin{equation}\label{rho.measep}
    2^{-9}<2^{6+l}\rho\leq 2^{-8}.
\end{equation}
Note that in the above estimate, we have also used the facts that $u\equiv 0$ in $B_{2^{6+i}\rho}(y)\setminus \Omega$ and that
\begin{align*}
    d_\Omega(x)\leq  c2^i\max\{d_\Omega(y),2^6\rho\}\quad\text{for any }x\in B_{2^{6+i}\rho}(y).
\end{align*}
Since
\begin{align*}
    B_{2^{4+i}\rho}(x_0)\subset B_{2^{5+i}\rho}(y)\subset B_{2^{6+i}\rho}(x_0)\quad\text{and}\quad B_{2^{7+l}\rho}(x_0)\subset B_{2^{-7}}(x_0),
\end{align*}
we further have 
\begin{equation}\label{ineq10.measep}
\begin{aligned}
    J_1+J_2&\leq c\left(\dashint_{\Omega\cap B_{2^{7}\rho}(x_0)}|u/d_{\Omega}^s|^{p-1}\,dx\right)^{\frac1{p-1}}+c\left(\sum_{i=0}^{l}2^{-{s}i}\dashint_{\Omega\cap B_{2^{7+i}\rho}(x_0)}|u/d_{\Omega}^s|^{p-1}\,dx\right)^{\frac1{p-1}}\\
    &\leq c\lambda+c\left(\sum_{i=0}^{l}2^{-{s}i}\lambda^{p-1}\right)^{\frac{1}{p-1}}\\
    &\leq c\lambda
\end{aligned}
\end{equation}
for some constant $c=c(n,s,p)$, where we have used H\"older's inequality and \eqref{ineq0.measep}.
We next observe 
\begin{align}\label{ineq1.measep}
   \rho^{sp/(p-1)}\max\{d_{\Omega}(y),2^6\rho\}^{-s}\leq  c\max\{d_{\Omega}(0),3/2\}^{-s}
\end{align}
for some constant $c=c(n,s)$. To show this, first we note $d_{\Omega}(0)\leq d_{\Omega}(y)+3$, as 
\begin{align*}
    d_{\Omega}(0)=|\overline{0}|\leq |\overline{y}|\leq |y-\overline{y}|+|y|\leq |y-\overline{y}|+|y-x_0|+|x_0-z|+|z|\leq d_{\Omega}(y)+3,
\end{align*}
where $\overline{0},\overline{y}\in\bbR^n\setminus \Omega$ are such that $d_{\Omega}(0) = |\overline{0}|$ and $d_{\Omega}(y)=|y-\overline{y}|$, respectively. Thus, if $d_{\Omega}(y)\leq 2^{6}\rho$, then we have 
\begin{align*}
    \rho^{sp/(p-1)}\max\{d_{\Omega}(y),2^6\rho\}^{-s}\leq c\rho^{s/(p-1)} \leq  c\max\{d_{\Omega}(0),3/2\}^{-s}.
\end{align*}
In addition, if $d_{\Omega}(y)>2^{6}\rho$, then we get
\begin{align*}
    \rho^{sp/(p-1)}\max\{d_{\Omega}(y),2^6\rho\}^{-s}\leq \rho^{sp/(p-1)}d_{\Omega}(y)^{-s}\leq \frac{c\rho^{s/(p-1)}}{(d_{\Omega}(y)+3)^{s}}\leq c\max\{d_{\Omega}(0),3/2\}^{-s},
\end{align*}
which verifies \eqref{ineq1.measep}.
On the other hand, we estimate $J_3$ as 
\begin{align*}
    J_3&\leq c2^{-\frac{sp}{p-1}l}\max\{d_{\Omega}(y),32\rho\}^{-s}\left[\left(\dashint_{B_{3/2}\cap \Omega}|u|^{p-1}\,dx\right)^{1/(p-1)}+\mathrm{Tail}(u;B_{3/2})\right]\\
    &\leq c\left(\dashint_{B_{3/2}\cap \Omega}|u/d_\Omega^s|^{p-1}\,dx\right)^{1/(p-1)}+c\max\{d_{\Omega}(0),3/2\}^{-s}\mathrm{Tail}(u;B_{3/2})\\
    &\leq c\lambda
\end{align*}
for some $c=c(n,s,p)$, where we have used H\"older's inequality, \eqref{rho.measep}, \eqref{ineq1.measep},  \eqref{defn.lambda1} and the fact that 
\[ d_\Omega(x)\leq 2\max\{d_{\Omega}(0),3/2\} \quad\text{for any }x\in B_{3/2}\cap \Omega. \]
Combining the above estimate and \eqref{ineq10.measep}, we have \eqref{ineq00.measep}.

    Having \eqref{ineq00.measep} at hand, we now prove \eqref{goal0.measep} by taking $\delta=\delta(n,s,p,\epsilon) \in (0,1)$ sufficiently small and $N=N(n,s,p) \ge 1$ sufficiently large. We distinguish two cases, depending on the location of the point $z$.
 
    \textbf{Step 1: The interior case.} We first assume that $d_\Omega(z)=\mathrm{dist}(z,\bbR^n\setminus\Omega)>2\rho$. In this case, we observe $B_{3\rho/2}(z)\Subset \Omega$. Let $w\in W^{s,p}(B_{3\rho/2}(z))\cap L^{p-1}_{sp}(\bbR^n)$ be a weak solution to 
    \begin{equation*}
\left\{
\begin{alignedat}{3}
(-\Delta_p)^s{w}&= 0&&\qquad \mbox{in  $\Omega \cap B_{\rho}(z)$}, \\
w&=u&&\qquad  \mbox{in $\bbR^n\setminus (B_{\rho}(z)\cap \Omega)=\bbR^n\setminus B_\rho(z)$}.
\end{alignedat} \right.
\end{equation*}
Note that the existence of $w$ follows from \cite[Proposition 2.12]{BraLinSch18}.  We observe
\begin{equation}\label{ineq2.measep}
            d_\Omega(z)/4\leq d_\Omega(x)\leq 4d_\Omega(z)\quad\text{for any }x\in \Omega\cap B_{5\rho/4}(z)=B_{5\rho/4}(z),
        \end{equation} which follows from the following:
        \begin{align*}
            d_\Omega(x)=|x-\overline{x}|\leq |x-\overline{z}|\leq |x-z|+|z-\overline{z}|\leq 5\rho/4 +d_{\Omega}(z)<4d_{\Omega}(z)
        \end{align*}
        and
        \begin{align*}
            d_\Omega(x)=|x-\overline{x}|\geq |z-\overline{x}|-|x-z|\geq d_{\Omega}(z)-5\rho/4>d_{\Omega}(z)/4,
        \end{align*}
        where $\overline{x},\overline{z}\in \bbR^n\setminus \Omega$ are such that $d_{\Omega}(x) = |x-\overline{x}|$ and $d_{\Omega}(z)=|z-\overline{z}|$, respectively. 
       By \eqref{ineq4.comp0} and \eqref{ineq3.comp0}, we have
        \begin{equation*}
        \begin{aligned}
            &\int_{B_{9\rho/8}(z)}\dashint_{B_{9\rho/8}(z)}|d_s(u-w)|^p\frac{\,dx\,dy}{|x-y|^n}\\
            &\leq c\left[{\mbox{\Large$\chi$}}_{\{p<2\}}\widetilde{E}(u/\rho^s;B_{5\rho/4}(z))+\left(\dashint_{\Omega\cap B_{5\rho/4}(z)}|\rho^sf|^{p_*}\,dx \right)^{\frac1{p_*(p-1)}}\right]^{\frac{p(2-p)}{2}}\\
            &\qquad\times \left(\dashint_{\Omega\cap B_{5\rho/4}(z)}|\rho^sf|^{p_*}\,dx \right)^{\frac{p^2}{2p_*(p-1)}}
        \end{aligned}
        \end{equation*}
for some constant $c=c(n,s,p)$. By H\"older's inequality and \cite[Lemma 4.7]{Coz17} along with the fact that $u-w\equiv 0$ in $B_{9\rho/8}(z)\setminus B_\rho(z)$, we obtain 
\begin{equation}\label{ineq31.measep}
        \begin{aligned}
            & \left(\dashint_{B_{\rho}(z)}|(u-w)/\rho^s|^{p-1}\,dx\right)^{1/(p-1)} \\
            &\leq c\left[{\mbox{\Large$\chi$}}_{\{p<2\}}\widetilde{E}\left(u/\rho^s;B_{5\rho/4}(z)\right)+\left(\dashint_{\Omega\cap B_{3\rho/2}(x_0)}|\rho^sf|^{p_*}\,dx \right)^{\frac1{p_*(p-1)}}\right]^{\frac{2-p}{2}}\\
            &\qquad\times \left(\dashint_{\Omega\cap B_{3\rho/2}(x_0)}|\rho^sf|^{p_*}\,dx \right)^{\frac{p}{2p_*(p-1)}},
        \end{aligned}
        \end{equation}
        where we have also used the fact that $B_{5\rho/4}(z)\subset B_{3\rho/2}(x_0)$, as $|x_0-z|\leq \rho/8$.
Multiplying both sides of \eqref{ineq31.measep} by the constant $\rho^s/d^s_{\Omega}(z)$, we get
\begin{equation*}
        \begin{aligned}
            &\left(\dashint_{B_{\rho}(z)}|(u-w)/d_\Omega^s|^{p-1}\,dx\right)^{1/(p-1)}\\
            &\leq c\left[{\mbox{\Large$\chi$}}_{\{p<2\}}\mathcal{E}(u;B_{5\rho/4}(z))+\left(\dashint_{\Omega\cap B_{3\rho/2}(x_0)}|\rho^sf|^{p_*}\,dx \right)^{\frac1{p_*(p-1)}}\right]^{\frac{2-p}{2}}\\
            &\qquad\times \left(\dashint_{\Omega\cap B_{3\rho/2}(x_0)}|\rho^sf|^{p_*}\,dx \right)^{\frac{p}{2p_*(p-1)}},
        \end{aligned}
        \end{equation*}
        where we have also used \eqref{ineq2.measep} and the fact that 
        \begin{align*}
            \rho^s/d^s_{\Omega}(z)\leq 1.
        \end{align*}
We now use \eqref{ineq0.measep}, \eqref{ineq00.measep} with $y=z$ and \eqref{ineq.ftnl} to see that 
\begin{align}\label{ineq3.measep}
    \left(\dashint_{B_{\rho}(z)}|(u-w)/d_\Omega^s|^{p-1}\,dx\right)^{\frac1{p-1}}\leq c\left({\mbox{\Large$\chi$}}_{\{p<2\}}\mathcal{E}(u;B_{5\rho/4}(z))+\delta\lambda \right)^{(2-p)/2} \left(\delta\lambda \right)^{p/2}\leq c\delta^{p/2}\lambda 
\end{align}
for some constant $c=c(n,s,p)$. In addition, by Lemma \ref{lem.bdd} together with \eqref{ineq2.measep}, we derive
\begin{align*}
    \|w/d_\Omega^s\|_{L^\infty(B_{\rho/2}(z))}\leq cd^{-s}_\Omega(z)\|w\|_{L^\infty(B_{\rho/2}(z))}&\leq cd^{-s}_\Omega(z)\widetilde{E}(w;B_\rho(z)).
\end{align*}
for some constant $c=c(n,s,p)$. By \eqref{ineq2.measep}, \eqref{ineq00.measep} with $y=z$, \eqref{ineq.ftnl} and \eqref{ineq3.measep}, we further estimate
\begin{align*}
   \|w/d_\Omega^s\|_{L^\infty(B_{\rho/2}(z))}\leq c\mathcal{E}(w;B_{\rho}(z))\leq  c\left[\mathcal{E}(w-u;B_\rho(z))+\mathcal{E}(u;B_\rho(z))\right]\leq c\lambda
\end{align*}
for some constant $c=c(n,s,p)$.
Summarizing, we have obtained 
\begin{align*}
    \left(\dashint_{B_\rho(z)\cap\Omega}|(u-w)/d_\Omega^s|^{p-1}\,dx\right)^{1/(p-1)}\leq c\delta^{p/2}\lambda\quad\text{and} \quad\|w/d_\Omega^s\|_{L^\infty(B_{\rho/2}(z)\cap\Omega)}\leq c\lambda.
\end{align*}
At this stage, we proceed as in the proof of \cite[(4.26) in Lemma 4.4]{ByuKimKum25}, with $|(u-v)/d^s|$ replaced by $|(u-w)/d_\Omega^s|^{p-1}$, in order to deduce that if $N\geq N_1$ for a sufficiently large constant $N_1=N_1(n,s,p) \ge 1$, then
\begin{align}\label{ineq4.measep}
    \mathcal{C}\cap B_{\rho/8}(z)\subset \{x\in B_{\rho/8}(z)\cap \Omega\,:\, M^{\Omega}_{\rho/16}(|(u-w)/d_\Omega^s|^{p-1})(z)>\lambda^{p-1}\}.
\end{align}
Therefore, in light of weak 1-1 estimates, we have 
\begin{align}\label{goal1.measep}
    |\mathcal{C}\cap B_{\rho/8}(z)|\leq  \lambda^{-p-1}\dashint_{B_\rho(z)\cap\Omega}|(u-w)/d_\Omega^s|^{p-1}\,dx\leq c\delta^{(p-1)p/2}\leq \epsilon
\end{align}
by taking $\delta=\delta(n,s,p,\epsilon)>0$ sufficiently small.

\textbf{Step 2: The boundary case.} We next assume that $d_{\Omega}(z) = \mathrm{dist}(z,\bbR^n\setminus \Omega)\leq 2\rho$. 
        Then there is a point $z_0 = (z'_{0},z_{0,n})\in (\bbR^n\setminus \Omega)\cap B_{3\rho}(x_0)$ such that $|z-z_0|=d_\Omega(z)$. 
        In addition, due to Definition \ref{defn.rei} and Lemma \ref{lem.rei} along with the fact that $\Omega$ is $(\delta,64\rho)$-Reifenberg flat, we may assume 
        \begin{align*}
            B_{64\rho}(z_0)\cap\{x_n>z_{0,n}+64\rho\delta\}\subset \Omega\cap B_{64\rho}(z_0)\subset B_{64\rho}(z_0)\cap\{x_n>z_{0,n}-64\rho \delta\}.
        \end{align*}
        Let us write 
        \begin{equation*}
            \overline{z_0}\coloneqq(z_0',z_{0,n}+64\rho\delta)
        \end{equation*}
        to see that $d_\Omega(\overline{z_0})\leq \rho$, 
        \begin{align*}
            B^+_{32\rho}(\overline{z_0}) \subset \Omega\cap B_{32\rho}(\overline{z_0})\subset B_{32\rho}(\overline{z_0})\cap\{x_n>\overline{z_{0,n}}-32\rho \times 4\delta\}
        \end{align*}
        with $\delta < 2^{-10}$, and
        \begin{align*}
            B_\rho(z)\subset B_{4\rho}(\overline{z}).
        \end{align*}
        Note that since $\Omega$ is $(\delta,64\rho)$-Reifenberg flat, it is also $(4\delta,32\rho)$-Reifenberg flat.
Let $w\in W^{s,p}(\bbR^n)$ be a weak solution to 
\begin{equation*}
\left\{
\begin{alignedat}{3}
(-\Delta_p)^s{w}&= 0&&\qquad \mbox{in  $\Omega \cap B_{32\rho}(\overline{z_0})$}, \\
w&=u&&\qquad  \mbox{in $\bbR^n\setminus (B_{32\rho}(\overline{z_0})\cap \Omega)$}.
\end{alignedat} \right.
\end{equation*}
Now we are going to prove 
\begin{align*}
    \mathcal{E}(w;B_{32\rho}(\overline{z_0}))\leq c\lambda,
\end{align*}
where $c=c(n,s,p)$.
By Lemma \ref{lem.comp0} together with the fact that  
\begin{align*}
    \widetilde{E}(u/\rho^s;B_{40\rho}(\overline{z_0}))\leq \mathcal{E}(u;B_{40\rho}(\overline{z_0}))\quad \text{and}\quad B_{40\rho}(\overline{z_0})\subset B_{45\rho}(x_0),
\end{align*} 
we get 
\begin{equation*}
\begin{aligned}
   & \left(\dashint_{\Omega\cap B_{32}(\overline{z_0})}|(u-w)/d_\Omega^s|^{p-1}\,dx\right)^{1/(p-1)}\\
    &\leq  c\left[{\mbox{\Large$\chi$}}_{\{p<2\}}\mathcal{E}(u;B_{40\rho}(\overline{z_0}))+\left(\dashint_{\Omega\cap B_{45}(x_0)}|\rho^sf|^{p_*}\,dx\right)^{\frac{1}{p_*(p-1)}}\right]^{\frac{2-p}{2}}\\
        &\quad\times \left(\dashint_{\Omega\cap B_{45}(x_0)}|\rho^sf|^{p_*}\,dx\right)^{\frac{p}{2p_*(p-1)}},
\end{aligned}
\end{equation*}
where $c=c(n,s,p)$. Since $|\overline{z_0}-x_0|\leq |z_0-\overline{z}_0|+|x_0-z_0|\leq 5\rho$, we use \eqref{ineq00.measep} with $y=z_0$, \eqref{ineq.ftnl} and \eqref{ineq0.measep} to get 
\begin{align}\label{ineq5.measep}
     \left(\dashint_{\Omega\cap B_{32}(\overline{z_0})}|(u-w)/d_\Omega^s|^{p-1}\,dx\right)^{1/(p-1)}\leq c\delta^{p/2}\lambda,
\end{align}
where $c=c(n,s,p)$. Using this, the fact that $w\equiv u$ in $\bbR^n\setminus B_{32\rho}(\overline{z_0})$ and \eqref{ineq00.measep} with $y=\overline{z_0}$, we obtain
\begin{align*}
    \mathcal{E}(w;B_{32\rho}(\overline{z_0}))\leq \mathcal{E}(u;B_{32\rho}(\overline{z_0}))+\mathcal{E}(u-w;B_{32\rho}(\overline{z_0}))\leq c\lambda
\end{align*}
for some constant $c=c(n,s,p)$.
For $\epsilon_1>0$, by applying Lemma \ref{lem.comprei1}, there is a function $v\in W^{s,p}(B_{16\rho}(z_0))\cap L^{p-1}_{sp}(\bbR^n)$ such that 
\begin{equation*}
\left(\dashint_{B_{4\rho}(\overline{z_0})\cap \Omega}|(w-v)/d_\Omega^s|^{p-1}\,dx\right)^{1/(p-1)}\leq c\epsilon_1\lambda\quad\text{and}\quad \|v/d_\Omega^s\|_{L^\infty(B_{4\rho}(\overline{z_0})\cap \Omega)}\leq c\lambda
\end{equation*}
for some constant $c=c(n,s,p)$, whenever $\delta=\delta(n,s,p,\epsilon_1)>0$ is sufficiently small. Using this in combination with \eqref{ineq5.measep}, we get
\begin{align*}
    \left(\dashint_{B_{4\rho}(\overline{z_0})\cap\Omega}|(u-v)/d_\Omega^s|^{p-1}\,dx\right)^{1/(p-1)}\leq c(\delta^{p/2}+\epsilon_1)\lambda\quad\text{and}\quad \|v/d_\Omega^s\|_{L^\infty(B_{4\rho}(\overline{z_0})\cap \Omega)}\leq c\lambda
\end{align*}
for some constant $c=c(n,s,p)$. Since $B_{\rho}(z)\subset B_{4\rho}(\overline{z_0})$, we have 
\begin{align*}
    \left(\dashint_{B_{\rho}(z)\cap\Omega}|(u-v)/d_\Omega^s|^{p-1}\,dx\right)^{1/(p-1)}\leq c(\delta^{p/2}+\epsilon_1)\lambda\quad\text{and}\quad \|v/d_\Omega^s\|_{L^\infty(B_{\rho}(z)\cap \Omega)}\leq c\lambda.
\end{align*}
Similarly as in the proof of \eqref{ineq4.measep}, we have that if $N \ge N_1$ for a sufficiently large constant $N_{1} = N_1(n,s,p) \ge 1$, then 
\begin{align*}
    \mathcal{C}\cap B_{\rho/8}(z)\subset \{x\in B_{\rho/8}(z)\cap \Omega\,:\, M^{\Omega}_{\rho/16}(|(u-w)/d_\Omega^s|^{p-1})(z)>\lambda^{p-1}\}.
\end{align*} 
Therefore, using weak 1-1 estimates, we get 
\begin{align}\label{goal2.measep}
    |\mathcal{C}\cap B_{\rho/8}(z)|\leq c(\delta^{(p-1)p/2}+\epsilon_1^{p-1})|B_{\rho/8}|\leq \epsilon|B_{\rho/8}|
\end{align}
by taking $\varepsilon_{1} = \varepsilon_{1}(n,s,p,\varepsilon) > 0$ and then $\delta=\delta(n,s,p,\epsilon)>0$ sufficiently small. 
Combining \eqref{goal1.measep} and \eqref{goal2.measep}, we conclude with \eqref{goal0.measep}.
\end{proof}
In light of Lemma \ref{lem.measep}, we have the following:
\begin{corollary}\label{cor.measep}
    Given $\epsilon>0$, there exists a small constant $\delta=\delta(n,s,p,\epsilon)>0$ such that if $\Omega$ is $(\delta,1)$-Reifenberg flat, then for all $\lambda\geq\lambda_1$ and nonnegative integer $k$, 
    \begin{align*}
        &\left|\left\{z\in\Omega\cap B_{1/4}\,:\,M^{\Omega}_{2^{-6}}(|u/d_\Omega^s|^{p-1})(z)>(N^k\lambda)^{p-1}\right\}\right|\\
        &\leq (c\epsilon)^k\left|\left\{z\in\Omega\cap B_{1/4}\,:\,M^{\Omega}_{2^{-6}}(|u_\Omega/d^s|^{p-1})(z)>\lambda^{p-1}\right\}\right|\\
        &\quad+\sum_{i=1}^k(c\epsilon)^i\left|\left\{z\in\Omega\cap B_{1/4}\,:\,M^{\Omega}_{sp_*,2^{-6}}(|f|^{p_*})(z)>(N^{k-i}\delta\lambda)^{{p_*}{(p-1)}}\right\}\right|
    \end{align*}
    holds, where $c=c(n,s,p)$ and the constant $N$ is determined in Lemma \ref{lem.measep}.
\end{corollary}
\begin{proof}
    Since $\Omega$ is $(\delta,1)$-Reifenberg flat, it is also $(\delta,64\rho)$-Reifenberg flat for any $\rho\leq 2^{-20}$. Once we have Lemma \ref{lem.measep}, we can follow the proof of \cite[Corollary 4.6]{ByuKimKum25}, with $|u/d^s|$ and $|f|^{2_*}$ replaced by $|u/d_\Omega^s|^{p-1}$ and $|f|^{p_*}$, respectively, in order to get the desired estimate when $\delta=\delta(n,s,p,\epsilon)>0$ is sufficiently small.
\end{proof}
With the above corollary at hand, we are now ready to prove Theorem \ref{thm.lp}.
\begin{proof}[Proof of Theorem \ref{thm.lp}.]
    By Lemma \ref{lem.scale} and Lemma \ref{lem.rei}, we may assume $R_0=1$. Let $\epsilon>0$ be a free parameter which will be determined later. We set
    \begin{equation*}
        \lambda\coloneqq \mathcal{E}_\Omega(u;B_{3/2})+\left(\dashint_{\Omega\cap B_{3/2}}|f|^q\,dx\right)^{\frac{1}{q(p-1)}}\geq \lambda_1,
    \end{equation*}
    where $\lambda_1$ is given in \eqref{defn.lambda1}. Then, with the constants
    $N=N(n,s,p)\geq1$ and $\delta = \delta(n,s,p,\epsilon)\in(0,1)$ determined in Lemma \ref{lem.measep}, we proceed as in the proof of \cite[Lemma 4.7]{ByuKimKum25}, together with Corollary \ref{cor.measep}, to see that if $0<\epsilon\leq \epsilon_1$ for some small constant $\epsilon_1=\epsilon_1(n,s,p)>0$, then 
    \begin{equation}\label{ineq1.thmlp}
    \begin{aligned}
        &\sum_{i=1}^{l}(N^k\lambda)^{\frac{nq(p-1)}{n-sq}}\left|\left\{z\in\Omega\cap B_{1/4}\,:\,M^{\Omega}_{2^{-6}}(|u/d_\Omega^s|^{p-1})(z)>(N^k\lambda)^{p-1}\right\}\right|\\
        &\leq \sum_{k=1}^{l}\sum_{i=1}^k(N^{k-1}\lambda)^{\frac{nq(p-1)}{n-sq}}(c\epsilon)^i\left|\left\{z\in\Omega\cap B_{1/4}\,:\,M^{\Omega}_{sp_*,2^{-6}}(|f|^{p_*})(z)>(N^{k-i}\delta\lambda)^{p_*(p-1)}\right\}\right|\eqqcolon J
    \end{aligned}
    \end{equation}
    for some constant $c=c(n,s,p)$, whenever $\delta=\delta(n,s,p,\epsilon)>0$ is sufficiently small. By \cite[Lemma 7.3]{CacCab95} and \eqref{ineq.maxfrac}, we estimate $J$ as 
    \begin{equation}\label{ineq2.thmlp}
    \begin{aligned}
        J&\leq \sum_{i=1}^{l}\left(c\epsilon N^{\frac{nq(p-1)}{n-sq}}\right)^i\sum_{k=1}^l(N^{k-i}\lambda)^{\frac{nq(p-1)}{n-sq}}\\
        &\quad\quad\qquad\quad \times\left|\left\{z\in\Omega\cap B_{1/4}\,:\,M^{\Omega}_{sp_*,2^{-6}}(|f/(\delta\lambda)^{p-1}|^{p_*})(z)>N^{p_*(p-1)(k-i)}\right\}\right|\\
        &\leq \lambda^{\frac{nq(p-1)}{n-sq}}\sum_{i=1}^{l}\left(c\epsilon N^{\frac{nq(p-1)}{n-sq}}\right)^i\left(\left\|M^\Omega_{sp_*,2^{-6}}(|f/(\delta\lambda)^{p-1}|^{p_*})\right\|_{L^{\frac{nq}{p_*(n-sq)}}(\Omega\cap B_{1/4})}+1\right)\\
        &\leq c\lambda^{\frac{nq(p-1)}{n-sq}}(\|f/(\lambda\delta)^{p-1}\|_{L^{q}(\Omega\cap B_{1/2})}+1) 
    \end{aligned}
    \end{equation}
    for some constant $c=c(n,s,p,q)$, where we have chosen $\epsilon=\epsilon(n,s,p,q)>0$ sufficiently small to get $c\epsilon N^{\frac{nq(p-1)}{n-sq}}<1/2$. Since $\epsilon=\epsilon(n,s,p,q)$ is now fixed, so is $\delta=\delta(n,s,p,q)$. We now employ \cite[Lemma 7.3]{CacCab95} once again to derive
    \begin{equation}\label{ineq3.thmlp}
    \begin{aligned}
        & \||u/(\lambda d_\Omega^s)|^{p-1}\|_{L^{\frac{nq}{(n-sq)}}(\Omega\cap B_{1/8})} \leq \|M^\Omega_{2^{-6}}(|u/(\lambda d_\Omega^s)|^{p-1})\|_{L^{\frac{nq}{(n-sq)}}(\Omega\cap B_{1/8})}\\
        &\leq c\sum_{i=1}^{\infty}N^{k(p-1)\frac{nq}{n-sq}}\left|\left\{z\in\Omega\cap B_{1/4}\,:\,M^{\Omega}_{2^{-6}}(|u/d_\Omega^s|^{p-1})(z)>(N^k\lambda)^{p-1}\right\}\right|.
    \end{aligned}
    \end{equation}
    Combining \eqref{ineq1.thmlp}, \eqref{ineq2.thmlp} and \eqref{ineq3.thmlp} yields 
    \begin{align*}
         \||u/(\lambda d_\Omega^s)|^{p-1}\|_{L^{\frac{nq}{(n-sq)}}(\Omega\cap B_{1/8})}\leq c(\|f/\lambda^{p-1}\|_{L^{q}(\Omega\cap B_{1/2})}+1) \leq c
    \end{align*}
    for some constant $c=c(n,s,p,q)$, where we have also used the fact that $\delta=\delta(n,s,p,q)$ and $\lambda^{p-1}\geq \|f\|_{L^{q}(\Omega\cap B_{1/2})}$.
    Therefore, we have 
    \begin{align*}
         \|u/d_\Omega^s\|_{L^{\frac{nq(p-1)}{n-sq}}(\Omega\cap B_{1/8})}\leq c\lambda \leq c\mathcal{E}_\Omega(u;B_1)+c\left(\dashint_{\Omega\cap B_1}|f|^q\,dx\right)^{\frac{1}{q(p-1)}}
    \end{align*}
    for some constant $c=c(n,s,p,q)$. Lastly, we show that
    \begin{align}\label{ineq4.thmlp}
        \mathcal{E}_\Omega(u;B_1)\leq c\widetilde{E}(u;B_{3/2})
    \end{align}
    for some constant $c=c(n,s,p)$. If $B_{9/8}\subset \Omega$, then $d_\Omega(x)\geq c$ and \eqref{ineq4.thmlp} directly follows. Suppose $B_{9/8}\cap (\bbR^n\setminus \Omega)\neq\emptyset$, then by Lemma \ref{lem.hardy}, we have
    \begin{align*}
        \int_{\Omega\cap B_{5/4}}|u/d_\Omega^s|^{p}\,dx\leq c[u]^p_{W^{s,p}(B_{11/8})}\leq c\left(\widetilde{E}(u;B_{3/2})^p+\|f\|_{L^{p_*}(B_{3/2}\cap \Omega)}^{p/(p-1)}\right)
    \end{align*}
    for some constant $c=c(n,s,p)$, where we have used Lemma \ref{lem.ene}. This implies \eqref{ineq4.thmlp}, which in turn leads to
    \begin{align*}
         \|u/d_\Omega^s\|_{L^{\frac{nq(p-1)}{n-sq}}(\Omega\cap B_{1/8})}\leq c\left(\widetilde{E}(u;B_{3/2})+\|f\|^{1/(p-1)}_{L^{q}(B_{3/2}\cap \Omega)}\right),
    \end{align*}
    where $c=c(n,s,p,q)$. A standard covering argument finally gives the desired estimate.
\end{proof}

\section{Sharp boundary regularity in nonsmooth domains}\label{sec.sob}
In this section, we prove higher regularity of weak solutions up to the boundary. 
We start with the following lemma, 
which shows that the size of the $r$-neighborhood of $\partial\Omega$ has a good decay.
\begin{lemma}\label{lem.distest}
    For any $\alpha\in(0,1)$, there is a sufficiently small $\delta=\delta(n,\alpha)>0$ such that if  $\Omega$ is $(\delta,1)$-Reifenberg flat, then 
    \begin{align*}
        |\{x\in B_{1}\cap \Omega\,:\,d_\Omega(x)< r\}|\leq cr^\alpha
    \end{align*}
    holds whenever $r\in(0,1]$, where $c=c(n,\alpha)$.
\end{lemma}
\begin{proof}
    It suffices to prove the lemma when $r\in(0,2^{-20}]$. Let us define 
    \begin{equation*}
        g(r)\coloneqq|\{x\in B_{1}\cap \Omega\,:\,d_\Omega(x)< r\}|\quad\text{for any }r\in(0,2^{-20}].
    \end{equation*}
    Then Lemma \ref{lem.reidist} implies
    \begin{align*}
        g(\delta r)\leq c\delta g(r)+cr
    \end{align*}
    for some constant $c=c(n)$. Dividing both sides by $(\delta r)^{\alpha}$, we get
    \begin{align*}
        \frac{g(\delta r)}{(\delta r)^\alpha}\leq c\delta^{1-\alpha}\frac{g(r)}{r^\alpha}+c\delta^{-\alpha}r^{1-\alpha}.
    \end{align*}
    By taking $\delta=\delta(n,\alpha)>0$ sufficiently small, we obtain 
    \begin{align*}
        \frac{g(\delta r)}{(\delta r)^\alpha} \leq \frac{1}{4}\frac{g(r)}{r^\alpha}+cr^{1-\alpha}
    \end{align*}
    for some constant $c=c(n,\alpha)$.
    Now let us define 
    \begin{align*}
        N_{\rho_0}\coloneqq \sup_{\rho_0 \le \rho\leq 2^{-20}\delta}\frac{g(\rho)}{\rho^\alpha}
    \end{align*}
    for some $\rho_0<2^{-20}\delta$. Then we derive
    \begin{align*}
        N_{\rho_0}\leq \frac12 N_{\rho_0}+c
    \end{align*}
    for some constant $c=c(n,\alpha)$, as $\delta=\delta(n,\alpha)$ is fixed and  $g(\rho)/\rho^{\alpha}\leq c(n,\alpha)$ for any $\rho\in[2^{-20}\delta,1]$. Thus, by taking $\rho_0\to 0$, we have
    \begin{equation}\label{ineq1.distest}
        \sup_{0< \rho\leq 2^{-20}\delta}\frac{g(\rho)}{\rho^\alpha}\leq c.
    \end{equation}
    Since the constant $\delta=\delta(n,\alpha)$ is fixed, the desired estimate directly follows from \eqref{ineq1.distest}.
\end{proof}

Now using Theorem \ref{thm.lp} and Lemma \ref{lem.distest}, we improve the estimate given in Theorem \ref{thm.lp}.
\begin{lemma}\label{lem.basic1}
    Let $u$ be a weak solution to \eqref{eq.diri}, 
where $f\in L^{n/s}(\Omega\cap B_2)$. Let us fix $\sigma\in(s,s+1/\gamma)$, where $\gamma\in[1,\infty)$. Then there is a sufficiently small $\delta=\delta(n,s,p,\sigma,\gamma)>0$ such that if $\Omega$ is $(\delta,2)$-Reifenberg flat, then we have 
\begin{align*}
    \|u/d_\Omega^\sigma\|_{L^{\gamma}(\Omega\cap B_{1/2})}\leq c\left(\widetilde{E}(u;B_{1})+\|f\|^{1/(p-1)}_{L^{n/s}(\Omega\cap B_{1})}\right)
\end{align*}
for some constant $c=c(n,s,\Lambda,\sigma, \gamma)$. 
\end{lemma}
\begin{proof}
We first prove that for any $\epsilon\in(0,1)$, 
\begin{align}\label{ineq1.basic1}
    \int_{\Omega\cap B_{1/2}}\frac{1}{d^{\epsilon}_\Omega(x)}\,dx=\int_{0}^{\infty}|\{x\in\Omega\cap B_{1/2}\,:\,1/d_\Omega^\epsilon(x)>t\}|\,dt\leq c
\end{align}
holds for some constant $c=c(n,\epsilon)$, provided that $\delta=\delta(n,\epsilon)>0$ is sufficiently small. 
By Lemma \ref{lem.distest}, for any $r\in(0,1]$, we have
\begin{align}\label{ineq0.basic1}
    |\{x\in\Omega\cap B_{1/2}\,:\,d_\Omega(x)<r\}|\leq cr^{(\epsilon+1)/2}
\end{align}
for some constant $c=c(n,\epsilon)$, whenever $\delta=\delta(n,\epsilon)$ is sufficiently small. 
We now split as follows:
\begin{align*}
    \int_{0}^{\infty}|\{x\in\Omega\cap B_{1/2}\,:\,1/d_\Omega^\epsilon(x)>t\}|\,dt&=\int_{0}^{\infty}|\{x\in\Omega\cap B_{1/2}\,:\,1/t^{1/\epsilon}>d_\Omega(x)\}|\,dt\\
    &=\int_{0}^{1}(\cdots)\,dt+\int_{1}^{\infty}(\cdots)\,dt\eqqcolon J_1+J_2.
\end{align*}
We directly have $J_1\leq |B_{1/2}|$. As for $J_2$, since $1/t^{1/\epsilon}\in(0,1]$ when $t\geq1$, \eqref{ineq0.basic1} and the fact that $\frac{(\epsilon+1)}{2\epsilon}>1$ imply
\begin{align*}
    J_2\leq c\int_{1}^\infty (1/t)^{\frac{(\epsilon+1)}{2\epsilon}}\,dt\leq c(n,\epsilon).
\end{align*}
Thus we get \eqref{ineq1.basic1}.

We now choose $\tilde{q}\coloneqq \frac{2\gamma}{1-\gamma(\sigma-s)}$. Then, using H\"older's inequality and \eqref{ineq1.basic1} with $\epsilon=\frac{\gamma \tilde{q}(\sigma-s)}{q-\gamma} \in (0,1)$, we get
\begin{equation}\label{ineq0.siggam}
\begin{aligned}
    \int_{\Omega\cap B_{1/2}}|u/d_\Omega^{\sigma}|^{\gamma}\,dx&=\int_{\Omega\cap B_{1/2}}|u/d_\Omega^{s}|^{\gamma}|1/d_\Omega^{\sigma-s}|^{\gamma}\,dx \\
    &\leq \left(\int_{\Omega\cap B_{1/2}}|u/d_\Omega^{s}|^{\tilde{q}}\,dx\right)^{\gamma/\tilde{q}}\left(\int_{\Omega\cap B_{1/2}}|1/d_\Omega^{\sigma-s}|^{\gamma \tilde{q}/(\tilde{q}-\gamma)}\,dx\right)^{(\tilde{q}-\gamma)/\tilde{q}}\\
    &\leq c\left(\int_{\Omega\cap B_{1/2}}|u/d_\Omega^{s}|^{\tilde{q}}\,dx\right)^{\gamma/\tilde{q}}.
\end{aligned}
\end{equation}
Plugging the estimate given in Theorem \ref{thm.lp} into the last inequality in \eqref{ineq0.siggam}, we get the desired estimate when  $\delta=\delta(n,s,p,\sigma,\gamma)>0$ is sufficiently small. 
\end{proof}

\subsection{Higher Sobolev and H\"older regularity}
In this subsection, we prove boundary higher Sobolev and H\"older estimates for solutions. The proof uses Lemma \ref{lem.distest}, Lemma \ref{lem.basic1}, Theorem \ref{thm.lp} and interior regularity results for fractional $p$-Laplace equations. 

First, we prove Theorem \ref{thm.hig} concerning boundary higher Sobolev regularity.
\begin{proof}[Proof of Theorem \ref{thm.hig}.]
Throughout the proof, we simply denote $d(x)\equiv d_\Omega(x)$.

Without loss of generality, we may assume that $\sigma >s$. 
We fix $\sigma_1\coloneqq (\sigma+s+1/\gamma)/2\in(\sigma,s+1/\gamma)$
and investigate the integral
\begin{align*}
    I\coloneqq \int_{B_{1/8}}\frac{|u(x+h)-u(x)|^{\gamma}}{|h|^{\sigma_1\gamma}}\,dx,
\end{align*}
where $h\in B_{1/1000}\setminus \{0\}$.
To do this, we split as follows:
\begin{align*}
    I& = \sum_{i=1}^{5}\int_{\mathcal{A}_i}\frac{|u(x+h)-u(x)|^{\gamma}}{|h|^{\sigma_1\gamma}}\,dx \eqqcolon \sum_{i=1}^5I_i,
\end{align*}
where we write
\begin{align*}
    &\mathcal{A}_1\coloneqq\{x\in B_{1/8}\cap \Omega\,:\,x+h\in B_{3/16}\cap \Omega\text{ and }5|h|\geq\max\{d(x+h),d(x)\} \},\\
    &\mathcal{A}_2\coloneqq\{x\in B_{1/8}\cap \Omega\,:\,x+h\in B_{3/16}\cap \Omega\text{ and }5|h|<\max\{d(x+h),d(x)\} \},\\
    &\mathcal{A}_3\coloneqq\{x\in B_{1/8}\cap \Omega\,:\,x+h\in B_{3/16}\setminus \Omega\},\\
    &\mathcal{A}_4\coloneqq\{x\in B_{1/8}\setminus \Omega\,:\,x+h\in B_{3/16}\cap \Omega\},\\
    &\mathcal{A}_5\coloneqq\{x\in B_{1/8}\setminus \Omega\,:\,x+h\in B_{3/16}\setminus \Omega\}.
\end{align*}

    \textbf{Step 1: Estimate of $I_1$.} We first consider the term $I_1$. We observe
    \begin{align*}
        I_1&\leq c  \left(\int_{\mathcal{A}_1}\frac{|u(x)|^\gamma}{d^{\sigma_1\gamma}(x)}{\,dx}+\int_{\mathcal{A}_1}\frac{|u(x+h)|^\gamma}{d^{\sigma_1\gamma}(x+h)}{\,dx}\right)\leq c\int_{B_{1/4}\cap \Omega}\frac{|u(x)|^\gamma}{d^{\sigma_1\gamma}(x)}{\,dx},
    \end{align*}
    where we have used the change of variables and the fact that $x+h\in B_{1/4}\cap \Omega$. Lemma \ref{lem.basic1} implies that if $\Omega$ is $(\delta,2)$-Reifenberg flat for sufficiently small $\delta=\delta(n,s,p,q,\gamma,\sigma)>0$, then
     \begin{align}\label{ineqi1.dsregmove}
         I_1\leq  c\left(\widetilde{E}(u;B_{1})+\|f\|^{1/(p-1)}_{L^q(\Omega\cap B_1)}
         \right)^{\gamma}.
     \end{align}
     holds for some constant $c=c(n,s,p,q,\gamma,\sigma)$.
     
    \textbf{Step 2: Estimate of $I_2$.} We next estimate the term $I_2$ by using both interior and boundary regularity results. 
    We may assume $d(x)\geq d(x+h)$, which implies $x+h\in B_{d(x)/4}(x)\subset B_{d(x)/2}(x)\subset\Omega$. Thus, by \cite[Theorem 1.4]{BraLinSch18} and \cite[Theorem 1.2]{GarLin24} together with the fact that $\sigma<\sigma_1<\min\{1,(sp-n/q)/(p-1)\}$, we have 
    \begin{align*}
        \frac{|u(x+h)-u(x)|}{|h|^{\sigma_1}} & \leq [u]_{C^{\sigma_1}\left(B_{{d(x)}/4}(x)\right)} \\&\leq \frac{c}{d^{\sigma_1}(x)}\left(\dashint_{B_{d(x)/2}(x)}|u|^p\,dz\right)^{1/p}+\frac{c}{d^{\sigma_1}(x)}\mathrm{Tail}\left(u;B_{d(x)/2}(x)\right)\\
        &\quad+\frac{c}{d^{\sigma_1}(x)}\left(\dashint_{B_{d(x)/2}(x)}|d^{sp}(x)f|^q \,dz\right)^{\frac1{q(p-1)}}\\
        &\eqqcolon J_1+J_2+J_3,
    \end{align*}
    where $c=c(n,s,p,\sigma_1)$. Before estimating the terms $J_1,J_2$ and $J_3$, we observe
    \begin{align}\label{ineq01.measep}
        d(z)=|z-\overline{z}|\leq |z-\overline{x}|\leq |z-x|+|x-\overline{x}|\leq (2^{k}+1)d(x)
    \end{align}
    for any $z\in B_{2^kd(x)}(x)$ with any integer $k$, where $|z-\overline{z}|=d(z)$ and $|x-\overline{x}|=d(x)$ for some $\overline{z},\overline{x}\in\bbR^n\setminus\Omega$.
        By taking 
        \begin{align}\label{ineq001.measep}
        s_0\coloneqq ({\sigma_1}-1/\gamma+s)/2<s    \quad\text{and}\quad \overline{q}\coloneqq \max\{n/(s-s_0),p\},
        \end{align}
        which implies ${\sigma_1}<s_0+1/\gamma<s+1/\gamma$.
        Thus, we have
        \begin{equation}\label{q.j1}
        \begin{aligned}
            d^{{\sigma_1}-s_0}(x)J_1\leq c\left(\dashint_{B_{d(x)/2}(x)}\left|\frac{u(z)}{d^{s_0}(x)}\right|^p\,dz\right)^{1/p}&\leq c\left(\dashint_{B_{d(x)/2}(x)}\left|\frac{u(z)}{d^{s_0}(z)}\right|^p\,dz\right)^{1/p}\\
            &\leq c\left(\dashint_{B_{d(x)/2}(x)}\left|\frac{u}{d^{s_0}}\right|^{\overline{q}}\,dz\right)^{1/{\overline{q}}}\\
            &\leq c\left(\int_{B_{d(x)/2}(x)}\left|\frac{u}{d^{s}}\right|^{\overline{q}}\,dz\right)^{1/{\overline{q}}},
        \end{aligned}
        \end{equation}
        where we have used H\"older's inequality and the fact that
        \begin{align*}
            d^{-n}(x)d^{-s_0\overline{q}}(z)\leq cd(z)^{-n-s_0\overline{q}} \leq cd^{-s\overline{q}}(z)\quad\text{for any }z\in B_{d(x)/2}(x).
        \end{align*}        
        Therefore, applying Lemma \ref{lem.basic1} to the last term in \eqref{q.j1}, we get 
        \begin{align}\label{ineq0001.measep}
            d^{{\sigma_1}-s_0}(x)J_1\leq \left(\int_{B_{1/4}\cap \Omega}\left|\frac{u}{d^{s}}\right|^{\overline{q}}\,dz\right)^{1/{\overline{q}}}\leq c\left(\widetilde{E}(u;B_{1})+\|f\|^{1/(p-1)}_{L^q(\Omega\cap B_1)}\right)
        \end{align}
        for some constant $c=c(n,s,p,q,\sigma)$, whenever $\delta=\delta(n,s,p,q,\sigma)>0$ is sufficiently small.
        We next estimate the tail term $J_2$ as 
        \begin{align*}
            d^{{\sigma_1}-s_0}(x)J_2&\leq \frac{c}{d^{s_0}(x)}\left[\left(\sum_{i=0}^{l}2^{-spi}\dashint_{B_{2^id(x)}(x)}|u|^{p-1}\,dz\right)^{1/(p-1)}+2^{-\frac{sp}{p-1}l}\mathrm{Tail}(u;B_{2^ld(x)}(x))\right],
    \end{align*}
    where $l$ is the positive integer satisfying
    \begin{align}\label{ineq.dx}
        1/4\leq 2^ld(x)<1/2.
    \end{align}
    Using the fact that $u\equiv 0$ in $B_{1}\setminus \Omega$, \eqref{cond.meas} with $R_0=2$ and \eqref{ineq01.measep}, we obtain
     \begin{align*}
         & \frac1{d^{s_0}(x)}\left(\sum_{i=0}^{l}2^{-spi}\dashint_{B_{2^id(x)}(x)}|u|^{p-1}\,dz\right)^{1/(p-1)} \\
         &\leq c\left(\sum_{i=0}^{l}2^{-spi}\dashint_{B_{2^id(x)}(x)\cap \Omega}\left|\frac{u(z)}{d^{s_0}(x)}\right|^{p-1}\,dz\right)^{1/(p-1)}\\
         &\leq c\left(\sum_{i=0}^{l}2^{(-sp+(p-1)s_0)i}\dashint_{B_{2^id(x)}(x)\cap \Omega}\left|\frac{u}{d^{s_0}}\right|^{p-1}\,dz\right)^{1/(p-1)} \eqqcolon J_{2,1}
     \end{align*}
     for some constant $c=c(n,s,p)$. We now further estimate 
     \begin{align*}
         J_{2,1}&\leq c\left[\sum_{i=0}^{l}2^{(-sp+(p-1)s_0)i}\left(\dashint_{B_{2^id(x)}(x)\cap \Omega}\left|\frac{u}{d^{s_0}}\right|^{\overline{q}}\,dz\right)^{(p-1)/\overline{q}}\right]^{1/(p-1)}\\
         &\leq c\left[\sum_{i=0}^{l}2^{(-sp+(p-1)s_0)i}\left(\int_{B_{2^id(x)}(x)\cap \Omega}\left|\frac{u}{d^s}\right|^{\overline{q}}\,dz\right)^{(p-1)/\overline{q}}\right]^{1/(p-1)}\\
         &\leq c\left[\sum_{i=0}^{l}2^{(-sp+(p-1)s_0)i}\left(\int_{B_{1/2}\cap \Omega}\left|\frac{u}{d^s}\right|^{\overline{q}}\,dz\right)^{(p-1)/\overline{q}}\right]^{1/(p-1)}
     \end{align*}
     for some constant $c=c(n,s,p)$, where we have used H\"older's inequality and the fact that 
     \begin{align*}
         d^{ -\overline{q}s_0}(z)|B_{2^id(x)}(x)\cap \Omega|^{-1}\leq cd^{ -\overline{q}s_0}(z)(2^id(x))^{-n}\leq cd^{-\overline{q}s_0-n}(z) \leq cd^{-s\overline{q}}(z)
     \end{align*}
     for any $z\in B_{2^id(x)}(x)\cap \Omega$,
    which follows from \eqref{cond.meas}, \eqref{ineq001.measep} and \eqref{ineq01.measep}. Thus, by Lemma \ref{lem.basic1}, we get
    \begin{align}\label{ineq.last.measep}
        \frac1{d^{s_0}(x)}\left(\sum_{i=0}^{l}2^{-spi}\dashint_{B_{2^id(x)}(x)}|u|^{p-1}\,dz\right)^{1/(p-1)}\leq c\left(\widetilde{E}(u;B_{1})+\|f\|^{1/(p-1)}_{L^q(\Omega\cap B_1)}\right),
    \end{align}
    provided that $\delta=\delta(n,s,p,q,\gamma,\sigma)>0$ is sufficiently small.
    On the other hand, using \eqref{ineq.dx}, we get
    \begin{align*}
        \frac1{d^{s_0}(x)}2^{-\frac{sp}{p-1}l}\mathrm{Tail}(u;B_{2^ld(x)}(x))\leq c\widetilde{E}(u;B_1).
    \end{align*}
    Combining this, \eqref{ineq.last.measep} and \eqref{ineq0001.measep}, we get 
    \begin{align*}
        J_1+J_2\leq \frac{c}{d^{\sigma_1-s_0}(x)}\left(\widetilde{E}(u;B_{1})+\|f\|^{1/(p-1)}_{L^q(\Omega\cap B_1)}\right).
    \end{align*}
    In addition, we have 
    \begin{align*}
        J_3\leq c(d(x))^{-\sigma_1+\frac{sp}{p-1}-\frac{n}{q(p-1)}}\|f\|^{1/(p-1)}_{L^q(\Omega\cap B_1)}\leq cd^{s_0 - \sigma_1}(x)\|f\|^{1/(p-1)}_{L^q(\Omega\cap B_1)}
    \end{align*}
    for some constant $c=c(n,s,p,q,\sigma)$, as $\frac{sp}{p-1}-\frac{n}{q(p-1)}>s_0$. 
    Therefore, we have that if $x\in \mathcal{A}_2$ with $d(x)\geq d(x+h)$, then
    \begin{align*}
      \frac{|u(x+h)-u(x)|^{\gamma}}{|h|^{\sigma_{1}\gamma}}\leq c[u]^{\gamma}_{C^{\sigma_1}(B_{d(x)/4}(x))}\leq \frac{c}{d^{(\sigma_1-s_0)\gamma}(x)}\left(\widetilde{E}(u;B_{1})+\|f\|^{1/(p-1)}_{L^q(\Omega\cap B_1)}\right)^{\gamma},
    \end{align*}
    for some $c=c(n,s,p,q,\gamma,\sigma)$. Similarly, we can also show that 
    if $x\in \mathcal{A}_2$ with $d(x+h)\geq d(x)$, then
    \begin{align*}
      \frac{|u(x+h)-u(x)|^{\gamma}}{|h|^{\sigma_{1}\gamma}}\leq c[u]^{\gamma}_{C^{\sigma_1}(B_{d(x+h)/4}(x+h))}
        \leq \frac{c}{d^{(\sigma_1-s_0)\gamma}(x+h)}\left(\widetilde{E}(u;B_{1})+\|f\|^{1/(p-1)}_{L^q(\Omega\cap B_1)}\right)^{\gamma}.
    \end{align*}
    Thus, we get 
    \begin{align*}
        I_2&\leq \int_{\mathcal{A}_2} \frac{|u(x+h)-u(x)|^{\gamma}}{|h|^{\sigma_{1}\gamma}}\,dx\\
        &\leq c\left[\int_{B_{1/4}\cap \Omega}\left(\frac{1}{d^{(\sigma_1-s_0)\gamma}(x+h)}+\frac{1}{d^{(\sigma_1-s_0)\gamma}(x)}\right)\,dx\right]\left(\widetilde{E}(u;B_{1})+\|f\|^{1/(p-1)}_{L^q(\Omega\cap B_1)}\right)^{\gamma} 
    \end{align*}
    for some constant $c=c(n,s,p,q,\gamma,\sigma)$.
    Now, by \eqref{ineq1.basic1} and the fact that $(\sigma_1 - s_0)\gamma<1$, we have 
    \begin{align}\label{ineqi2.dsregmove}
        I_2\leq c\left(\widetilde{E}(u;B_{1})+\|f\|^{1/(p-1)}_{L^q(\Omega\cap B_1)}\right)^{\gamma}
    \end{align}
    for some constant $c=c(n,s,p,q,\gamma,\sigma)$, whenever $\delta=\delta(n,s,p,q,\gamma,\sigma)>0$ is sufficiently small. 
    
    \textbf{Step 3: Estimates of $I_3$, $I_4$ and $I_5$.} Concerning the terms $I_3$, we observe
    \begin{equation}\label{ineq2.dsregmove}
    \begin{aligned}
        \int_{\mathcal{A}_3}\frac{|u(x+h)-u(x)|^\gamma}{|h|^{\sigma_1\gamma}}\,dx&=\int_{\mathcal{A}_3}\frac{|u(x)|^\gamma}{|h|^{\sigma_{1}\gamma}}\,dx\leq \int_{B_{1/8}\cap\Omega}\frac{|u(x)|^\gamma}{d^{\sigma_{1}\gamma }(x)}\,dx,
    \end{aligned}
    \end{equation}
    where we have used the fact that 
    \begin{align*}
        d(x)=|x-\overline{x}|\leq |x-(x+h)|=|h|
    \end{align*}
    with $\overline{x}\in\bbR^n\setminus\Omega$, as $x+h\in B_{1/4}\setminus \Omega$. 
    Similarly, we have 
    \begin{align*}
         \int_{\mathcal{A}_4}\frac{|u(x+h)-u(x)|^\gamma}{|h|^{\sigma_1\gamma}}\,dx\leq \int_{B_{1/4}\cap\Omega}\frac{|u(x+h)|^\gamma}{d^{\sigma_{1}\gamma }(x+h)}\,dx.
    \end{align*}
    Thus, using Lemma \ref{lem.basic1}, we have
\begin{align}\label{ineqi34.dsregmove}
    \int_{\mathcal{A}_3}\frac{|u(x+h)-u(x)|^\gamma}{|h|^{\sigma_1\gamma}}\,dx+\int_{\mathcal{A}_4}\frac{|u(x+h)-u(x)|^\gamma}{|h|^{\sigma_1\gamma}}\,dx\leq c\left(\widetilde{E}(u;B_1)+\|f\|^{1/(p-1)}_{L^q(\Omega\cap B_1)}\right)^{\gamma}
\end{align}
for some constant $c=c(n,s,p,q,\gamma,\sigma)$, whenever $\delta=\delta(n,s,p,q,\gamma,\sigma)>0$ is sufficiently small.
   Lastly, since $x+h,x\in B_1\setminus \Omega$ for $x \in \mathcal{A}_{5}$, we directly get
    \begin{align}\label{ineqi5.dsregmove}
         \int_{\mathcal{A}_5}\frac{|u(x+h)-u(x)|^\gamma}{|h|^{\sigma_1\gamma}}\,dx=0.
    \end{align}

\textbf{Step 4: Conclusion.} 
Combining all the estimates \eqref{ineqi1.dsregmove}, \eqref{ineqi2.dsregmove}, \eqref{ineqi34.dsregmove} and \eqref{ineqi5.dsregmove}, we have 
\begin{align*}
   I = \int_{B_{1/8}}\frac{|u(x+h)-u(x)|^\gamma}{|h|^{\sigma_1\gamma}}\,dx\leq c\left(\widetilde{E}(u;B_1)+\|f\|^{1/(p-1)}_{L^q(\Omega\cap B_1)}\right)^{\gamma}
\end{align*}
for some constant $c=c(n,s,p,q,\gamma,\sigma)$, whenever $\delta=\delta(n,s,p,q,\gamma,\sigma)>0$ is sufficiently small. Since this holds for any $h\in B_{1/1000}\setminus \{0\}$, an application of \cite[Lemma 2.3]{DieKimLeeNow24} yields
\begin{align}\label{ineqlast.hig}
    [u]_{W^{\sigma,\gamma}(B_{1/16})}\leq c\left(\widetilde{E}(u;B_1)+\|f\|^{1/(p-1)}_{L^q(\Omega\cap B_1)}+\|u\|_{L^\gamma(B_{1/8})}\right)
\end{align}
for some constant $c=c(n,s,p,q,\sigma)$. In light of the fact that $d(x)\leq 2$ for any $x\in\Omega\cap B_1$, Theorem \ref{thm.lp} and H\"older's inequality, we obtain
\begin{align}\label{ineqlast2.hig}
    \|u\|_{L^\gamma(B_{1/8})}\leq c\|u/d^s\|_{L^\gamma(B_{1/8}\cap \Omega)}\leq c\left(\widetilde{E}(u;B_1)+\|f\|^{1/(p-1)}_{L^q(\Omega\cap B_1)}\right)
\end{align}
for some constant $c=c(n,s,p,q,\sigma)$, provided that $\delta=\delta(n,s,p,q,\sigma)>0$ is sufficiently small.
Plugging this into \eqref{ineqlast.hig} and 
using a standard covering argument, we get the desired estimate. 
\end{proof}
We end this subsection with proving Theorem \ref{cor} regarding higher H\"older regularity.
\begin{proof}[Proof of Theorem \ref{cor}]
    Suppose $f\in L^{n/s}(\Omega\cap B_2)$, and let $\alpha\in(0,s)$ be a fixed number. We choose 
    \begin{equation}\label{cond.hol}
        \sigma_1\coloneqq \frac{\alpha+s}{2}\quad\text{and}\quad \gamma\coloneqq\frac{4n}{s-\alpha}.
    \end{equation}
    Estimating in a completely similar way as in the proof of Theorem \ref{thm.hig} with $q=n/s$, but without using Lemma \ref{lem.basic1}, we derive 
    \begin{align*}
        \sup_{h\in B_{1/1000}\setminus \{0\}}\int_{B_{1/8}}\frac{|u(x+h)-u(x)|^\gamma}{|h|^{\sigma_1\gamma}}\,dx\leq c\left(\widetilde{E}(u;B_{1})+\|f\|^{1/(p-1)}_{L^{n/s}(\Omega\cap B_1)}\right)^{\gamma}
    \end{align*}
    for some constant $c=c(n,s,p,\alpha)$, whenever $\delta=\delta(n,s,p,\alpha)>0$ is sufficiently small. We then employ \cite[Lemma 2.3]{DieKimLeeNow24} to see that
    \begin{align*}
        \|u\|_{W^{(5\alpha+3s)/8,\gamma}(B_{1/16})}\leq c\left(\widetilde{E}(u;B_{1})+\|f\|^{1/(p-1)}_{L^{n/s}(\Omega\cap B_1)}\right)
    \end{align*}
     for some constant $c=c(n,s,p,\alpha)$, where  we have also used \eqref{ineqlast2.hig}. Now, by \cite[Theorem 8.2]{DinPalVal12}, along with the fact that $(5\alpha+3s)/8-n/\gamma>\alpha$ which follows from \eqref{cond.hol}, we obtain 
     \begin{align*}
         [u]_{C^\alpha(B_{1/32})}\leq c\left(\widetilde{E}(u;B_{1})+\|f\|^{1/(p-1)}_{L^{n/s}(\Omega\cap B_1)}\right).
     \end{align*}
     By using a standard covering argument, we get the desired estimate.
\end{proof}

\subsection{Calder\'on-Zygmund type estimates} In this subsection, we present Calder\'on-Zygmund type estimates for linear equations on nonsmooth domains.

We first prove an optimal integrability result for a weak solution in terms of the right-hand side.
\begin{lemma}\label{lem.comp.high}
    Assume that $\Omega$ satisfies \eqref{cond.meas} with $R_0 = r$. Let $v\in X^{s,2}_0(B_r(z)\cap \Omega)$ be the weak solution to 
    \begin{equation}\label{eq.high}
\left\{
\begin{alignedat}{3}
(-\Delta)^s{v}&= f&&\qquad \mbox{in  $ B_r(z)\cap \Omega$}, \\
{v}&=0&&\qquad  \mbox{in $\bbR^n\setminus (\Omega\cap B_r(z)) $}
\end{alignedat} \right.
\end{equation}
with $B_{r/2}(z) \cap \Omega \neq \emptyset$. Then for any $q \in [2_{*}, n/(2s))$, there exists a constant $c=c(n,s,q)$ such that
\begin{align*}
    \left(\dashint_{B_r(z)\cap\Omega}|v|^{\frac{nq}{n-2sq}}\,dx\right)^{\frac{n-2sq}{nq}}\leq c\left(\dashint_{B_r(z)\cap \Omega}r^{2sq}|f|^q\,dx\right)^{\frac1q}.
\end{align*}
\end{lemma}
\begin{proof}
Let $w\in W^{s,2}(\bbR^n)$ be the weak solution to 
\begin{equation*}
\left\{
\begin{alignedat}{3}
(-\Delta)^s{w}&=  f_+&&\qquad \mbox{in  $ B_r(z)$}, \\
{w}&=0&&\qquad  \mbox{in $\bbR^n\setminus B_r(z) $},
\end{alignedat} \right.
\end{equation*}
where $f_+\equiv0$ in $B_r(z)\setminus \Omega$. Then, by \cite[Theorem 16]{LeoPerPRiSor15}, we get 
\begin{align}\label{ineq0.high}
   \left(\dashint_{B_r(z)}|w|^{\frac{nq}{n-2sq}}\,dx\right)^{\frac{n-2sq}{nq}}\leq c\left(\dashint_{B_r(z)}r^{2sq}|f_+|^q\,dx\right)^{\frac1q},
\end{align}
where $c=c(n,s,q)$. Since the Green function of $(-\Delta)^s$ with respect to $B_{r}(z)$ is nonnegative (see \cite{CheKimSon10}) and $f_+\geq0$, we have $w\geq0$ in $\bbR^n$. Thus, we see that
\begin{equation*}
\left\{
\begin{alignedat}{3}
(-\Delta)^s{(v-w)}&\leq  0&&\qquad \mbox{in  $ B_r(z)\cap\Omega$}, \\
{v-w}&\leq0&&\qquad  \mbox{in $\bbR^n\setminus (B_r(z)\cap\Omega) $}.
\end{alignedat} \right.
\end{equation*}
In addition, since $B_{r/2}(z)\cap \Omega\neq\emptyset$, by \eqref{cond.meas}, we have 
\begin{align*}
    |B_r(z)\cap \Omega| \ge  |B_r(z)|/c
\end{align*}
for some constant $c=c(n)$.
Using this and the maximum principle given in \cite[Lemma 2.3.3]{FerRos24}, we have $v\leq w$ in $\bbR^n$. Thus, \eqref{ineq0.high} and the fact that $v=f_+=0$ in $B_r(z)\setminus \Omega$ imply
\begin{align}\label{ineq1.high}
   \left(\dashint_{B_r(z)\cap \Omega}|v_+|^{\frac{nq}{n-2sq}}\,dx\right)^{\frac{n-2sq}{nq}}\leq c\left(\dashint_{B_r(z)\cap \Omega}r^{2sq}|f_+|^q\,dx\right)^{\frac1q}
\end{align}
for some constant $c=c(n,s,q)$. Since \eqref{ineq1.high} also holds with $v_+$ and $f_+$ replaced by $v_-$ and $f_-$, respectively, we have the desired estimate.
\end{proof}

\begin{remark}\label{rmk.comp.high}
We are now able to obtain the following comparison estimate by using Lemma \ref{lem.comp.high} and the linearity of the operator. Let $u$ be a a weak solution to 
\begin{equation*}
\left\{
\begin{alignedat}{3}
(-\Delta)^s{u}&= f&&\qquad \mbox{in  $B_{2r}(z)\cap \Omega$}, \\
u&=0&&\qquad  \mbox{in $B_{2r}(z)\setminus \Omega $},
\end{alignedat} \right.
\end{equation*}
and let $v$ be the weak solution to 
\begin{equation*}
\left\{
\begin{alignedat}{3}
(-\Delta)^s{v}&= 0&&\qquad \mbox{in  $B_{r}(z)\cap \Omega$}, \\
{v}&=u&&\qquad  \mbox{in $\bbR^n\setminus (B_{r}(z)\cap \Omega) $}.
\end{alignedat} \right.
\end{equation*}
Since $u-v$ is the weak solution to \eqref{eq.high}, the estimate given in Lemma \ref{lem.comp.high} holds with $v$ replaced by $u-v$.
\end{remark}

We now revise the estimate given in Theorem \ref{thm.hig} which is useful for the proof of Theorem \ref{thm.cz}.
\begin{lemma}\label{lem.hig}
    Let $v$ be a weak solution to 
    \begin{equation*}
\left\{
\begin{alignedat}{3}
(-\Delta)^s{v}&= 0&&\qquad \mbox{in  $B_{4r}(z)\cap \Omega$}, \\
{v}&=0&&\qquad  \mbox{in $B_{4r}(z)\setminus\Omega $},
\end{alignedat} \right.
\end{equation*}
where $z\in \partial\Omega$ and $\Omega$ is $(\delta,4r)$-Reifenberg flat. Let us fix $\gamma\geq1$ and $\sigma<\min\{s+1/\gamma,1\}$. If $\delta=\delta(n,s,\gamma,\sigma)>0$ is sufficiently small, then we have 
\begin{align*}
    r^{\sigma-n/\gamma}[v]_{W^{\sigma,\gamma}(B_{3r}(z))}\leq c\widetilde{E}(v-(v)_{B_{4r}(z)};B_{4r}(z))
\end{align*}
for some constant $c=c(n,s,\gamma,\sigma)$.
\end{lemma}
\begin{proof}
By Lemma \ref{lem.scale} and Theorem \ref{thm.hig} with $f=0$, we have
    \begin{align}\label{ineq2.hig}
    r^{\sigma-n/\gamma}[v]_{W^{\sigma,\gamma}(B_{3r}(z))}\leq c\widetilde{E}(v;B_{4r}(z))
\end{align}
for some constant $c=c(n,s,\gamma,\sigma)$, whenever $\delta=\delta(n,s,\gamma,\sigma)>0$ is sufficiently small. 
Since $z\in \partial\Omega$, \eqref{cond.meas} implies 
\begin{align}\label{ineq1.ave}
|B_{4r}(z)|/4^n \leq | B_{4r}(z)\setminus \Omega |.
\end{align}
Using \eqref{ineq1.ave} along with the facts that $v\equiv 0$ in $\bbR^n\setminus \Omega$ and that
\begin{align*}
|(v)_{B_{4r}(z)}-(v)_{B_{4r}(z)\setminus \Omega}|\leq \dashint_{B_{4r}(z)\setminus\Omega}|v-(v)_{B_{4r}(z)}|\,dx,
\end{align*}
we estimate 
\begin{align*}
    |(v)_{B_{4r}(z)}|=\dashint_{B_{4r}(z)}|v-(v)_{B_{4r}(z)\setminus \Omega}|\,dx&\leq \dashint_{B_{4r}(z)}|v-(v)_{B_{4r}(z)}|\,dx+|(v)_{B_{4r}(z)}-(v)_{B_{4r}(z)\setminus \Omega}|\\
    &\leq c\dashint_{B_{4r}(z)}|v-(v)_{B_{4r}(z)}|\,dx
\end{align*}
for some constant $c=c(n)$, which with H\"older's inequality further gives
    \begin{align}\label{ineq1.hig}
        \widetilde{E}(v;B_{4r}(z))&\leq \widetilde{E}(v-(v)_{B_{4r}(z)};B_{4r}(z))+c(v)_{B_{4r}(z)}\leq c\widetilde{E}(v-(v)_{B_{4r}(z)};B_{4r}(z))
    \end{align}
    for some constant $c=c(n,s)$. 
Connecting \eqref{ineq1.hig} to \eqref{ineq2.hig} yields the desired result. 
\end{proof}
We now prove Theorem \ref{thm.cz}.

\begin{proof}[Proof of Theorem \ref{thm.cz}.]
We fix $\gamma<nq/(n-2sq)$ and 
\[ \sigma \in (s,\min\{s+1/\gamma,1\}) \quad \text{with} \quad \sigma<2s+n/\gamma - n/q. \]
We next choose
\begin{equation*}
    \sigma_1\coloneqq(\sigma+\min\{s+1/\gamma,1,2s+n/\gamma-n/q\})/2\in(\sigma,\min\{s+1/\gamma,1\})
\end{equation*}
to see that $\sigma_1 - n/\gamma < 2s-n/q$. Then we set
\begin{equation}\label{set.beta}
    \beta\coloneqq \frac{\sigma_1 \gamma}{2s\gamma+n(1-\gamma/q)}\in(0,1)
\end{equation}
Let us set $1\leq\rho<r\leq 2$ 
and fix $h\in B_{2^{-10}(r-\rho)^{1/\beta}}\setminus \{0\}$. By \cite[Lemma 2.11]{DieKimLeeNow24j}, there exist a constant $c=c(n)$, a finite index set $I$ and a sequence $\{z_i\}_{i \in I}\subset B_{\rho}$ such that 
    \begin{align}\label{ineq0.czb}
        B_{\rho}\subset \bigcup_{i\in I} B_{|h|^\beta}(z_i)\subset B_{r},\quad \sup_{x\in\bbR^n}\sum_{i\in I}{\mbox{\Large$\chi$}}_{B_{2^k|h|^\beta}(z_i)}(x)\leq c2^{nk},\quad |I|\leq c/|h|^{n\beta},
    \end{align}
    where $|I|$ denotes the number of elements in the set $I$. 
Accordingly, we introduce the finite difference operator $\delta_{h}$ defined by
\[ \delta_{h}u(x) \coloneqq (\delta_{h}u)(x) \coloneqq u(x+h)-u(x), \qquad x \in B_{\rho}. \]
We are going to prove that, for any $h \in B_{2^{-10}(r-\rho)^{1/\beta}}\setminus\{0\}$, 
\begin{equation}\label{boot.czb}
\begin{aligned}
    \int_{B_{\rho}}|\delta_hu|^\gamma \,dx \leq c|h|^{\gamma(\sigma_1(1-\beta)+\alpha\beta)}[u]^{\gamma}_{W^{\alpha,\gamma}(B_r)}
    +c|h|^{\sigma_1 \gamma}\left[\widetilde{E}(u;B_r)^\gamma+\left(\int_{B_{r}\cap \Omega}|f|^{q}\,dx\right)^{\gamma/q}\right],
\end{aligned}
\end{equation}
holds for some constant $c=c(n,s,\gamma,q,\sigma,r-\rho)$, whenever $\delta=\delta(n,s,q,\gamma,\sigma)>0$ is sufficiently small and $u\in W^{\alpha,\gamma}(B_r)$ for some $\alpha\in [0,\sigma_1)$. In particular when $\alpha=0$, we identify $W^{0,\gamma} \equiv L^\gamma$ and $[u]_{W^{0,\gamma}(B_r)} \equiv \|u\|_{L^{\gamma}(B_r)}$. 
    Now, with $i \in I$ being fixed, we estimate 
    \begin{align*}
        J\coloneqq \int_{B_{|h|^\beta}(z_i)}|\delta_hu|^\gamma \,dx,
    \end{align*}
    depending on the location of the ball $B_{2|h|^\beta}(z_i)$. 

        \textbf{Step 1: The exterior case.} We first assume that $B_{2|h|^\beta}(z_i)\cap \Omega=\emptyset$. In this case, since $|h|\leq |h|^\beta$ for $|h|\leq 1$ and $\beta<1$, we have $x,x+h\in B_{2|h|^\beta}(z_i)$ when $x\in B_{|h|^\beta}(z_i)$. Thus, we get
        \begin{align}\label{ineq01.czb}
            J=\int_{B_{|h|^\beta}(z_i)}|\delta_hu|^\gamma\,dx=0.
        \end{align}
        
        \textbf{Step 2: The interior case.} We next assume that $B_{2|h|^\beta}(z_i)\cap (\bbR^n\setminus\Omega)=\emptyset$. From \cite[Proposition 2.6]{BraLin17} and the fact that $0<|h| < |h|^{\beta}/2$, we note 
        \begin{align}\label{ineq1.czb}
            \left\|\frac{\delta_h u}{|h|^{\sigma_1}}\right\|_{L^\gamma(B_{|h|^\beta}(z_i))}\leq c\left([u]_{W^{\sigma_1,\gamma}(B_{3|h|^\beta/2}(z_i))}+\frac{\|u-(u)_{B_{2|h|^\beta}(z_i)}\|_{L^\gamma(B_{3|h|^\beta/2}(z_i))}}{|h|^{\beta\sigma_1 }}\right).
        \end{align}
     Now, using \cite[Theorem 1.4]{DieNow23} together with the fact that $B_{2|h|^\beta}(z_i)\Subset \Omega$, we have  
        \begin{align*}
            &|h|^{-n\beta/\gamma}[u]_{W^{\sigma_1,\gamma}(B_{3|h|^\beta/2}(z_i))}\\
            &\leq c|h|^{-\sigma_1\beta }\left(\widetilde{E}(u-(u)_{B_{2|h|^\beta}(z_i)};B_{2|h|^\beta}(z_i))+|h|^{-n\beta/q}\||h|^{2s\beta}f\|_{L^q(B_{2|h|^\beta}(z_i))}\right)\\
            &\leq c|h|^{-\sigma_1\beta}|h|^{-n\beta/\gamma }\|u-(u)_{B_{2|h|^\beta}(z_i)}\|_{L^\gamma(B_{2|h|^\beta/2}(z_i))}\\
            &\quad+c|h|^{-\sigma_1\beta}\left(\mathrm{Tail}(u-(u)_{B_{2|h|^\beta}(z_i)};B_{2|h|^\beta}(z_i))+|h|^{-n\beta/q}\||h|^{2s\beta}f\|_{L^q(B_{2|h|^\beta}(z_i))}\right).
        \end{align*}
        Thus, we estimate $J$ as 
         \begin{equation}\label{ineq2.czb}
         \begin{aligned}
             J&\leq c|h|^{\gamma\sigma_1(1-\beta)}\left(\int_{B_{2|h|^\beta}(z_i)}|u-(u)_{B_{2|h|^\beta}(z_i)}|^\gamma\,dx+|h|^{n\beta}\mathrm{Tail}(u-(u)_{B_{2|h|^\beta}(z_i)})^\gamma\right)\\
             &\quad +c|h|^{\gamma\sigma_1(1-\beta)+n\beta(1-\gamma/q)+2s\beta\gamma}\left(\int_{B_{2|h|^\beta}(z_i)}|f|^q\,dx\right)^{\gamma/q}.
         \end{aligned}
         \end{equation}
         
        \textbf{Step 3: The boundary case.} We finally assume that $B_{2|h|^\beta}(z_i)\cap \Omega\neq \emptyset$ and $B_{2|h|^\beta}(z_i)\cap (\bbR^n\setminus\Omega)\neq \emptyset$. 
        Let $v_i$ be the weak solution to 
\begin{equation*}
\left\{
\begin{alignedat}{3}
(-\Delta)^s{v}_i&= 0&&\qquad \mbox{in  $B_{8|h|^\beta}(z_i)\cap \Omega$}, \\
{v}_i&=u&&\qquad  \mbox{in $\bbR^n\setminus (B_{8|h|^\beta}(z_i)\cap \Omega) $}.
\end{alignedat} \right.
\end{equation*}
We then see that $u-v_i$ is the weak solution to \eqref{eq.high} with $u$ and $B_r(z)$ replaced by $u-v_i$ and $B_{8|h|^\beta}(z_i)$, respectively. 
Thus, by H\"older's inequality and Remark \ref{rmk.comp.high}, we get
\begin{equation}\label{ineq.j1.cz}
\begin{aligned}
    \left(\dashint_{\Omega\cap B_{8|h|^\beta}(z_i)}|u-v|^{\gamma}\,dx\right)^{1/\gamma} & \leq \left(\dashint_{\Omega\cap B_{8|h|^\beta}(z_i)}|u-v|^{\frac{nq}{n-2sq}}\,dx\right)^{\frac{n-2sq}{nq}} \\
    & \leq c\left(\dashint_{\Omega\cap B_{8|h|^\beta}(z_i)}||h|^{2s\beta}f|^{q}\,dx\right)^{1/q}
\end{aligned}
\end{equation}
for some constant $c=c(n,s,q)$. 
Using this, we have
\begin{equation}\label{ineq.last.cz}
\begin{aligned}
J&\leq c\int_{B_{|h|^\beta}(z_i)}{|\delta_h(u-v)|^\gamma}\,dx+c\int_{B_{|h|^\beta}(z_i)}{|\delta_hv|^\gamma}\,dx\\
&\leq c|h|^{2s\beta\gamma+n\beta(1-\gamma/q)}\left(\int_{B_{8|h|^\beta}(z_i)\cap\Omega}|f|^q\,dx\right)^{\gamma/q}+\underbrace{c\int_{B_{|h|^\beta}(z_i)}{|\delta_hv|^\gamma}\,dx}_{\eqqcolon J_1}.
\end{aligned}
\end{equation}
To estimate $J_1$, we note that there is a point $\overline{z_i}\in\partial\Omega$ such that 
\begin{equation*}
    B_{|h|^\beta}(z_i)\subset B_{3|h|^\beta}(\overline{z_i})\subset B_{4|h|^\beta}(\overline{z_i})\subset B_{8|h|^\beta}(z_i).
\end{equation*}
Thus, we see that $v_i$ is also a weak solution to
\begin{equation*}
\left\{
\begin{alignedat}{3}
(-\Delta)^s{v}_i&= 0&&\qquad \mbox{in  $B_{4|h|^\beta}(\overline{z_i})\cap \Omega$}, \\
{v}_i&=0&&\qquad  \mbox{in $ B_{4|h|^\beta}(\overline{z_i})\setminus \Omega $}.
\end{alignedat} \right.
\end{equation*}
By \eqref{ineq1.czb} with $u=v$ and Lemma \ref{lem.hig} together with the fact that $\sigma_1<\min\{s+1/\gamma,1\}$, we derive
\begin{align*}
    J_1^{1/\gamma}&\leq c|h|^{\sigma_1}\left([v]_{W^{\sigma_1,\gamma}(B_{4|h|^\beta/2}(z_i))}+|h|^{(1-\beta)\sigma_1 }\|v-(v)_{B_{4|h|^\beta}(z_i)}\|_{L^\gamma(B_{4|h|^\beta/2}(z_i))}\right)\\
     &\leq c|h|^{(1-\beta)\sigma_1}\|v-(v)_{B_{5|h|^\beta}(z_i)}\|_{L^\gamma(B_{5|h|^\beta}(z_i))}\\
            &\quad+c|h|^{(1-\beta)\sigma_1}|h|^{n\beta/\gamma }\mathrm{Tail}(v-(v)_{B_{5|h|^\beta}(z_i)};B_{5|h|^\beta}(z_i)),
\end{align*}
when $\delta>0$ is sufficiently small.
We now use \eqref{ineq.j1.cz} to further estimate $J_1$ as
\begin{align*}
    J_1
    &\leq c|h|^{\gamma\sigma_1(1-\beta)}\left(\int_{B_{8|h|^\beta}(z_i)}|u-(u)_{B_{8|h|^\beta}(z_i)}|^\gamma\,dx+|h|^{n\beta}\mathrm{Tail}(u-(u)_{B_{8|h|^\beta}(z_i)};B_{8|h|^\beta}(z_i))^{\gamma}\right)\\
             &\quad +c|h|^{2s\beta\gamma+\gamma\sigma_1(1-\beta)+n\beta(1-{\gamma}/{q})}\left(\int_{B_{8|h|^\beta}(z_i)\cap \Omega}|f|^q\,dx\right)^{\gamma/q}.
\end{align*}
Plugging this into \eqref{ineq.last.cz}, we get
\begin{equation}\label{ineq3.czb}
    \begin{aligned}
        J&\leq c|h|^{\gamma\sigma_1(1-\beta)}\left(\int_{B_{8|h|^\beta}(z_i)}|u-(u)_{B_{8|h|^\beta}(z_i)}|^\gamma\,dx+|h|^{n\beta}\mathrm{Tail}(u-(u)_{B_{8|h|^\beta}(z_i)};B_{8|h|^\beta}(z_i))^\gamma\right)\\
             &\quad +c|h|^{2s\beta\gamma+n\beta(1-\gamma/q)}\left(\int_{B_{8|h|^\beta}(z_i)\cap \Omega}|f|^q\,dx\right)^{\gamma/q}.
    \end{aligned}
\end{equation}

\textbf{Step 4: Tail estimates.} 
    By \eqref{ineq01.czb}, \eqref{ineq2.czb} and \eqref{ineq3.czb}, we derive
    \begin{equation}\label{ineq4.czb}
    \begin{aligned}
        J&\leq c|h|^{\gamma\sigma_1(1-\beta)}\left(\int_{B_{8|h|^\beta}(z_i)}|u-(u)_{B_{8|h|^\beta}(z_i)}|^\gamma\,dx+|h|^{n\beta}\mathrm{Tail}(u-(u)_{B_{8|h|^\beta}(z_i)};B_{8|h|^\beta}(z_i))^\gamma\right)\\
             &\quad +c|h|^{2s\beta\gamma+n\beta(1-\gamma/q)}\left(\int_{B_{8|h|^\beta}(z_i)\cap \Omega}|f|^q\,dx\right)^{\gamma/q}.
    \end{aligned}
\end{equation}
We then estimate the tail term. We first split as
\begin{equation*}
\begin{aligned}
    & \mathrm{Tail}(u-(u)_{B_{8|h|^\beta}(z_i)};B_{8|h|^\beta}(z_i)) \\
    &\leq c\sum_{j=1}^{l}2^{-2sj}\dashint_{B_{2^{j+3}|h|^\beta}(z_i)}|u-(u)_{B_{2^{j+3}|h|^\beta}(z_i)}|\,dx+ c2^{-2sl}\mathrm{Tail}(u-(u)_{B_{2^{l+4}|h|^\beta}(z_i)};B_{2^{l+4}|h|^\beta}(z_i))\\
    &\leq c\sum_{j=1}^{l}2^{-2sj}\dashint_{B_{2^{j+3}|h|^\beta}(z_i)}|u-(u)_{B_{2^{j+3}|h|^\beta}(z_i)}|\,dx+c|h|^{2s\beta}\widetilde{E}(u;B_r),
\end{aligned}
\end{equation*}
where $l$ is the positive integer satisfying
    \begin{equation}\label{intl.czb}
        (r-\rho)/32<2^{3+l}|h|^\beta\leq (r-\rho)/16.
    \end{equation}
    Letting
    \begin{equation}\label{defn.sigma2}
    \sigma_2\coloneqq(\sigma_1+2s)/2\in(\sigma_1,2s),
    \end{equation}
    and using H\"older's inequality, we get
\begin{align*}        &\left(\sum_{j=1}^{l}2^{-2sj}\dashint_{B_{2^{j+3}|h|^\beta}(z_i)}|u-(u)_{B_{2^{j+3}|h|^\beta}(z_i)}|\,dx\right)^\gamma\\
    &= \left(\sum_{j=1}^{l} 2^{(\sigma_2 -2s)j}2^{-\sigma_2 j}\dashint_{B_{2^{j+3}|h|^\beta}(z_i)}|u-(u)_{B_{2^{j+3}|h|^\beta}(z_i)}|\,dx\right)^\gamma\\
    & \le \left(\sum_{j=1}^{l}2^{j(\sigma_{2}-2s)\gamma/(\gamma-1)}\right)^{\gamma-1}\left[\sum_{j=1}^{l}2^{-\sigma_{2}\gamma j}\left(\dashint_{B_{2^{j+3}|h|^{\beta}}(z_i)}|u-(u)_{B_{2^{j+3}|h|^{\beta}}(z_i)}|\,dx\right)^{^{\gamma}}\right] \\
    &\leq c\sum_{j=1}^{l}2^{-\sigma_2\gamma j}\dashint_{B_{2^{j+3}|h|^\beta}(z_i)}|u-(u)_{B_{2^{j+3}|h|^\beta}(z_i)}|^\gamma\,dx.
    \end{align*}
This in turn gives the following tail estimate:
    \begin{equation}\label{ineq5.czb}
    \begin{aligned}
        & |h|^{n\beta} \mathrm{Tail}(u-(u)_{B_{8|h|^\beta}(z_i)};B_{8|h|^\beta}(z_i))^\gamma \\
        &\leq c\sum_{j=1}^{l}2^{-(\sigma_2 \gamma+n)j}\int_{B_{2^{j+3}|h|^\beta}(z_i)}|u-(u)_{B_{2^{3+j}|h|^\beta}(z_i)}|^\gamma\,dx
    +c|h|^{n\beta+2s\beta \gamma}\widetilde{E}(u;B_r)^\gamma .
    \end{aligned}
    \end{equation}

    We now plug \eqref{ineq5.czb} into \eqref{ineq4.czb} to get 
    \begin{equation}\label{ineq6.czb}
    \begin{aligned}
        \int_{B_{|h|^\beta}(z_i)}|\delta_h u|^\gamma\,dx&\leq c|h|^{\gamma\sigma_1(1-\beta)}\sum_{j=0}^{l}2^{-(\sigma_2 \gamma+n)j}\int_{B_{2^{j+3}|h|^\beta}(z_i)}|u-(u)_{B_{2^{3+j}|h|^\beta}(z_i)}|^\gamma\,dx\\
        &\quad+c|h|^{2s\beta \gamma+n\beta}\widetilde{E}(u;B_1)^\gamma+c|h|^{2s\beta\gamma+n\beta(1-\gamma/q)}\left(\int_{B_{8|h|^\beta}(z_i)\cap \Omega}|f|^q\,dx\right)^{\gamma/q}.
    \end{aligned}
    \end{equation}
By employing \eqref{ineq0.czb} and \eqref{ineq6.czb}, we have 
\begin{equation}\label{ineq7.czb}
\begin{aligned}
    \int_{B_{3/4}}|\delta_hu|^\gamma \,dx &\leq \sum_{i\in I }\int_{B_{|h|^\beta}(z_i)}|\delta_h u|^\gamma \,dx\\
    &\leq c\sum_{i\in I} |h|^{\gamma\sigma_1(1-\beta)}\sum_{j=0}^{l}2^{-(\sigma_2\gamma+n)j}\int_{B_{2^{j+3}|h|^\beta}(z_i)}|u-(u)_{B_{2^{3+j}|h|^\beta}(z_i)}|^\gamma\,dx\\
        &\quad+c\sum_{i\in I}|h|^{n\beta+2s\beta \gamma}\widetilde{E}(u;B_r)^\gamma\\
        &\quad+c\sum_{i\in I}|h|^{2s\beta\gamma+n\beta(1-\gamma/q)}\left(\int_{B_{8|h|^\beta}(z_i)\cap \Omega}|f|^q\,dx\right)^{\gamma/q}\\
        &\leq c\sum_{i\in I} |h|^{\gamma\sigma_1(1-\beta)}\sum_{j=0}^{l}2^{-(\sigma_2 \gamma+n)j}\int_{B_{2^{j+3}|h|^\beta}(z_i)}|u-(u)_{B_{2^{3+j}|h|^\beta}(z_i)}|^\gamma\,dx\\
        &\quad+c|h|^{2s\beta\gamma+n\beta(1-\gamma/q)}\left[\widetilde{E}(u;B_r)^\gamma+\left(\int_{B_{r}\cap \Omega}|f|^{q}\,dx\right)^{\gamma/q}\right].
\end{aligned}
\end{equation}
Indeed, we have also used 
\begin{align*}
    \sum_{i\in I}\left(\int_{B_{8|h|^\beta}(z_i)\cap \Omega}|f|^q\,dx\right)^{\gamma/q}\leq\left(\sum_{i\in I}\int_{B_r\cap \Omega}|f|^q{\mbox{\Large$\chi$}}_{B_{8|h|^\beta}(z_i)\cap \Omega}\,dx\right)^{\gamma/q}\leq c\left(\int_{B_r\cap \Omega}|f|^q\,dx\right)^{\gamma/q},
\end{align*}
which follows from the fact that $\gamma/q\geq1$ and \eqref{ineq0.czb}.

\textbf{Step 5: Fractional bootstrap.} 
We now use a bootstrap argument together with a suitable choice of parameter in order to further estimate the term given in \eqref{ineq7.czb}. Assume that $u\in W^{\alpha,\gamma}(B_r)$ for some $\alpha \in (0,\sigma_1)$. Then, by the fractional Poincar\'e inequality, \eqref{intl.czb} and \eqref{ineq0.czb},   
\begin{align*}
    &\sum_{i\in I} |h|^{\gamma\sigma_1(1-\beta)}\sum_{j=1}^{l}2^{-(\sigma_2\gamma+n)j}\int_{B_{2^{j+3}|h|^\beta}(z_i)}|u-(u)_{B_{2^{3+j}|h|^\beta}(z_i)}|^\gamma\,dx\\
    &\leq c\sum_{i\in I} |h|^{\gamma\sigma_1(1-\beta)}\sum_{j=1}^{l}2^{-(\sigma_2 \gamma+n)j}(2^{j+3}|h|^\beta)^{\alpha  \gamma}\int_{B_{2^{j+3}|h|^\beta}(z_i)}\int_{B_{2^{j+3}|h|^\beta}(z_i)}\frac{|u(x)-u(y)|^\gamma}{|x-y|^{n+\alpha \gamma}}\,dx\,dy\\
    &\leq  c|h|^{\gamma\sigma_1(1-\beta)}|h|^{\alpha \beta \gamma}\sum_{j=1}^{l}2^{-j(\sigma_2-\sigma_1)\gamma-nj}\sum_{i\in I}\int_{B_{r}}{\mbox{\Large$\chi$}}_{B_{2^{j+3}|h|^\beta}(z_i)}(x)\left[\int_{B_{r}}\frac{|u(x)-u(y)|^\gamma}{|x-y|^{n+\alpha \gamma}}\,dy\right]\,dx\\
    &\leq c|h|^{\gamma(\sigma_1(1-\beta)+\alpha\beta)}\sum_{j=1}^{l}2^{-j(\sigma_2-\sigma_1)\gamma}\int_{B_{r}}\int_{B_{r}}\frac{|u(x)-u(y)|^\gamma}{|x-y|^{n+\alpha \gamma}}\,dy\,dx.
\end{align*}
Plugging this into \eqref{ineq7.czb} and then using the fact that $\sigma_2>\sigma_1$ (recall \eqref{defn.sigma2}), we arrive at \eqref{boot.czb}, where we have also used the fact that \eqref{set.beta} implies
\begin{align*}
    2s\beta\gamma+n\beta(1-\gamma/q) \ge \sigma_1 \gamma.
\end{align*}
Next, we consider the two sequences $\{\alpha_i\}$ and $\{\overline{\alpha}_i\}$ defined by
\begin{align*}
    \alpha_i\coloneqq \sigma_1(1-\beta)\sum_{k=0}^{i-1}\beta^k - \frac{(\sigma_1-\sigma)(1-\beta)}{4}\sum_{k=0}^{i-1}\beta^k\quad\text{and}\quad \overline{\alpha}_{i}\coloneqq \alpha_i+\frac{(\sigma_1 - \sigma)(1-\beta)}{4}
\end{align*}
for any $i\geq1$, in order to see that 
\begin{align*}
    \lim_{i\to\infty}\alpha_i = \frac{3\sigma_1 + \sigma}{4} \quad\text{and}\quad \overline{\alpha}_{i+1} = \sigma_{1}(1-\beta)+\alpha_i \beta.
\end{align*}
Therefore, there is a constant $i_0=i_0(n,s,q,\gamma,\sigma)$ such that 
\begin{align*}
    \alpha_{i_0}>(\sigma+\sigma_1)/2.
\end{align*}
We now choose a sequence 
\begin{align*}
    r_i=\frac{3}{2}-\frac{i}{2i_0}\quad \text{for any }i=0,\ldots,i_0.
\end{align*}
Then \eqref{boot.czb} with the choices $\rho=(r_i+r_{i-1})/2$ and $r=r_{i-1}$, along with the fact that $\alpha_i \le \sigma_1$, implies
\begin{align*}
    \sup_{h}\int_{B_{(r_i+r_{i-1})/2}}\frac{|\delta_hu|^\gamma}{|h|^{\overline{\alpha}_{i}\gamma}}\,dx \leq c\left[ [u]^\gamma_{W^{\alpha_{i-1},\gamma}(B_{r_{i-1}})} + \widetilde{E}(u;B_{r_{i-1}})^\gamma + \left(\int_{B_{r_{i-1}}\cap \Omega}|f|^{q}\,dx\right)^{\gamma/q} \right]
\end{align*}
for any  $i\geq1$, where the supremum is taken over all $h \in \bbR^n$ satisfying $0<|h|<2^{-10}(2i_0)^{-1/\beta}$. When $i=1$, by taking $\alpha_{0}=0$, this is understood as
\begin{align*}
    \sup_{h}\int_{B_{(r_0+r_1)/2}}\frac{|\delta_hu|^\gamma}{|h|^{\overline{\alpha}_1\gamma}}\,dx\leq c\left[ \|u\|^{\gamma}_{L^\gamma(B_{r_0})} + \widetilde{E}(u;B_{r_0})^\gamma + \left(\int_{B_{r_0}\cap \Omega}|f|^{q}\,dx\right)^{\gamma/q} \right]
\end{align*}
and follows from \eqref{ineq7.czb} via elementary manipulations. 
Since $\overline{\alpha}_{i} > \alpha_{i}$, \cite[Lemma 2.3]{DieKimLeeNow24} implies
\begin{equation}\label{iter.czb}
    [u]^{\gamma}_{W^{\alpha_i , \gamma}(B_{r_i})} \le c\left[ \|u\|^{\gamma}_{L^{\gamma}(B_{r_{i-1}})} +  [u]^{\gamma}_{W^{\alpha_{i-1},\gamma}(B_{r_{i-1}})} + \widetilde{E}(u;B_{r_{i-1}})^\gamma + \left(\int_{B_{r_{i-1}}\cap \Omega}|f|^{q}\,dx\right)^{\gamma/q} \right]
\end{equation}
whenever $i \ge 1$. Iterating \eqref{iter.czb} $i_0$-times, we deduce
\begin{equation}\label{iter2.czb}
[u]^{\gamma}_{W^{\sigma , \gamma}(B_{1})} \le c\left[\|u\|^{\gamma}_{L^{\gamma}(B_{3/2})} + \widetilde{E}(u;B_{3r/2}) + \left(\int_{B_{3/2}\cap\Omega}|f|^{q}\,dx\right)^{\gamma/q}\right]. 
\end{equation}
To estimate the first term in the right-hand side, we consider the weak solution $v$ to
\begin{equation*}
\left\{
\begin{alignedat}{3}
(-\Delta)^s{v}&= 0&&\qquad \mbox{in  $ B_2\cap \Omega$}, \\
{v}&=u&&\qquad  \mbox{in $\bbR^n\setminus (\Omega\cap B_2) $}
\end{alignedat} \right.
\end{equation*}
to see that 
\begin{equation*}
\left\{
\begin{alignedat}{3}
(-\Delta)^s{(u-v)}&= f&&\qquad \mbox{in  $ B_2\cap \Omega$}, \\
{u-v}&=0&&\qquad  \mbox{in $\bbR^n\setminus (\Omega\cap B_2) $}.
\end{alignedat} \right.
\end{equation*}
In light of Lemma \ref{lem.comp.high} and Lemma \ref{lem.bdd}, we derive
\begin{equation*}
\begin{aligned}
    \|u\|_{L^\gamma(B_{3/2})} & \leq \|u-v\|_{L^\gamma(B_{3/2})}+\|v\|_{L^\gamma(B_{3/2})} \\
    &\leq c\left(\|f\|_{L^q(B_{2}\cap \Omega)}+\widetilde{E}(v;B_{2})\right) \leq c\left(\|f\|_{L^q(B_{2}\cap \Omega)}+\widetilde{E}(u;B_{2})\right)
\end{aligned}
\end{equation*}
for some constant $c=c(n,s,q)$.
Connecting this estimate to \eqref{iter2.czb}, we obtain 
\begin{align*}
     \|u\|^\gamma_{W^{\sigma,\gamma}(B_{1})} \leq c\widetilde{E}(u;B_{2})^\gamma+c\left(\int_{B_{2}\cap \Omega}|f|^{q}\,dx\right)^{\gamma/q},
\end{align*}
which completes the proof. 
\end{proof}

\subsection{Sharp gradient estimates} In this subsection, we prove sharp boundary gradient estimates for fractional $p$-Laplace equations.

When $s>(p-1)/p$, Lipschitz estimates for the $s$-fractional $p$-Laplace equation are established in the very recent paper \cite{BisTop24}. We slightly modify those results in a form suitable for our analysis.
\begin{lemma}
Let $u\in W^{s,p}(B_{2R}(z_0))\cap L^{p-1}_{sp}(\bbR^n)$ be a weak solution to 
\begin{equation}\label{eq.lip}
(-\Delta_p)^su=f\quad\text{in }B_{2R}(z_0),
\end{equation}
where $s>(p-1)/p$ and $f \in L^{\infty}(B_{2R}(z_0))$. 
Then we have 
    \begin{align}\label{ineq0.lip}
        \|\nabla u\|_{L^\infty(B_R(z_0))}\leq \frac{c}{R}\left(\widetilde{E}(u;B_{2R}(z_0))+\|R^{sp}f\|^{1/(p-1)}_{L^\infty(B_{2R}(z_0))}\right),
    \end{align}
for some $c=c(n,s,p)$.
\end{lemma}
\begin{proof}
    By Lemma \ref{lem.bdd}, we have 
    \begin{align}\label{ineq1.lip}
        \|u\|_{L^\infty(B_{3R/2}(z_0))}\leq c\left(\widetilde{E}(u;B_{2R}(z_0))+\|R^{sp}f\|^{1/(p-1)}_{L^\infty(B_{2R}(z_0))}\right)
    \end{align}
    for some constant $c=c(n,s,p)$. Note from \cite{KorKuuLin19} that a weak solution $u$ to \eqref{eq.lip} is also a viscosity solution to \eqref{eq.lip}. Moreover, $u_R(x)\coloneqq u(R x+z_0)/R$ is a viscosity solution to 
    \begin{equation*}
        (-\Delta_p)^s u_R=f_R\quad\text{in }B_2,
    \end{equation*}
    where $f_R(x)\coloneqq R^{sp-(p-1)}f(Rx+z_0)$.
    Thus, using \cite[Theorem 2.1]{BisTop24} and scaling back, we get
    \begin{align}\label{ineq2.lip}
        \|\nabla u\|_{L^\infty(B_R(z_0))}\leq \frac{c}{R}\left(\|u\|_{L^\infty(B_{3R/2}(z_0))}+\mathrm{Tail}(u;B_{3R/2}(z_0))+\|R^{sp}f\|^{1/(p-1)}_{L^\infty(B_{2R}(z_0))}\right),
    \end{align} 
    where $c=c(n,s,p)$. Combining \eqref{ineq1.lip} and \eqref{ineq2.lip} yields \eqref{ineq0.lip}.
\end{proof}

Using this estimate, we now prove Theorem \ref{thm.lip}.
\begin{proof}[Proof of Theorem \ref{thm.lip}.]
    The proof is very similar to that of Theorem  \ref{thm.hig} . Let us fix $\gamma<1/(1-s)$ and $h\in B_{1/8}\setminus \{0\}$. We investigate the following terms:
    \begin{align*}
    \int_{B_{1/8}}\frac{|u(x+h)-u(x)|^{\gamma}}{|h|^{\gamma}}\,dx = \sum_{i=1}^{5}\int_{\mathcal{A}_i}\frac{|u(x+h)-u(x)|^\gamma}{|h|^\gamma}\,dx \eqqcolon \sum_{i=1}^5I_i,
\end{align*}
    where the sets $\mathcal{A}_i$ are defined in the proof of Theorem \ref{thm.hig}.
    First, Lemma \ref{lem.basic1} with $\sigma=1$ implies that if $\delta=\delta(n,s,p,\gamma)>0$ is sufficiently small, then
    \begin{align}\label{ineqi1.gra}
        I_1\leq \int_{ \mathcal{A}_1}\frac{|u(x)|^\gamma}{d^\gamma(x)}\,dx+ \int_{ \mathcal{A}_1}\frac{|u(x+h)|^\gamma}{d^{\gamma}(x+h)}\,dx
        &\leq c\left(\widetilde{E}(u;B_1)+\|f\|^{1/(p-1)}_{L^\infty(B_1\cap \Omega)}\right)
    \end{align}
    for some constant $c=c(n,s,p,\gamma)$, as $1<s+1/\gamma$.
    We now estimate the term $I_2$ as
    \begin{align*}
        I_2&\leq c\int_{\mathcal{A}_2}\left(\|\nabla u\|^\gamma_{L^\infty(B_{d(x)/4}(x))}+\|\nabla u\|^\gamma_{L^\infty(B_{d(x+h)/4}(x+h))}\right)\,dz
    \end{align*}
    for some constant $c=c(n,s,p,\gamma)$. By \eqref{ineq0.lip}, we have 
    \begin{align}\label{ineq2.reg}
        \|\nabla u\|_{L^\infty(B_{d(x)/4}(x))}\leq \frac{c}{d(x)}\left(\widetilde{E}(u;B_{d(x)/2}(x))+\|d(x)^{{sp}}f\|^{1/(p-1)}_{L^\infty(B_{d(x)/2}(x))}\right)
    \end{align}
    for some constant $c=c(n,s,p)$. We now choose 
    \begin{equation*}
        \gamma_0\coloneqq(\gamma+1/(1-s))/2\quad\text{and}\quad s_0\coloneq 1-1/\gamma_0<s
    \end{equation*}
    to see that $(1-s_0)\gamma<1$. As in the estimate of $J$ given in the proof of Theorem \ref{thm.hig}, we have 
    \begin{align*}
       \frac{1}{d^{s_0}(x)}\widetilde{E}(u;B_{d(x)/2}(x))\leq c\left(\widetilde{E}(u;B_1)+\|f\|^{1/(p-1)}_{L^\infty(\Omega\cap B_1)}\right),
    \end{align*}
    where $c=c(n,s,p)$. Plugging this into \eqref{ineq2.reg} yields
    \begin{align*}
        \|\nabla u\|_{L^\infty(B_{d(x)/4}(x))}\leq \frac{c}{d^{1-s_0}(x)}\left(\widetilde{E}(u;B_1)+\|f\|^{1/(p-1)}_{L^\infty(\Omega\cap B_1)}\right)
    \end{align*}
    where $c=c(n,s,p)$. Similarly, we have
    \begin{align*}
        \|\nabla u\|_{L^\infty(B_{d(x+h)/4}(x+h))}\leq \frac{c}{d^{1-s_0}(x+h)}\left(\widetilde{E}(u;B_1)+\|f\|^{1/(p-1)}_{L^\infty(\Omega\cap B_1)}\right).
    \end{align*}
    Thus, assuming $d(x) \le d(x+h)$ without loss of generality, we get
    \begin{align*}
        I_2\leq \int_{\mathcal{A}_2}\frac{c}{d^{(1-s_0)\gamma}(x)}\left(\widetilde{E}(u;B_1)+\|f\|^{1/(p-1)}_{L^\infty(\Omega\cap B_1)}\right)^\gamma\,dx
    \end{align*}
    for some $c=c(n,s,p,\gamma)$. 
    Using \eqref{ineq1.basic1} together with the fact that $(1-s_0)\gamma<1$, we further estimate
    \begin{align*}
        I_2\leq \int_{\Omega\cap B_{1/4}}\frac{c}{d^{(1-s_0)\gamma}(x)}\left(\widetilde{E}(u;B_1)+\|f\|_{L^\infty(\Omega\cap B_1)}\right)^\gamma\,dx\leq c\left(\widetilde{E}(u;B_1)+\|f\|^{1/(p-1)}_{L^\infty(\Omega\cap B_1)}\right)^\gamma
    \end{align*}
    for some constant $c=c(n,s,p,\gamma)$, whenever $\delta=\delta(n,s,p,\gamma)>0$ is sufficiently small.
   By considering \eqref{ineq2.dsregmove}, \eqref{ineqi5.dsregmove} and \eqref{ineqi1.gra}, we deduce 
   \begin{align*}
       I_3+I_4+I_5\leq c\left(\widetilde{E}(u;B_1)+\|f\|^{1/(p-1)}_{L^\infty(\Omega\cap B_1)}\right)^\gamma
   \end{align*}
   for some constant $c=c(n,s,p,\gamma)$, whenever $\delta=\delta(n,s,p,\gamma)>0$ is sufficiently small.
    Combining the estimates found for the $I_i$, we arrive at
    \begin{align*}
   \sup_{h\in B_{1/8}\setminus\{0\}}\int_{B_{1/8}}\frac{|u(x+h)-u(x)|^{\gamma}}{|h|^{\gamma}}\,dx \leq    c\left(\widetilde{E}(u;B_1)+\|f\|^{1/(p-1)}_{L^\infty(\Omega\cap B_1)}\right)^\gamma.
    \end{align*}
    Then the desired estimate follows via a standard difference quotient characterization of Sobolev spaces \cite[Lemma 8.2]{Giu03}, whenever $\delta=\delta(n,s,p,\gamma)>0$ is sufficiently small. The proof is complete. 
\end{proof}

\printbibliography

\end{document}